\documentclass[twoside,12pt,leqno]{amsproc}
\usepackage[pagebackref]{hyperref}
\usepackage{amsmath,amssymb,amsthm,amsfonts,amsrefs,a4wide,verbatim,mathtools}
\hypersetup{citecolor=blue, linkcolor=blue, colorlinks=true}
\usepackage{enumerate,datetime,booktabs}
\usepackage{color,graphicx,wrapfig}

\usepackage{amsrefs}
\newtheorem{mainthm}{Theorem}

\newtheorem{thm}{Theorem}[section]
\newtheorem{cor}[thm]{Corollary}
\newtheorem{lemma}[thm]{Lemma}
\newtheorem{prop}[thm]{Proposition}

\newtheorem{defna}[thm]{Definition}
\newenvironment{defn}{\begin{defna} \rm}{\end{defna}}
\newtheorem{remarka}[thm]{Remark}
\newenvironment{remark}{\begin{remarka} \rm}{\end{remarka}}
\newtheorem{nota}[thm]{Notation}
\newenvironment{notation}{\begin{nota} \rm}{\end{nota}}

\def\gu#1#2{\mathord{\mathrm{GU}_{#1}(#2)}}
\def\su#1#2{\mathord{\mathrm{SU}_{#1}(#2)}}
\def\gl#1#2{\mathord{\mathrm{GL}_{#1}(#2)}}
\def\slin#1#2{\mathord{\mathrm{SL}_{#1}(#2)}}
\def\gaml#1#2{\mathord{\mathrm{\Gamma L}_{#1}(#2)}}
\def\gamu#1#2{\mathord{\mathrm{\Gamma U}_{#1}(#2)}}
\newcommand{\F}{\mathbb{F}}
\newcommand{\Z}{\mathbb{Z}}
\newcommand{\eps}{\varepsilon}
\newcommand{\cD}{\mathcal{D}}
\newcommand{\PP}{\mathbb{P}}
\newcommand{\Irr}{\textup{Irr}}
\newcommand{\UU}{{\sf U}}

\newcommand{\CC}{\mathcal{C}}
\newcommand{\stab}{\mathrm{Stab}}
\def\inv{\mathrm{inv}}
\renewcommand{\r}{r}  %we have sometimes used \mathrm{r} via this Macro
\renewcommand{\geq}{\geqslant} % \geq -> \geqslant upon submission
\renewcommand{\leq}{\leqslant} % \leq -> \leqslant upon submission

\makeatletter        % begin hack so date appears with amsproc
\def\@adminfootnotes{%
  \let\@makefnmark\relax  \let\@thefnmark\relax
  \ifx\@empty\@date\else \@footnotetext{\@setdate}\fi%%   <-- added
  \ifx\@empty\@subjclass\else \@footnotetext{\@setsubjclass}\fi
  \ifx\@empty\@keywords\else \@footnotetext{\@setkeywords}\fi
  \ifx\@empty\thankses\else \@footnotetext{%
    \def\par{\let\par\@par}\@setthanks}%
  \fi}\makeatother   % end hack so date appears

\title[Involution centralisers in unitary
groups]{Involution centralisers in finite\\ unitary groups of odd characteristic}
\author{S.P. Glasby, Cheryl E. Praeger and Colva
  M. Roney-Dougal}
\address{S.P. Glasby \& Cheryl E. Praeger: 
Department of Mathematics and Statistics,
UWA, Perth, WA 6009, Australia}
\address{Colva M. Roney-Dougal: Mathematical Institute, Univ.~St Andrews, KY16 9SS, UK}

\begin{document}

\date{\today}

\begin{abstract}
We analyse the complexity of constructing involution centralisers 
in unitary groups over fields of odd order. 
In particular, we 
 prove logarithmic bounds on the number of random elements 
  required to generate a subgroup of the centraliser of a strong
  involution that contains the last 
  term of its derived series. We use this to strengthen previous bounds on
  the complexity of recognition algorithms
for unitary groups in odd
  characteristic. Our approach generalises and extends two previous
  papers by the second author and collaborators on strong involutions
  and regular semisimple elements of linear groups. 
\end{abstract}

\maketitle

\section{Introduction}

Parker and Wilson~\cite{ParkerWilson}  showed in 2010 that
involution-centraliser methods could be used to solve several
computationally difficult problems, and gave complexity analyses for
these algorithms in simple Lie type groups in odd characteristic.
Central to these approaches are conjugate pairs $(t, t^g)$ of involutions.
If $g$ is a uniformly distributed random element of a group $G$, and
$y = tt^g$ has odd order $2k+1$, then $z = gy^k$ is a uniformly
distributed random element of $C = C_G(t)$. This observation is due to
Richard Parker,
see~\cite[Theorem 3.1]{Bray00}. Parker and Wilson in~\cite[Theorem 2]{ParkerWilson}
showed that if $G$ is a simple classical group of dimension $n$,  then the
proportion of elements $g$ of $G$ such that $tt^g$ has odd order is bounded
below by $cn^{-1}$, for some constant $c$, so that with high probability 
$O(n)$ random elements $g$ suffice to construct such a random element 
$z$. Moreover, for infinitely many odd field orders, if $G$ is linear
or unitary then the lower bound $cn^{-1}$ cannot be improved
(see~\cite[p.~897]{ParkerWilson} and~\cite[Theorem 1.2]{BGPW}).  
If $y=tt^g$ has even order $2k$, then $z=y^k$ is an involution in the
centraliser $C$ of~$t$. However, these elements $z$ are \emph{not} uniformly
distributed in $C$; instead $z$ is uniformly distributed only within its 
$C$-conjugacy class. 

In this paper 
we analyse the centralisers $C$ of strong involutions $t$ (see
Definition~\ref{def:invols}) in unitary groups $G$ in odd characteristic. 
%We
%first analyse the fundamental case of conjugate
%involution pairs $(t, t^g)$ whose product $y = tt^g$ is regular semisimple
%on the underlying vector space. \
We show that there exists an absolute constant $D$ such that given a
strong involution $t$, 
a set of  $D \log n$
random elements $g$ suffices to construct a set of involutions that
generates a group containing the last term in the derived series of
$C_G(t)$. A careful analysis of the highest power of 2 dividing $|tt^g|$
is required here. 
Our methods build on the work of Praeger and Seress~\cite{gl},
and of Dixon, Praeger and Seress~\cite{DPS}, but we encounter fundamental new
difficulties: the structure of regular semisimple elements in
$\gu{n}{q}$ that are ``almost irreducible'' (in a sense that we shall
make precise in Definition~\ref{defn:U*closed_irred}) and conjugate to their
inverses 
is very different
from those in $\gl{n}{q}$. In future work, we plan to address the
 symplectic and orthogonal groups. For these families of groups completely 
different arguments will be
required: for example,  one may readily compute that 
in $\mathrm{Sp}_{4}(3)$ and $\mathrm{Sp}_{6}(3)$ there are no
regular semisimple elements that are inverted by involutions.

\begin{defn}\label{def:invols}
For an involution $t \in \gl{n}{q^2}$, we write $E_{+}(t)$ and
$E_{-}(t)$ to denote its eigenspaces for eigenvalues $+1$ and $-1$.
Such a $t$ is \emph{strong} if $n/3 \leq \dim(E_+(t)) \leq
2n/3$. 
For an element $x$ of a group $G$, let $\inv(x)$ denote $x^{|x|/2}$ when $|x|$
is even, and $1_G$ otherwise. 
\end{defn}

\begin{defn}\label{def:nearunif}
A random variable $x$ on a finite group $G$ is \emph{nearly uniformly distributed} if for all $g \in
G$ the probability $\PP(x = g)$ that $x$ takes the value $g \in G$ satisfies
$$
\frac{1}{2|G|} < \PP(x = g) < \frac{3}{2|G|}.
$$ 
\end{defn}

Our first main technical theorem is as follows.

\begin{mainthm}\label{main} 
There exist positive constants $\kappa, n_0 \in \mathbb{R}$ such that the
following is true. Suppose that $n\geq n_{0}$, that $t$ is a strong involution in
$\gu{n}{q}$ with $q$ odd, and that $g$ is a nearly uniformly distributed random element of $\gu{n}{q}$.
Let $z(g):=\inv(tt^{g})$, and let $z(g)_{\varepsilon}$ be the
restriction of $z(g)$ to the eigenspace $E_{\varepsilon}(t)$ (where
$\varepsilon
\in \{+, -\}$). Then

\begin{enumerate}[{\rm (i)}]
\item $z(g)_{+}$ is a strong involution with probability at least $\kappa/\log
n$; and
\item  $z(g)_{-}$ is a strong involution with probability at least $\kappa/\log
n$.
\end{enumerate}
\end{mainthm}

Our proof shows that the values $n_0=250$ and $\kappa=0.0001$ suffice.
The comparable values in \cite{DPS}  for the case of special linear groups with
nearly uniform random elements are
$n_0=700, \kappa=0.0001$, and echoing the view expressed there, `we 
believe that these constants are far from best possible'.

From Theorem~\ref{main}, and \cite[Theorem 1.1]{DPS},
 we are able to deduce  the following result (see
\S\ref{sec:final_proofs}).

\begin{mainthm}\label{thm:gen_cent}
There exist  constants $\lambda, n_1 \in \mathbb{R}$ such that the following is
true. 
Let $n \geq n_1$, let $G = \gl{n}{q}$ or $\gu{n}{q}$ 
with $q$ odd, and let $t \in G$ be a strong
involution. For $\varepsilon \in \{+, -\}$,
let  $S_\varepsilon = \slin{}{E_\varepsilon(t)}$ if $G = \gl{n}{q}$,
or $\su{}{E_\varepsilon(t)}$ if $G  = \gu{n}{q}$. 
Let $A$ be a sequence of at least $\lambda \log n$
random elements of $G$, chosen independently 
and nearly uniformly, and let $H = \langle \inv(tt^g) \mid g \in A
\rangle$. Then 
$$\PP(H \mbox{ contains } S_+ \times S_-) > 0.9(1 -
q^{-n/3} - q^{-2n/3}).$$
\end{mainthm}

One of our motivations for proving the preceding two theorems was an
application to computational group theory. Two key steps in many
algorithms (for example, those in \cites{LOB, LO, ParkerWilson})
are first to construct an involution $t$ in a group
$G$ of Lie type, and then to
construct a subgroup of the centraliser of $t$ that contains
the last term, $C_G(t)^\infty$, in the derived series of $C_G(t)$. For
some of these algorithms, including the constructive recognition
algorithms in \cite{LO}, the involution $t$ is required to be strong. 

\begin{defn}\label{def:r(g, t)}
Let $G$ be a group. For an involution $t$ and an element $g$ of $G$,
 we let $R(g,t)$  be  $\textup{inv}(y)$ when $y:= tt^g$ has even order, and
$gy^{\lfloor |y|/2 \rfloor}$ when $|y|$ is odd. It follows that $R(g, t) \in
C_G(t)$.
\end{defn}

Building on work of L\"ubeck,
Niemeyer and Praeger \cite{LNP}, we can remove the degree
restriction in Theorem~\ref{thm:gen_cent}, and include the step of finding
a strong involution, whilst only
slightly worsening the probability of success (see \S\ref{sec:final_proofs}).

\begin{mainthm}\label{thm:gen_cent2}
    There exists a positive constant $\mu$ such that for all $n\geq3$,
    for all odd~$q$, and for $G=\gl{n}{q}$ or $\gu{n}{q}$,
 the following holds with probability at least
    $0.89(1- q^{-n/3} - q^{-2n/3})$. 
    A sequence $S$ of $\lceil\mu\log n\rceil$
    independent nearly uniformly distributed random elements of $G$ contains an
    element~$x$ such that $t:=\textup{inv}(x)$ is a strong involution,
    and moreover $C_G(t) \geq \langle R(g,t) \mid g\in S \rangle \geq
    C_G(t)^{\infty}$. 
% the elements $R(g,t)$, $g\in S$, lie in $C_G(t)$ and
%    generate a group containing $C_G(t)^{(\infty)}$.
%There exists a constant $\mu \in \mathbb{R}$ such that
%for all $n>2$ and for all odd $q$, the following holds. 
%There is a Las Vegas algorithm which on input a strong involution
%$t \in \gu{n}{q}$, constructs  generators for a subgroup of
%$C_{\gu{n}{q}}(t)$ containing $(C_{\gu{n}{q}}(t))^{\infty}$. It
% requires $\mu \log n$ nearly
%uniformly distributed random elements and succeeds with probability at
% least $0.9(1- q^{-n/3} - q^{-2n/3})$.  
\end{mainthm}

Leedham-Green and O'Brien in \cite{LO} define certain generating
sets for the quasisimple classical groups in odd characteristic,
called \emph{standard generators}, and use a recursive approach, via
repeated involution centralisers,  to
find these standard generators in the given group. Our improved analysis in
Theorem~\ref{thm:gen_cent2} of
the number of random elements required to 
construct an involution centraliser enables us to replace a
factor of $n$ in their complexity analysis with a factor of $\log n$.  Let $\xi$
denote an upper bound on the number of field operations needed to
construct an independent nearly uniformly distributed random element
of $\su{n}{q}$, and let $\chi(q)$ be an upper bound on the number of field
operations equivalent to a call to a discrete logarithm oracle for
$\F_{q}$.
Reasoning in the same way as \cite[\S 1.1]{DPS}, the following
can be deduced from \cite{LO} and 
 Theorem~\ref{thm:gen_cent2}. 

\begin{mainthm}\label{thm:lgob_analysis}
Let $q$ be odd, and let $S
= \su{n}{q}$. 
There is a Las Vegas algorithm that takes as input a set $A$
of generators for $S$ of 
bounded cardinality, and returns standard generators
for $S$ as straight line programmes of length $O(\log^3 n)$ in $A$. 
The algorithm has complexity $O(\log n (\xi + n^3 \log n + n^2 \log n
\log \log n \log q + \chi(q^2)))$,  measured in field operations.
\end{mainthm}

To prove Theorem~\ref{main}, we carry out an extensive analysis
of the products of conjugate involutions in $\gu{n}{q}$. Some of our
results may be of independent interest, so in the remainder of this
section we describe them.  

\begin{defn}\label{def:more_invols}
Denote the characteristic polynomial of a square matrix $y$ by
$c_y(X)$.  Such a matrix $y$ 
is \emph{regular semisimple} if 
$c_y(X)$ is multiplicity-free. 
Let $V = \F_{q^2}^n$,  with $q$ odd,  equipped with a unitary form
having Gram matrix
the identity matrix $I_n$. 
We say that an involution $t \in \gl{n}{q^2}$ is
\emph{perfectly balanced}  if $\dim(E_+(t)) = \lfloor
n/2 \rfloor$. 
  Following \cite{gl}, we
 define $\mathcal{C}(V)$ to be the class of
perfectly balanced involutions in $\gl{n}{q^2}$, 
and we define $\mathcal{C}_\UU(V)$ to be $\mathcal{C}(V) \cap
\gu{n}{q}$. 
\end{defn}

We let 
\begin{equation}\label{defIU}
\mathbf{I}_\UU(V) = \mathbf{I}_{\UU}(n, q) =  \left\{(t, t') \in \mathcal{C}_\UU(V) \times 
\mathcal{C}_\UU(V) \, \mid \,
y:= tt' \mbox{ is regular semisimple}
	  \right\}. 
\end{equation}

\begin{mainthm}\label{thm:iota_bound}
%Let $\iota_\UU(n, q)$ be as above. 
For $q$ odd, let $\iota_\UU(n, q) =
|\mathbf{I}_\UU(V)|/|\mathcal{C}_\UU(V)|^2$ be the probability
that a random element $(t, t') \in \mathcal{C}_\UU(V) \times \mathcal{C}_\UU(V)$
lies in $\mathbf{I}_\UU(V)$.  If $n\ne 3$ then $\iota_{\UU}(n, q) > 0.25$, and 
$\iota_{\UU}(3, q) > 0.142$.
\end{mainthm}

\begin{remark}\label{rem:iota_bound}
We prove that  $\iota_{\UU}(2, q) > 0.25$, that 
$\iota_{\UU}(n, q) >  0.343$ for $n \geq 4$ even, 
%. We  show in 
%Lemma~\ref{lem:iota_even_odd} that we can derive the
%probability $\iota_\UU(n, q)$ for $n$ odd from the even-dimensional
%case, and hence show that 
and  that  $\iota_\UU(n, q) >
0.254$  for $n \geq 5$ odd. 
We shall also prove in Corollary~\ref{cor:iota_lim} that the limits 
as $m \rightarrow
\infty$ of $\iota_{\UU}(2m, q)$ and $\iota_{\UU}(2m+1, q)$ exist, and
determine each limit.
\end{remark}

%Since the probabilities for $n$ odd can be derived from those when $n$
%is even, we shall focus on the even case. 

%\begin{mainthm}\label{thm:r_lims}
%  Let $q$ be odd. As $m \rightarrow \infty$, the limit of $\r_{\UU}(2m, q)$ exists, and
%$\r_{\UU}(2m, q)$ satisfies
%\[
%  \mu \delta- \nu(q) q^{1-m/2} <r_\UU(2m,q)<\mu + \nu(q)  q^{1-m/2}
%\]
%where
%\[\mu=\frac{q^2-1}{q^2+2q}\left(1-\frac{2}{q(q+1)}\right)^{q-1}, \quad 
%  \delta=1-\frac{8}{9q^3}, \quad \mbox{ and } 2 < \nu(q) < 9.
%\]
%The constant in the additive error term satisfies $2 < \nu(q) < 9$. 
%\end{mainthm}

%See Theorem~\ref{T:R2n_3_bounds} for a precise expression for $\nu(q)$. 
%In Corollary~\ref{cor:r_exp} we also give an infinite product
%that evaluates to $\lim_{m \rightarrow \infty} r_{\UU}(2m, q)$. 

The structure of this paper is as follows. In
\S\ref{sec:prelims} we begin our exploration of the conjugacy classes of
$\gu{n}{q}$, and of the characteristic polynomials of elements of
$\gu{n}{q}$. In
\S\ref{sec:invol_reduce} we define a set of ordered pairs of
conjugate involutions $(t, t^g)$ such that $\inv(tt^g)|_{E_+(t)}$ is
guaranteed to be a strong involution. Thus to prove
Theorem~\ref{main} it suffices to show that this set is
sufficiently large.  In \S\S\ref{sec:poly},
\ref{sec:cent} and \ref{sec:invol}
 we
classify the $\UU*$-irreducible  regular semisimple elements of
$\gu{n}{q}$ (that is, such elements that are as close to irreducible as
possible, see Definition~\ref{defn:U*closed_irred}),
determine their centralisers, and 
count the number of involutions inverting
them. In \S\ref{S:bds} we calculate various upper and lower
bounds on the number of monic polynomials that correspond to
irreducible
factors
of the 
characteristic polynomials of these $\UU*$-irreducible regular
semisimple 
elements. In
\S\ref{sec:genfnrU} we define and analyse our key generating
function,  $R_\UU(q,
u)$. In \S\ref{sec:r_u_as_prod} we
factorise $R_{\UU}(q, u)$,  and prove bounds on
the coefficients of certain generating functions that refine the
information in $R_{\UU}(q, u)$. This additional information allows us
to control the powers of $2$ dividing the
orders of the roots of the characteristic polynomial of $tt^g$, and
hence to bound the
dimension of the $(-1)$-eigenspace of $\inv(tt^g)$.
 In \S\ref{sec:thm_main} we prove
Theorem~\ref{main}, and finally in \S\ref{sec:final_proofs}
we prove Theorems~\ref{thm:gen_cent}, \ref{thm:gen_cent2} and \ref{thm:iota_bound}. 

\subsection{Acknowledgements}
The work for this paper began whilst the third author was a Cheryl
E. Praeger Visiting Research Fellow, and the authors are grateful for
the hospitality of the Universities of St Andrews and Western
Australia, and the Hausdorff Research Institute for Mathematics,
Bonn. We are grateful for support from Australian Research Council
Discovery Project grants DP160102323 and  DP190100450. We thank Eamonn
O'Brien for his extremely careful reading of several drafts of this article.

\section{Preliminaries}\label{sec:prelims}

In this section we  study the conjugacy classes and
characteristic polynomials of involutions in $\gu{n}{q}$,  and of 
regular semisimple
elements  of $\gu{n}{q}$ that are products of involutions.  We
shall assume throughout the paper that $q$ is an odd prime power. 
%We 
%also show that $|\Delta_{\UU}(2m, q)| = |\mathbf{I}_{\UU}(2m, q)|$. 

Let $V = \F_{q^2}^n$ be the natural module for $\gu{n}{q}$, and unless
stated otherwise, take the sesquilinear form fixed by $\gu{n}{q}$ to have
the identity matrix $I_n$ as its Gram matrix as in
Definition~\ref{def:more_invols}.
Determining conjugacy in $\gu{n}{q}$ is straightforward:

\begin{thm}{\rm (Wall,~\cite[p.\,34]{Wall})}\label{thm:gu_conj}
Let $g, h \in \gu{n}{q}$. If $g$ and $h$ are conjugate in
$\gl{n}{q^2}$ then they are conjugate in $\gu{n}{q}$.
\end{thm}

\begin{defn}\label{def:type}
An involution $t \in \gl{n}{q^2}$ has \emph{type} $(a, b)$
if $\dim(E_+(t)) = a$ and $\dim(E_-(t)) = b$. 
\end{defn}

For $q$ odd, involutions in
$\gl{n}{q^2}$ are conjugate if and only if they have the same
type. The following corollary of Theorem~\ref{thm:gu_conj} is therefore
immediate.

\begin{cor}\label{cor:gu_conj}
Each type $(n_+, n_-)$ of involution in $\gl{n}{q^2}$ forms 
a unique conjugacy class in $\gu{n}{q}$. In particular, $\mathcal{C}_\UU(V)$ is a $\gu{n}{q}$-conjugacy
class.
\end{cor}

We define three involutory operations on polynomials over $\F_{q^2}$.
Let 
\begin{equation}\label{eqn:f}
f(X):= X^n + a_{n-1}X^{n-1} + \cdots + a_0 \in \F_{q^2}[X],
\end{equation} 
and let $\sigma$ be the
involutory automorphism $\sigma: x \mapsto x^q$ of $\F_{q^2}$.
Then
we define the \emph{$\sigma$-conjugate} of $f(X)$ to be
\[
{f}^{\sigma}(X):= X^n + a_{n-1}^q X^{n-1} + \cdots + a_0^q.\]
If $a_0 \neq 0$ then we also define the \emph{$*$-conjugate} of
$f(X)$ to be
\begin{equation}\label{eqn:f*} 
f^*(X)  :=   X^n  +a_1 a_0^{-1} X^{n-1} + a_2 a_0^{-1} X^{n-2} + \cdots
  + a_{n-1}a_0^{-1}X + a_0^{-1} 
\end{equation} and define
\[f^{\sim}(X) :=  f^{\sigma *}(X) = f^{* \sigma}(X).  
\]

It is clear that the operations $*$ and $\sim$ are involutions on the set of
monic polynomials of degree $n$ over $\F_{q^2}$ with nonzero constant term.

By abuse of notation, we also write $\sigma$ for the automorphism of
$\gl{n}{q^2}$ induced by replacing each matrix entry by its image
under $\sigma$. We write $A^T$ for the transpose of a matrix $A$, and write
$h^\sim=h^{-\sigma T}$ for $h \in \gl{n}{q^2}$.

Our choice of unitary form means that $h\in\gl{n}{q^2}$ 
lies in $\gu{n}{q}$ if and only if $h\,h^{\sigma T}= I$. 
In other words, we have the following.

\begin{lemma}\label{lem:gu_member}
A conjugate of $h \in \gl{n}{q^2}$ lies in $\gu{n}{q}$ if and
only if $h$ is conjugate to $h^{\sim}$.
\end{lemma}

Notice that the
characteristic polynomial $g(X) = c_h(X)$ of $h \in \gl{n}{q^2}$ satisfies $c_{h^{-1}}(X) =
g^*(X)$, and  $c_{h^\sim}(X)=g^\sim(X)$.

\begin{cor}\label{cor:gu_cp}
Let $y \in \gl{n}{q^2}$ be regular semisimple, and let $g(X) =
c_y(X)$. 
\begin{enumerate}[{\rm (i)}]
\item A conjugate of $y$
lies in $\gu{n}{q}$ if and only if $g(X) =
g^\sim(X)$.
\item If $y \in \gu{n}{q}$ then $y$ is conjugate in $\gu{n}{q}$ to
  $y^{-1}$ if and only if 
 $g(X) = g^*(X)$. 
\end{enumerate}
\end{cor}

\begin{proof}
In both parts, one direction is clear, and the other follows from the
fact that since $y$ is regular semisimple, $g(X)$ is equal to the
minimal polynomial $m_y(X)$, and $g(X)$ is multiplicity-free. The
fact that $y$ and $y^{-1}$ are conjugate in $\gu{n}{q}$ (rather than
just in $\gl{n}{q^2}$) follows from
Theorem~\ref{thm:gu_conj}. 
\end{proof}

Recall that $\CC(V)$ is the class of perfectly balanced
involutions (Definition~\ref{def:more_invols}). The importance of $\CC(V)$ 
for studying regular semisimple elements is illustrated by the
following:

\begin{lemma}[{\cite[Lemma 3.1]{gl}}]\label{L:ClassC_important}
Let $t, y \in \gl{n}{q^2}$, such that $y$ is regular
semisimple, and $t$ is an involution inverting $y$. Let $t' = ty$.
\begin{enumerate}[{\rm (i)}]
\item  If $\mathrm{gcd}(c_y(X), X^2-1) = 1$, then
 all involutions
  inverting $y$ lie in $\mathcal{C}(V)$ and $n$ is even.
\item If $t, t' \in \gl{n}{q^2}$ are conjugate in $\gl{n}{q^2}$
then either $t, t' \in \mathcal{C}(V)$ or $-t, -t' \in
\mathcal{C}(V)$.
\item If $t, t' \in \mathcal{C}(V)$ and $n$ is even, then
  $\mathrm{gcd}(c_y(X), X^2-1) = 1$. 
\end{enumerate}
\end{lemma}

\begin{proof}
Part (i) follows from \cite[Lemma 3.1(a) and Table 1]{gl},
Part (ii) is \cite[Lemma 3.1(c)]{gl}, and Part (iii) follows from
\cite[Lemma 3.1(b)(i)]{gl}. 
\end{proof}

We shall need the following three properties of polynomials. 

\begin{defn}\label{defn:U*closed_irred}
Let $g(X) \in \F_{q^2}[X]$. 
We say that $g(X)$ is \emph{separable} if it has no
repeated roots in the algebraic closure $\overline{\F_{q^2}}$.
The polynomial $g(X)$ 
is \emph{$\UU*$-closed} if $g(X) = g^\sim(X)
= g^*(X)$, and is  \emph{$\UU*$-irreducible} if
it is $\UU*$-closed and no proper nontrivial divisor 
of $g(X)$ is $\UU*$-closed. 
\end{defn}

\begin{defn}\label{def:pi(n, q)}
Define $\Pi_\UU(n, q)$ to be the set of separable, degree $n$, monic,
$\UU*$-closed polynomials over $\F_{q^2}$ with no
roots $0$, $1$, $-1$.
\end{defn}

For $n = 2m$, we
 define the following set, recalling that $\mathcal{C}_\UU(V)=
\mathcal{C}(V)\cap\gu{n}{q}$:
\begin{equation}\label{defRU}
\Delta_\UU(V) = \Delta_{\UU}(2m, q):= \left\{(t,y) \, \mid \,
\begin{array}{l}
t\in\mathcal{C}_\UU(V), y\in\gu{}{V}, y^t=y^{-1}, y\ \mbox{regular}\\
\mbox{semisimple, and $c_y(X)$ coprime to $X^2-1$}
	  \end{array}\right\}. 
\end{equation}
This set is analogous to the set $\mathbf{RI}(V)$ 
defined in \cite[Equation (3)]{gl}.
We now show
the link between $\Delta_\UU(V)$ and  and $\Pi_{\UU}(n, q)$. 
Recall the definition of $\mathbf{I}_\UU(n, q)$ from \eqref{defIU}.

\begin{lemma}\label{lem:delta_right}
With respect to a fixed unitary form, 
$\Delta_{\UU}(n, q)$ is equal to 
the set 
\[\left\{(t,y) \mid  
t\in \gu{n}{q}, \ y\in \su{n}{q}, \ 
t^{2}=1,
y^t=y^{-1}, 
c_y(X) \in \Pi_\UU(n,q) \right\}.\] 
When $n$ is even, $|\Delta_{\UU}(n, q)| =
|\mathbf{I}_\UU(n, q)|$. 
\end{lemma}

\begin{proof}
Let $S$ denote the displayed set. We show first that $S \subseteq \Delta_{\UU}(n,
q)$. Let
$(t, y) \in S$. Then $t$ inverts $y$, and our assumption that 
$c_y(X) \in \Pi_{\UU}(n, q)$ implies that $y$ is regular semisimple
and $\mathrm{gcd}(c_y(X), X^2 - 1) = 1$.
Hence $t \in \mathcal{C}_{\UU}(V)$ by  Lemma~\ref{L:ClassC_important},
and therefore $(t, y) \in
 \Delta_{\UU}(n, q)$. 

For the reverse containment, let 
$(t, y) \in \Delta_{\UU}(n, q)$. Since $y^t = y^{-1}$,  with $y$
regular semisimple and $c_y(X)$ coprime to $X^2-1$, all involutions
in $\gu{n}{q}$ inverting $y$ are in $\mathcal{C}_{\UU}(V)$ by
Lemma~\ref{L:ClassC_important}. 
Let $t' = ty$. Then $t'$
also inverts $y$, so
$t' \in \mathcal{C}_{\UU}(V)$  by Lemma~\ref{L:ClassC_important}. 
Hence, by Theorem~\ref{thm:gu_conj} the involutions $t$ and
$t'$ are $\gu{n}{q}$-conjugate. Then $y = tt'$ is a product of two
conjugate involutions in $\gu{n}{q}$, and so $y \in \su{n}{q}$. Therefore $(t, y) \in
S$. 

For the final claim, consider the map $\theta: (t, t') \mapsto (t,
tt') = (t, y)$ from
$\mathbf{I}_\UU(V)$ to $\mathcal{C}_\UU(V) \times \gu{}{V}$. It is clear that
$\theta$ is injective, and 
$c_{y}(X)$ is coprime to $X^2-1$ by Lemma~\ref{L:ClassC_important}(iii), so the image of $\theta$ is a
subset of $\Delta_\UU(V)$. Hence $|\mathbf{I}_{\UU}(V)| \leq
|\Delta_{\UU}(V)|$. It follows from \cite[Lemma 4.1(a)]{gl} that
$\Delta_{\UU}(V) \subseteq \mathrm{Im}(\theta)$, so these two sets have
equal sizes. 
\end{proof}

%\begin{defn}%\label{defdel_new}
%\begin{enumerate}
%\item[(a)] 
%\item[(b)] Assume that $n/2\leq s < n$ and set $h:= 2s-n$ so $0 \leq h < n$.
%Let $V$ be the natural formed space for $\gu{n}{q}$, and let
%$\Omega_\UU$ be the set of subspace pairs $(V_1, V_2)$ from
%with  $\dim(V_1)=h, \dim(V_2)=n-h$, each $V_i$ non-degenerate, and $V=V_1\perp V_2$.
%\end{enumerate}
%\end{defn}

\section{Pairs of involutions yielding strong
  involutions}\label{sec:invol_reduce} 

In this section, we characterise a certain set of ordered pairs of involutions
$(t, t^g)$ 
from $\gu{n}{q}$ whose product $y = tt^g$ is such that
$\inv(y)|_{E_+(t)}$ is strong.
Recall that $V = \F_{q^2}^n$ is the natural module for $\gu{n}{q}$.
%Recall from Definition~\ref{def:more_invols} that we let $\gu{n}{q}$
%act naturally on $V = \F_{q^2}^n$.

First we make a simple observation about subspaces of $E_+(t)$.

\begin{lemma}\label{lem:tinv}
Let 
$t\in\gu{n}{q}$ be an involution,  and let $U$ be a subspace of
$E_+(t)$. 
Then $U^\perp$ is $t$-invariant, and further if $U$ is non-degenerate then 
$U^\perp$ is also   non-degenerate and $U\cap U^\perp = 0$. 
\end{lemma}

\begin{proof}
Let $u\in U$ and $w\in U^\perp$. Then evaluating the form on $u$ and $w^t$ gives $(u,w^t) = 
(u^t, w)$ since the form is $t$-invariant and $t^2=1$. This is equal to zero since $u^t\in U$. Thus $(U^\perp)^t\subseteq U^\perp$ and we conclude that $(U^\perp)^t= U^\perp$. Finally, if $U$ is non-degenerate then $U\cap  U^\perp=0$ and thus also $U^\perp$ is non-degenerate.
\end{proof}

Recall the definition of
 \emph{type} (Definition~\ref{def:type}), and that types naturally
parametrise the conjugacy classes of involutions in $\gu{n}{q}$ by
Corollary~\ref{cor:gu_conj}. 

\begin{defn}\label{def:alpha_beta}
Given $0\leq\alpha<\beta\leq1$, 
an involution $t \in \gl{n}{q^2}$ of type $(n_{+}, n_{-})$ is 
\emph{$(\alpha,\beta)$-balanced} if $\alpha\leq n_{+}/n\leq\beta$.
\end{defn}

We shall now define a key set of ordered pairs of conjugate
involutions. 

\begin{defn}\label{def:L}
Let $K_{\UU,s}$ be the $\gu{n}{q}$-conjugacy class of
  involutions  of type $(s,n-s)$. Fix
  $0 \leq \alpha < \beta \leq 1$, and
  let $L_\UU(n,s,q;\alpha,\beta)$ be the set of ordered pairs
$(t,t')\in K_{\UU,s}\times K_{\UU,s}$ such that: 
\begin{enumerate}[{\rm (i)}]
\item $V_1 := E_+(t)\cap E_+(t')$ is a non-degenerate
subspace of $V$, and has
dimension $h = 2s - n$, (so, by Lemma~\ref{lem:tinv}, $V_2 := V_1^{\perp}$ 
is non-degenerate, $\langle t, t'  \rangle$-invariant, and of dimension $n-h=2(n-s)$);
\item $(t|_{V_2}, tt'|_{V_2})\in\Delta_\UU(n-h,q)$ and 
$\inv(tt^\prime|_{V_{2}})$ is $(\alpha,\beta)$-balanced. 
\end{enumerate}
\end{defn}

\begin{lemma}\label{L:V_i and W}
Let $(t,t^{\prime})\in L_\UU(n,s,q;\alpha,\beta)$, and let 
$V_2$ be as in Definition~$\ref{def:L}$. 
 If $W$ is a $\left\langle t,t^{\prime}\right\rangle $-invariant
 subspace of $V_2$, then $\dim W$ is even, and the involutions  $t|_{W}$ and
$t ^{\prime}|_{W}$ are both perfectly balanced.
\end{lemma}

\begin{proof}
The characteristic polynomial of $tt'|_{V_2}$ lies in
$\Pi_{\UU}(2(n-s), q)$ by Lemma~\ref{lem:delta_right}, and in particular it
is coprime to $X^2-1$.  
Thus also, by Definition~\ref{def:pi(n, q)}, the characteristic
polynomial of $tt'|_{W}$ lies in
$\Pi_{\UU}(r, q)$, where $r=\dim(W)$. It then follows from
Lemma~\ref{L:ClassC_important} that $r$ is even and both $t|_{W}$ and $t
^{\prime}|_{W}$ lie in $\mathcal{C}_{\UU}(W)$ (and hence are perfectly balanced).
\end{proof}

 \begin{lemma}\label{L:parta}
Let $s$ satisfy $2n/3
\geq s \geq n/2$, let $h = 2s-n$, let
$ \alpha =\max\left\{  0,1-\frac{2s}{3(n-s)}\right\}$   and let $\beta
=1-\frac{s}{3(n-s)}$. Then $\alpha < \beta$. Choose $(t, t') \in L_{\UU}(n, s, q; \alpha,
\beta)$, and let $V_1$ and $V_2$ be as in Definition~$\ref{def:L}$.
% Let $D = \langle t, t' \rangle$ and 
Let $z =  \inv(tt')$,  $V_{2+}:=V_{2}\cap E_{+}(z)$ and $V_{2-}:=E_{-}(z)$. 
 \begin{enumerate}[{\rm (i)}]
%\item[(i)]  For $\varepsilon=\pm$,  $V_{2\varepsilon}$ is a
%non-degenerate $D$-module of even dimension $2k_\varepsilon$. - NOT CITED

 \item  Each entry in the table below is the dimension of the intersection of
   the subspaces labelled by their row and column of the entry, and $k_+ +
k_- = (n-h)/2 = n-s$.
\[%
\begin{tabular}
[c]{c|ccc}
& $E_{+}(t)$ & $E_{-}(t)$ & $V$\\ \hline
$V_{1}$ & $h$ & $0$ & $h$\\
$V_{2+}$ & $k_+$ & $k_+$ & $2k_+$\\
$V_{2-}$ & $k_-$ & $k_-$ & $2k_-$\\
$V$ & $s$ & $n-s$ & n
\end{tabular}
\ \ \ \ \
\]
%\item[(iii)] The following equalities hold:   
%$V_2 = V_{2+}\perp V_{2-}=  E_+(t_{|V_2})\perp E_-(t)$, 
%$E_+(z) = V_1\perp V_{2+}$, and
%$E_+(t) = V_1\perp E_+(t_{|V_2})$.
\item The involution $z_{|E_+(t)}$ is $(1/3,2/3)$-balanced.
 \end{enumerate}
 \end{lemma}

\begin{proof} This proof has similarities to \cite[p.\,445]{DPS}, but we have
  modified the approach to make it more transparent  and to deal with the unitary form. 
It is clear from the definitions of $\alpha$ and $\beta$  that $0 \leq
\alpha \leq 1/3$ and $1/3 \leq \beta \leq 2/3$, and that if $\alpha =
1/3$ then $\beta > 1/3$, and so $\alpha < \beta$.

\medskip

\noindent (i).  The first and last rows of the table are clear, so we need
only prove the middle two rows.
Since $t$ is conjugate in $\gu{n}{q}$ to a diagonal matrix, and our
standard unitary form is the identity matrix, 
the spaces $E_\pm(t)$ are non-degenerate, 
and $V=E_+(t)\perp E_-(t)$, so $E_-(t) = E_+(t)^\perp$.
For the same reason  $E_-(z) = E_+(z)^\perp$, so $V = E_+(z) \perp
E_-(z)$. 

By definition, the subspace $V_1$ is fixed pointwise by $t$ and $t'$ and hence also
by $z$, 
so that $E_+(z)$ contains
both $V_1$ and $V_{2+}$. Since $V_{2+}\leq V_2=V_1^\perp$ we have
$V_1\perp V_{2+}\leq E_+(z)$.
Since $V=V_1\perp V_2$, an arbitrary vector $v\in E_+(z)$ is of the form $v=v_1+v_2$ for some $v_i\in V_i$ ($i=1, 2$). 
Thus $v = vz = v_1z + v_2z = v_1 + v_2z$ whence $v_2 = v_2z\in V_{2+}$.
This yields   $E_+(z) = V_1\perp V_{2+},$ and hence also  $V_{2-}=E_-(z)
=E_+(z)^\perp \leq V_1^\perp = V_2$. 
%It follows that  $$V_2 = V_{2+}\perp V_{2-}.$$

%An analogous argument to that in the first paragraph shows that  $E_+(t) = V_1\perp E_+(t_{|V_2})$,
%and $E_-(t)\leq V_2$. Thus  $E_-(t_{|V_2})= E_-(t)$ (of dimension $n-s = (n-h)/2$), 
%and $\dim E_+(t_{|V_2}) = s-h =(n-h)/2$. It follows that $V_2=  E_+(t_{|V_2})\perp E_-(t)$,
%and part (iii) is proved.

Let $D = \langle t, t' \rangle$.   By Lemma~\ref{lem:tinv},  $V_2$ is
$D$-invariant, and $D$ centralises $z$, so
$V_{2+}$ and $V_{2-}$ are $D$-invariant subspaces
of $V_{2}$.  It follows from Lemma~\ref{L:V_i and W}
that, for $\varepsilon=\pm$,  $V_{2\varepsilon}$ has even dimension, say $2k_\varepsilon$, and that 
$\dim(E_+(t)\cap V_{2\varepsilon}) = \dim(E_-(t) \cap
V_{2\varepsilon}) = \dim(V_{2\varepsilon})/2 = k_\varepsilon$. 
%Noting that $k_+ + k_- = (n-h)/2$ holds since $V_2 = V_{2+}\perp
%V_{2-}$ 
%This completes the proof of Part (i).

\medskip

\noindent (ii). The involution
 $z|_{V_{2}}$ is $(\alpha,\beta)$-balanced by Definition~\ref{def:L}, so 
\[\alpha \leq \frac{2k_+}{n-h}  = \frac{k_+}{n-s} \leq \beta.\]
Let $z':=z|_{E_+(t)}$, and notice that
\[ E_+(z') = E_+(z)\cap E_+(t) = V_1\perp (V_{2+}\cap E_+(t))\]
which by Part (i) has dimension $h + k_+$. Since $\dim E_+(t) 
=s$, the element $z'$ is  $(1/3,2/3)$-balanced if and only if  
$
1/3\leq (h+k_+)/s\leq2/3$. 

 From $n/2\leq s\leq 2n/3$ and $h=2s-n$, 
we deduce $0\leq h\leq n/3$.
%We claim that 
%$z_{|E_{+}(t)}$ is $(1/3,2/3)$-balanced
%whenever $z_{|V_{2}}$ is $(\alpha,\beta)$-balanced. 
By Part (i),  $k_++k_-=n-s$, so $h + k_+  + k_- = s$, and so
$(h+k_+)/s = 1 - k_-/s$.
From $\alpha\leq k_+/(n-s)\leq \beta$, we now deduce that 
\[
\frac{s}{3(n-s)} 
= 1-\beta\leq 1- \frac{k_+}{n-s} = \frac{k_-}{n-s} \leq1-\alpha\leq\frac{2s}{3(n-s)}.
\]
Hence $s/3\leq k_- \leq 2s/3$,
which in turn implies that
$1/3\leq 1-
k_-/s = (h+k_+)/s \leq 2/3$, as required.
\end{proof}

\section{\texorpdfstring{$\UU*$}--irreducible polynomials}\label{sec:poly} 

Recall the three involutory operations on polynomials that we defined in
\S\ref{sec:prelims}, 
and what it means for a
polynomial to be $\UU*$-irreducible (Definition~\ref{defn:U*closed_irred}).
In this section we classify the $\UU*$-irreducible polynomials, and
determine the \emph{$2$-part-orders} of their roots (that is, the maximal
power of 2 dividing the order of their roots).

In the remainder of the paper, 
we shall sometimes refer to a polynomial $f(X) \in \F_{q^2}[X]$ as
simply $f$, when the meaning is clear.

\begin{lemma}\label{lem:irred}
Let $f(X) \in \F_{q^2}[X]$ be monic, irreducible, and of degree $\deg f =
m$.
\begin{enumerate}[{\rm (i)}]
\item If $f^*(X) = f(X)$
then either $f(X)$ is $X+1$ or $X-1$ or $m$ is even.
\item If $f^\sigma(X) = f(X)$ then $m$ is odd.
\item If $f(X) \neq X \pm 1$ then at least one of $f^\sigma(X),
  f^*(X)$ does not equal $f(X)$.
\item If $f(X) \neq X \pm 1$ and $f(X) = f^\sim(X)$ then $m$ is odd. 
\end{enumerate}
\end{lemma}

\begin{proof}
Parts (i) and (ii) are proved in  \cite[Lemma 1.3.15(c) and Lemma
1.3.11(b)]{genfunc}, respectively. 
Part (iii) follows immediately, so consider
Part (iv). Suppose that $f(X) = f^\sim(X) \neq X \pm 1$.  By Part
(iii), at least one of
$f^\sigma, f^*$ is not equal to $f$. Conversely, $f^\sim =
f^{\sigma*} = f$, so we deduce that
$f^\sigma = f^*\neq f$. 

Let $\zeta \in \F_{q^{2m}}$ be a root of $f$, so that $f$ is the
minimal polynomial
of $\zeta$ over $\F_{q^2}$. Then the set of roots of $f^\sigma$ is $\{\zeta^q,
\zeta^{q^3}, \ldots, \zeta^{q^{2m - 1}}\}$ and the  set of 
roots of $f^*$ is $\{\zeta^{-1}, \zeta^{-q^2}, \ldots,
\zeta^{-q^{2m-2}}\}$. 
Since these sets are equal, $\zeta^{-1} = \zeta^{q^{2i+1}}$ for some
$i$, and  so $\zeta^{q^{2i + 1} +1} = 1$. Hence $\zeta \in \F_{q^{4i+2}}
= \F_{(q^2)^{2i+1}}$,  an odd degree extension of $\F_{q^2}$.
Thus $m=|\F_{q^2}(\zeta):\F_{q^2}|$ is odd.
\end{proof}

\begin{defn}\label{def:omega}
For a monic irreducible polynomial $f(X)$, we let $\omega(f)$ denote the
order of one (and hence all) of its roots. Similarly, 
$\omega(g)$ denotes the order of one (and hence all) of the roots of a
$\UU*$-irreducible polynomial $g(X)$.  We write $n_2$ for the
$2$-part of an integer $n$, and let  $\omega_2(f)$  denote
the 2-part of $\omega(f)$. 
\end{defn}

%Because of the properties enumerated in Lemma~\ref{lem:irred}, 
Let $y \in \gu{n}{q}$ be both regular semisimple and conjugate
to its inverse. We distinguish five possibilities for 
irreducible factors $f(X)$ of the characteristic polynomial
$c_y(X)$. %We write $m = \deg f$. 

\begin{prop}\label{prop:cases}
Let $y \in \gu{n}{q}$ be regular semisimple, and let 
$f(X)$ be an irreducible factor of $c_y(X) \in \F_{q^2}[X]$ of
degree $m$. 
If $y$ is conjugate to $y^{-1}$ in $\gu{n}{q}$, then $c_y(X)$ is
$\UU*$-closed, 
and $f(X)$ satisfies precisely
one of the following:
\begin{enumerate}[{{\bf Type }\bf A.}]
\item $f = f^* \neq f^\sigma$. 
Thus $ f^\sim=f^\sigma \neq f$, $ff^\sim \mid
  c_y(X)$, $m$ is even, and $\omega(f) \mid (q^m + 1)$. 
\item $f = f^{\sigma} \neq f^*$. Thus $f^\sim = f^*
  \neq f$,
 $ff^\sim \mid c_y(X)$, $m$ is odd, and $\omega(f) \mid (q^m - 1)$. 
\item $f \neq f^* = f^\sigma$. Thus $f^\sim = f$, 
  $ff^* \mid c_y(X)$, $m$ is odd, and $\omega(f) \mid (q^m + 1)$. 
\item $|\{f, f^*, f^\sigma, f^\sim\}|
  = 4$, $ff^*f^\sigma f^\sim \mid c_y(X)$, and 
$\omega(f) \mid (q^{2m} - 1)$. 
\item $f(X) = X \pm 1$.
\end{enumerate}
\end{prop}

\begin{proof} Let $g(X) = c_y(X)$. 
Since $y$ is conjugate to $y^{-1}$, we have $g =g^*$, 
and by Corollary~\ref{cor:gu_cp}, $g = g^\sim$. Thus
$g^*=g^{\sim}$, and hence 
$g = g^{\sigma}$, so $g$ is $\UU*$-closed.

Since $g = g^*$, the polynomial $f^*$ is a factor of $g$, and either
$f = f^*$ or $f f^*$ divides~$g$. 
Furthermore, since $g = g^\sim$, the polynomial $f^\sim$ is
a factor of $g$, and either $f = f^\sim$ or
$f(X)f^\sim(X)$ divides $g(X)$.

If $f  = f^* = f^\sigma$ then,  by
Lemma~\ref{lem:irred}, $m = 1$ and $f(X)=X\pm 1$, and we are in Type~E. Assume now that $|\{f, f^*, f^\sim\}| \geq 2$. 
Then it is easy to see that equalities between the polynomials,
$f, f^*, f^\sigma, f^\sim$, and the divisibility 
concerning $c_y(X)$, satisfy  the
conditions of precisely one of Types A to D.

 Let $\zeta$ be a root of $f$. In Type~A,
the degree $m$ of $f$ is even by Lemma~\ref{lem:irred}(i), and since $f = f^*$,
the  roots of $f$ in $\F_{q^{2m}}$
are
\[  \{\zeta, \zeta^{q^2}, \ldots, \zeta^{q^{2m - 2}}\} = \{\zeta^{-1},  
    \zeta^{-q^2}, \ldots, \zeta^{-q^{2m-2}}\}.\]
Thus $\zeta^{-1} = \zeta^{q^{2i}}$ for some $i$ with $0 < i \leq m-1$,
and so $\zeta^{q^{2i} + 1} = 1$, from which we deduce that
$\zeta \in \F_{q^{4i}} \cap \F_{q^{2m}} = \F_{q^{4(i, m/2)}}$. If $i \neq m/2$
  then $4(i, m/2) < 2m$, contradicting the irreducibility of
  $f$. Hence $i=m/2$ and $\zeta^{q^{m}+1} = 1$, as required.

In Type~B, the degree $m$ is odd  by Lemma~\ref{lem:irred}(ii), 
and we deduce from $f = f^\sigma$ that $\zeta^{q^{2i+1} - 1} = 1$ for some $i$, and hence that $2i+1 = m$. 
In Type~C, the degree $m$ is odd by Lemma~\ref{lem:irred}(iv), and we use $f =
f^{*\sigma}$ to reach a similar conclusion. For Type D, we shall
see in \S\ref{sec:cent} a construction of regular semisimple elements  that
is independent of the parity of $m$. 
\end{proof}

\begin{remark}
%There are about $q^{2m}/m$ irreducible
%polynomials of degree~$m$ over $\F_{q^2}$ and we shall see in Lemma~\ref{L:bds}
%that there are about $q^{2m}/(4m)$ polynomials in Type~D. Hence 
We shall eventually see that almost all irreducible
polynomials are of Type D, independent of the parity of~$m$.
\end{remark}

\begin{remark}
Observe that Type~E is equivalent to $f(X) = f^*(X) = f^\sigma(X)\ne X$.
Hence if $\deg f>1$, then the hypotheses for Types A, B, C can be
abbreviated to  $f = f^*$, $f = f^\sigma$, and
$f = f^\sim$, respectively. 
%Indeed, $f(X) = f^\sim(X)$ holds in
%Types~C and E only.
\end{remark}

\begin{defn}\label{defn:U*irred_cases}
If one (and hence all) of the irreducible factors of 
a $\UU*$-irreducible polynomial $g(X) \neq X\pm 1$ are of
Type A, B, C or D, then we say that $g(X)$ has this type.
\end{defn}

\begin{defn}\label{def:nq_r}
Let $N(q^2, r)$ denote the number of monic irreducible polynomials $f(X) \in
\F_{q^2}[X]$ of degree $r$ with
  $\gcd(f(X),X)=1$.
\end{defn}

\begin{lemma}[{\cite[Lemma 2.11]{DPS}}]\label{L: Pn}
Let $\mathcal{P}_{r, q^2}$ be the set of monic irreducible polynomials $f(X)$ of
degree $r$ over $\F_{q^2}$ $(q$ odd$)$ with nonzero roots $($so $\left\vert
\mathcal{P}_{r, q^2}\right\vert =N(q^2,r))$.

\begin{enumerate}[{\rm (i)}]
 \item   $\omega_{2}(f)\leq (q^{2r} - 1)_2$
for all $f(X)\in\mathcal{P}_{r, q^2}$.

\item $\omega_{2}(f)=(q^{2r} - 1)_2$ for at
least $N(q^2,r)/2$ of the $f(X)\in\mathcal{P}_{r, q^2}$. 
\end{enumerate}
If $r=2^{b}$ then $\omega_{2}(f)=(q^{2r} - 1)_2$ for exactly $
(q^{2r}-1)/(2r)$ of the $f(X)\in\mathcal{P}_{r, q^2}.$
\end{lemma}

\begin{defn}\label{def:D_minus}
Let $\cD_{4r}$ be the
  $\UU*$-irreducible polynomials in
  $\F_{q^2}[X]$  of type~D  and degree $4r$ 
(so that each irreducible factor has degree $r$). 
Let $\cD^-_{4r}$ be the subset of $\cD_{4r}$ consisting of
  those polynomials $g$  with $\omega_2(g) = (q^{2r}-1)_2=r_2(q^2-1)_2$. 
Let $N_{\UU}^-(q, 4r)$ be the number of monic
  $\UU*$-irreducible polynomials $g  \in \F_{q^2}[X]$ of degree $4r$
  such that $\omega_2(g) = (q^{2r}-1)_2$.
\end{defn}
 
We shall now show that $\cD^{-}_{4r}$ contains all monic
  $\UU*$-irreducible polynomials $g  \in \F_{q^2}[X]$ of degree $4r$
  such that $\omega_2(g) = (q^{2r}-1)_2$, so that $|\cD^-_{4r}| =
  N_{\UU}^-(q, 4r)$. 

\begin{lemma}\label{lem:gamma_12}
  Let $g(X)\in \F_{q^2}[X]$ be a $\UU*$-irreducible polynomial with an
  irreducible factor $f(X)$ of degree~$r$.
\begin{enumerate}[{\rm (i)}]
\item If $g(X)$ has type A, B or C, then $\omega_2(g) < (q^2-1)_2$.
\item If $g(X) \in  \cD_{4r}$ then $\omega_2(g) \leq (q^{2r}-1)_2$. At
  least $N(q^2, r)/8$ of the polynomials $g(X) \in \cD_{4r}$
  satisfy $\omega_2(g) = (q^{2r} - 1)_2$. 
\item $N_{\UU}^{-}(q, 4r)=\left\vert\cD^-_{4r}\right\vert\geq N(q^2,r)/8$,
  with equality  if $r = 1$; and  if $r=2^{b-1}\geq 1$ then
  $N^{-}_{\UU}(q, 4r)= (q^{2r}-1)/(8r)$.
\end{enumerate}
\end{lemma}

%\begin{lemma}\label{lem:gamma_12}
%Let $\Gamma \subseteq \F_{q^2}[X]$ denote the set of $\UU*$-irreducible polynomials of
%types
% A, B and C.  Let $r=2^{b-1}m$, where $b\geq1$ and $m$ is odd.
%\begin{enumerate}[{\rm (i)}]
%\item For each $g(X) \in \Gamma$, $\omega_2(g) < (q^2-1)_2$.
%\item For each $g(X) \in
%  \cD_{4r}$, $\omega_2(g) = \omega_2(f) \leq (q^{2r}-1)_2$, and at
%  least $\frac{1}{8}N(q^2, r)$ of the polynomials $g(X) \in \cD_{4r}$
%  satisfy $\omega_2(g) = (q^{2r} - 1)_2$. 
%\item each polynomial $g \in \cD_{4r}\setminus \cD^-_{4r}$ 
%satisfies $\omega_2(g) < (q^{2r}-1)_2 = 2^{b-1}(q^2-1)_2$;
%\item
%$N_{\UU}^{-}(q, 4r) = \left\vert \cD^-_{4r}\right\vert
%\geq\frac{1}{8}N(q^2,r)$, with equality  if $r = 1$;
% whilst
%$N_{\UU}^-(q, 4) =  \cD^-_4 = \frac{1}{8}(q^2-1) = \frac{1}{8}N(q^2, 1)$, 
%and  if $r=2^{b-1}$ (that is, $m=1$) then $N^{-}_{\UU}(q, 4r)= (q^{2r}-1)/8r$.
%\end{enumerate}
%\end{lemma}

\begin{proof}
  (i) It follows from Proposition~\ref{prop:cases} that $\omega_2(g)$ divides
  $2$, $q-1$, $q+1$ for Types, A, B and C, respectively. %Hence
  %$\omega_2(g) \mid \frac{1}{2}(q^2- 1)_2$.
The result follows.

\medskip

\noindent (ii) Let $g(X) \in
\cD_{4r}$. Then $\omega_2(g) \mid (q^{2r} - 1)_2$, by
Proposition~\ref{prop:cases}. 
By Lemma~\ref{L: Pn}, 
$\omega_2(f) = (q^{2r} - 1)_2$ for at least $N(q^2,
r)/2$ of the monic irreducibles in $\mathcal{P}_{r, q^2}$ and 
these
  irreducibles have roots of greater $2$-power order than those that
  correspond to irreducible factors of polynomials of types A, B or C,
  so the result follows.

%\noindent (iii) This is immediate from the definition of $\cD^-_{4r}$.
\medskip

\noindent (iii) The claim that $|\cD^-_{4r}| = N^-_\UU(q, 4r)$ follows
from Part (i), and the bound 
$|N^-_\UU(q, 4r)| \geq N(q^2, r)/8$ follows from Part (ii). 

Let $r=1$.
The set $\cD^-_4$ consists of polynomials
$(X-\zeta)(X-\zeta^{-1})(X-\zeta^{-q})(X-\zeta^q)$ in $\F_{q^2}[X]$
such that the order of $\zeta$ is divisible by $(q^2-1)_2$.
We now count the number of such polynomials. Observe that
$\zeta,\zeta^{-1},\zeta^{-q},\zeta^{q}$ all have the same order and
$\zeta\not\in\{\zeta^{-1},\zeta^{-q},\zeta^{q}\}$. It follows that
$|\{\zeta,\zeta^{-1},\zeta^{-q},\zeta^{q}\}|=4$. The elements of $\F^*_{q^2}$
with order divisible by $(q^2-1)_2$ are precisely the nonsquares, and there
are $(q^2-1)/2$ nonsquares. We take these nonsquares four at a time
to make $\UU*$-irreducible polynomials, and so
$|\cD^-_4|=(q^2-1)/8$.

Now consider $r=2^{b-1}>1$. By Lemma~\ref{L: Pn},  
there are $(q^{2r}-1)/(2r)$ degree $r$ monic irreducible polynomials
 $f(X)$ over $\F_{q^2}$ 
such that $\omega_2(f) = (q^{2r}-1)_2$.  Part (i) 
implies that each such $f(X)$ corresponds to a $\UU*$-irreducible
polynomial $g(X)$ of type D, yielding exactly
$(q^{2r}-1)/(8r)$ such $g(X)$. By Part (ii), 
each such $g(X)$ satisfies $\omega_2(g) = (q^{2r}-1)_2$, 
and hence lies in $\mathcal{D}^-_{4r}$. Conversely 
each polynomial in $\mathcal{D}^-_{4r}$ is of this form. Hence
$|\mathcal{D}^-_{4r}| =(q^{2r}-1)/(8r)$.
\end{proof}

%\begin{lemma}\label{lem:numberpolys} 
%Let $n$ be a positive integer. Then the number of $\UU*$-closed monic polynomials
%$g(X)\in\F_{q^2}[X]$ of degree $2n$ is equal to 
%$q^n + q^{n-1}$.
%\end{lemma}

%\begin{proof}
%Consider the polynomial 
%\[
%g(X) =X^{2n} + a_{2n-1}X^{2n-1} + \cdots + a_1 X + a_0,
%\]
%
%where $a_i \in \F_{q^2}$ for $0 \leq i \leq 2n-1$ and $a_0 \neq
%0$. Its $*$-conjugate polynomial is
%\[
%g^*(X) = X^{2n}  +a_1 a_0^{-1} X^{2n-1} + a_2 a_0^{-1} X^{2n-2} + \cdots
 % + a_{2n-1}a_0^{-1}X + a_0^{-1},
%\]
%
%and its $\sigma$-conjugate polynomial is 
%\[
 % {g}^{\sigma}(X) = X^{2n} + a_{2n-1}^q X^{2n-1} + \cdots + a_0^q.
%\]
%Since $g(X) = g^\sigma(X)$ it follows that  each of the $a_i$ lies in $\F_q$. Since $g^*(X)$ is also  equal to $g(X)$ we see that, for $1\leq i\leq n-1$, the coefficient
%$a_i= a_{2n-i}a_0^{-1}$ is determined by $a_{2n-i}$. In addition, 
%$a_0=a_0^{-1}$ and $a_n=a_na_0^{-1}$. Thus $a_0=\pm1$. If $a_0=1$ 
%then $a_n$ can be any element of $\F_q$, while if $a_0=-1$, then $a_n$ must be $0$.  
%Thus there are $q^{n-1}$ choices for the coefficients $a_i$, $i\ne 0, n$, 
%and for each of these choices, there are $q+1$ choices for $(a_0, a_n)$. 
%\end{proof}

\section{Centralisers of 
  \texorpdfstring{$\UU*$}{}-irreducible, regular semisimple elements}\label{sec:cent}

In this section we investigate the centralisers and normalisers of
cyclic subgroups  of $\gu{n}{q}$ whose generators are regular semisimple,
$\UU*$-irreducible, and conjugate in
$\gu{n}{q}$  to
their inverse. We also count the number of involutions in $\gu{n}{q}$ that
invert these generators.

%We write $\phi = \phi_e$
%for the field automorphism $x \mapsto x^q$ of $\F_{q^{2e}}$, and we
%let $z$ be a primitive element of $\F_{q^{2e}}$, where the value of
%$e$ will depend on the case which we are considering. 
The field $\F_{q^{2m}}$ may be regarded as a vector
space $\F_{q^2}^m$, and from this point of view the
multiplicative group of $\F_{q^{2m}}$ is a subgroup of
$\gl{m}{q^2}$ acting regularly on the nonzero vectors. There is a
single conjugacy class of such  \emph{Singer subgroups} of
$\gl{m}{q^2}$, and their generators are \emph{Singer
  cycles}. Thus if $\langle z \rangle \cong C_{q^{2m}-1}$ is a Singer
subgroup in $\gl{m}{q^2}$ then we may identify $\F_{q^{2}}^m$ with the additive group of the field
$\F_{q^{2m}}$, and $\langle z \rangle$ with the multiplicative group
$\F_{q^{2m}}^*$, so that the Singer cycle $z$ corresponds to multiplication by 
a primitive element~$\zeta$. 
Moreover,   $N_{\gl{m}{q^2}}(\langle z \rangle) = \langle z, s \rangle 
\cong C_{q^{2m}-1}\rtimes C_m\cong\gaml{1}{\F_{q^{2m}}/\F_{q^2}}$, with $s\colon z \mapsto z^{q^2}$ 
corresponding to the field automorphism $\phi\colon\zeta\mapsto \zeta^{q^2}$ 
of $\F_{q^{2m}}$ over $\F_{q^2}$
(see~\cite[Satz II.7.3]{hupp}). 

\subsection[]{Singer subgroups and regular semisimple elements}\label{singer}  

In this subsection we change our unitary form, and
work with matrices written relative to a decomposition
\[
  V=W_0\oplus W_0^\sim\quad
  \textup{where $W_0$ and $W_0^\sim$ are totally isotropic subspaces.}
\]
Choose an ordered basis
$(v_1,\dots,v_{2m})$ for $V$ such that $W_0:=\langle v_1,\dots,v_{m} \rangle$ and
$W_0^\sim:=\langle v_{m+1},\dots,v_{2m} \rangle$.   
In this subsection our unitary form has Gram matrix
\begin{equation}\label{E:J}
  J = \begin{pmatrix} 0 & I_m \\ I_m & 0 \end{pmatrix},
\end{equation}
where $I_m$ is the identity matrix. 
We denote this unitary group by $\gu{}{J}\cong\gu{2m}{q}$.
Consider the monomorphism
\begin{equation}\label{defA} 
  \alpha\colon \gl{m}{q^2} \rightarrow \gl{2m}{q^2}\quad\textup{defined by}\quad
  a\mapsto \begin{pmatrix} a & 0 \\ 0 & a^{-\sigma T} \end{pmatrix}.
\end{equation}
To see that $\alpha(a) \in \gu{}{J}$, it is straightforward to check that $\alpha(a)J\alpha(a^{\sigma T})=J$, or
equivalently $\alpha(a)^J=\alpha(a^{-\sigma T})$. 
%\begin{equation}\label{eq:AaJ} 
 % J^{-1} \alpha(a) J 
 % =   \begin{pmatrix}   0 & I_m \\ I_m & 0 \end{pmatrix}
%\begin{pmatrix} a & 0   \\   0 & a^{-\sigma T} \end{pmatrix}
 % \begin{pmatrix}   0 & I_m \\ I_m & 0 \end{pmatrix}
 % =   \begin{pmatrix} a^{-\sigma T} & 0 \\ 0 & a \end{pmatrix} = \alpha(a^{-\sigma T}).
%\end{equation}
Thus the automorphism $a\mapsto a^{-\sigma T}$ of $\gl{m}{q^2}$ induces
the automorphism $\alpha(a)\mapsto \alpha(a)^J$ on the image of $\alpha$.

\subsection{Types A and B: \texorpdfstring{$|\{f, f^\sigma, f^*,f^\sim\}|=2$}{}, with \texorpdfstring{$f \neq f^\sim$}{}}\label{sec:AB} 

To assist with our analysis of Types~A and~B of Proposition~\ref{prop:cases}, 
we first consider the more general situation where 
$y \in \gu{2m}{q}$ is regular semisimple  with
characteristic polynomial $f(X)f^\sim(X)$, where $f(X)$ is
irreducible. We later add the condition that $y$ is conjugate to
its inverse. 
% such that 
%$f(X)$ is irreducible of degree~$m$, and $f(X)\ne f^\sim(X)$. 
We can then write $V = W \oplus W^\sim = \F_{q^2}^m \oplus 
\F_{q^2}^m$ where the restrictions $y|_{W}$ and $y|_{W^\sim}$ 
to $W$ and $W^\sim$ have characteristic polynomials $f(X)$ and 
$f^\sim(X)$, respectively. It follows from  \cite[Definition 2.2 and Lemma 2.4]{npp}
that the subspaces $W$ and $W^\sim$ are totally isotropic. 

Before proceeding with our analysis we make a few general remarks.

\begin{remark}\label{rem:AB} 
Let $a\in \gl{m}{q}$. Then $a$ is conjugate in $\gl{m}{q}$ to its transpose.
In fact, by a result of Voss~\cite{v} 
(see also~\cite[Theorem 66]{k}), 
%(see also~\cite[Theorem 1]{tz} or Kaplansky's book \cite[Theorem 66]{k}), 
there is a symmetric matrix
$c\in\gl{m}{q}$ which conjugates $a$ to its transpose, that is 
$c=c^T$ and $c^{-1}ac=a^T$.

Let $a \in \gl{m}{q}$ be irreducible, with
characteristic polynomial $f(X)$.
If $a'\in\gl{m}{q}$ is also irreducible and
$\zeta^{q^i}$ is a root of 
its characteristic polynomial for some $i$, then $a'$ also has 
characteristic polynomial $f(X)$, and consequently $a'$
is conjugate to $a$ in $\gl{m}{q}$.
\end{remark}

%\newline\colorbox{red}{mention embedding $\phi\colon GL(m,q^2)\to\gu{2m}{J}$. Then
%$Z=\phi(z)$.}

\begin{lemma}\label{lem:gu_cent}
Let $V = W_0 \oplus W_0^\sim$, with $W_0$ totally isotropic. Let $J$
 be the Gram matrix of the form on $V$, as in \eqref{E:J},  and let
$\alpha$ be as in~\eqref{defA}.
 Let $z \in\gl{m}{q^2}$ be a Singer cycle for $\gl{m}{q^2}$, and let $s
 \in N_{\gl{m}{q^2}}(\langle z\rangle)$ be such that $z^s = z^{q^2}$. 
\begin{enumerate}[{\rm (i)}]
\item Let $H := \alpha(\gl{m}{q^2})$. Then $H$ is the stabiliser  in
  $\gu{}{V}$ of both $W_0$ and $W_0^\sim$. The stabiliser of the decomposition  
$V=W_0\oplus W_0^\sim$ is $\stab_{\gu{}{J}}(W_0\oplus W_0^\sim)=H\rtimes\langle J\rangle$.

\item Let $Z := \alpha(z)$. Then $C_{\gu{}{J}}(Z)=\langle Z\rangle\cong C_{q^{2m}-1}$,
and $N := N_{\gu{}{J}}(\langle Z\rangle )=\langle Z, B\rangle \cong C_{q^{2m}-1}. C_{2m}$,
where $B=\alpha(b)J$ for some  $b\in\gl{m}{q^2}$ such that $b^{-1}zb= z^{q \sigma T}$. 
Moreover, 
\[ Z^{B}=Z^{-q} \quad \mbox{ and } 
 B^2=\alpha(b^{1-\sigma T}) = \alpha(z^\ell s) \mbox{ for some } \ell
\in \mathbb{Z}.
\]
\item Let $y \in \gl{2m}{q^2}$ be regular semisimple with
characteristic polynomial $f(X)f^\sim(X)$, for some irreducible polynomial
$f(X)$. Then some conjugate of $\langle y \rangle$ lies in $\langle Z \rangle \leq \gu{}{J}$, and 
has centraliser~$\langle Z\rangle$ and normaliser $N$ in $\gu{}{J}$. 
\end{enumerate}
\end{lemma}

\begin{proof}
(i) %Let $H$ be the image of the monomorphism $\alpha$ defined in~\eqref{decomp}. 
The fact that $H \leq \gu{}{J}$ follows from our remark after \eqref{defA}. 
The spaces 
$W_0$ and $W_0^\sim$ are non-isomorphic irreducible $\F_{q^2}H$-submodules of $V$.
Define $\widehat{H}$ to be the stabiliser in $\gu{}{J}$ of both $W_0$
and $W_0^\sim$. We shall show that $\widehat{H} = H$.  
The restriction $\widehat{H}|_{W_0}=H|_{W_0}\cong \gl{m}{q^2}$, and the subgroup $K$ of
$\widehat{H}$ fixing $W_0$ pointwise must fix the hyperplane 
$\langle w\rangle^\perp \cap W_0^\sim$ of $W_0^\sim$ for each non-zero $w\in W_0$. Thus
$K$ induces a subgroup of scalar matrices on $W_0^\sim$.
However
$\left(\begin{smallmatrix}I&0\\0&\lambda I\end{smallmatrix}\right)\in\gu{}{J}$
  implies $\lambda=1$, so $K$ is trivial.
%  since $K\subset\gu{}{J}$, this subgroup $K$ is trivial.
  It follows that $\widehat{H}=H$. 
Finally each element of $\stab_{\gu{}{J}}(W_0\oplus W_0^\sim)$ either 
fixes setwise, or interchanges, the subspaces $W_0$ and $W_0^\sim$, and
it is straightforward to check that $J\in\gu{}{J}$,  and that $J$ interchanges
these two subspaces. Hence $\stab_{\gu{}{J}}(W_0\oplus W_0^\sim)=H\rtimes \langle J\rangle$.

\medskip
\noindent (ii)  The spaces  $W_0$ and $W_0^\sim$ are also non-isomorphic irreducible
$\F_{q^2}\langle Z\rangle$-modules, so that
$C_{\gu{2m}{q}}(Z)$ fixes each of $W_0, W_0^\sim$ setwise,
and so is contained in their stabiliser $H$, by Part~(i). 
Since $C_{\gl{m}{q^2}}(z)
=\langle z\rangle \cong C_{q^{2m}-1}$ 
(see~\cite[Satz II.7.3]{hupp}), we have
$C_{\gu{2m}{q}}(Z) =  \langle Z\rangle$. To prove the second
assertion we note that
$N$ must preserve the decomposition $V=W_0\oplus W_0^\sim$, 
since $N$ normalises $\langle Z\rangle$.
Thus $N\leq H\rtimes \langle J\rangle$, by Part (i). Also
$N\cap H$ is the image under $\alpha$ of 
$N_{\gl{m}{q^2}}(\langle z\rangle)$, and this is $\langle \alpha(z), \alpha(s)\rangle$,
(again see~\cite[Satz II.7.3]{hupp}). Now $N\cap H$ has index at most 2 in $N$,
and we shall construct $B\in N\setminus (N\cap H)$.

Let $f(X) = c_{z}(X)$, and let $\zeta$ be a root of $f(X)$. 
Then the roots of $f(X)$ are $\zeta^{q^{2i}}$ for $0\leq i\leq m-1$.
Since $|z|=q^{2m}-1$, the element  $z^q$ is irreducible and one 
of the roots of $c_{z^q}(X)$
 is $\zeta^q$. Similarly $z^{q\sigma}$ is irreducible and one of the
roots of $c_{z^{q \sigma}}(X)$
 is $(\zeta^q)^\sigma = \zeta^{q^2}$. Hence, by Remark~\ref{rem:AB},
$z$ is conjugate in $\gl{m}{q^2}$ to $z^{q\sigma}$ which in turn is conjugate
to $z^{q\sigma T}$. Let $b\in\gl{m}{q^2}$ be such that $b^{-1}zb=z^{q
  \sigma T}$, and let $B := \alpha(b)J$. Then, using the fact noted after \eqref{defA} that $\alpha(a)^J=\alpha(a^{-\sigma T})$ for all $a$, 
\[
Z^B = (\alpha(z)^{\alpha(b)})^J=\alpha(z^b)^J = \alpha(z^{q \sigma T})^J= \alpha((z^{q \sigma T})^{-\sigma T})
= \alpha(z^{-q}) = Z^{-q}.
\]
In particular, $B$ normalises $\langle Z\rangle$ and interchanges 
$W_0$ and $W_0^\sim$. Thus $N=\langle Z, \alpha(s), B\rangle$. A
straightforward computation shows that $B^2$ and $\alpha(s)$ both conjugate $Z$ to $Z^{q^2}$, 
so $B^2 \alpha(s)^{-1}\in C_{\gu{}{J}}(Z) = \langle Z\rangle$. Hence
$N=\langle Z, B\rangle$, and $B^2=\alpha(z^\ell s)$ for some integer~$\ell$. 
Another easy computation yields $B^2=(\alpha(b)J)^2=\alpha(b^{1-\sigma T})$, so 
$b^{1-\sigma T} = z^\ell s$.

\medskip
\noindent (iii) The fact that $y$ is conjugate to an element of
$\gu{2m}{q}$, and hence to an element of $\gu{}{J}$,
is immediate from Corollary~\ref{cor:gu_cp}. 
The primary decomposition of $V$ with respect to $y$, as discussed 
at the beginning of \S\ref{sec:AB}, is $V=W \oplus W^\sim$ 
where  both $W$ and $W^\sim$ are totally isotropic, and the 
restrictions of $y$ to $W$ and $W^\sim$ are irreducible with 
characteristic polynomials $f(X)$ and $f^\sim(X)$, respectively. 
Since $\gu{}{J}$ is transitive on ordered pairs of disjoint 
totally isotropic $m$-dimensional subspaces, replacing $y$ 
by a conjugate, if necessary, we may assume that $W=W_0$ and 
$W^\sim = W_0^\sim$.  

Then $y$ fixes both $W_0$ and $W_0^\sim$ setwise, and hence
$y=\alpha(y_0)$ for some $y_0\in\gl{m}{q^2}$, by Part (ii). 
From the definition of $\alpha(y_0)$, we have $y_0 := y|_{W_0}$.
It follows from \cite[Satz II.7.3]{hupp} that there exists 
$c\in \gl{m}{q^2}$ such that $y_0^c \in \langle z\rangle$ and 
$N_{\gl{m}{q^2}}(\langle y_0\rangle)^c = 
N_{\gl{m}{q^2}}(\langle z\rangle) = \langle z, s\rangle$. Thus, 
replacing $y$ by its conjugate $y^{\alpha(c)}$, we have
$C_{\gu{2m}{q}}(y) = C_{\gu{2m}{q}}(Z) = \langle Z \rangle$
and $N_{\gu{2m}{q}}(\langle y\rangle) \cap H =\langle Z, \alpha(s)\rangle$. 
Since $N_{\gu{2m}{q}}(\langle y\rangle)$ normalises $C_{\gu{2m}{q}}(y) 
=\langle Z \rangle$, it is contained in $N$, and since 
$B=\alpha(b)J$ normalises the cyclic group
$\langle Z \rangle$, it also normalises $\langle y \rangle$.
Thus $N_{\gu{2m}{q}}(\langle y\rangle)= N$.
\end{proof}

As a corollary we count the number of involutions which invert 
an element $y$ as in Lemma~\ref{lem:gu_cent}(iii). 

\begin{cor}\label{cor:ABinvols} 
Let $y \in \gl{2m}{q^2}$ be regular semisimple, with
$\UU*$-irreducible characteristic
polynomial $g(X) = f(X) f^\sim(X)$, where $f(X)$ is irreducible and 
$f(X)\ne f^\sim(X)$. 
Then up to conjugacy $y \in \gu{2m}{q}$, and
 some element of $\gu{2m}{q}$ inverts this conjugate of $y$ 
if and only if one of the following holds.
\begin{enumerate}[{\rm (i)}]
\item $f(X)$ is of Type A and exactly $q^m+1$ involutions invert $y$.
%, and $|y|_2 \leq 2$.

\item $f(X)$ is of Type B %, $m$ is odd, and $|y|$
%divides $q^m-1$. 
and exactly $q^m-1$ involutions invert $y$.
%and $|y|_2 \leq (q-1)_2$.
\end{enumerate}
\end{cor}

\begin{proof}
Note that $f(X)\ne X\pm 1$ since $f(X)\ne f^\sim(X)$, and hence, in particular, 
$|y|>2$. Let $z$, $W_0$ and $N = \langle Z, B \rangle$ be as in 
Lemma~\ref{lem:gu_cent}.
By Lemma~\ref{lem:gu_cent}(iii), we may assume up to conjugacy
 that $y=\alpha(z^i)$ for some 
$i \in \{1, \ldots, q^{2m} - 2\}$ such that $z^i$ is irreducible on $W_0$,
and every element of $\gu{2m}{q}$ that inverts $y$, if one exists, must lie in $N$.
%
%Also $N=\langle Z, B\rangle$, as in Lemma~\ref{lem:gu_cent}(ii). 
If both $n$ and $n'$ invert $y$, then $n'n^{-1}$
centralises $y$ and hence, by Lemma~\ref{lem:gu_cent}(iii), $n'=\alpha(z^i)n$,
for some $i$. Moreover, if $n$ inverts $y$ then certainly $\alpha(z^i )n$ 
also inverts $y$ for each $i$. Hence either $0$ or $|z|=q^{2m}-1$
elements invert $y$. We show first that such inverting elements exist,
in both types, and then we count the number of them which are involutions. 

It follows from Lemma~\ref{lem:gu_cent}(ii) that if an
element $n$ of $\gu{2m}{q}$ inverts $y$, then $n$ is of the form $\alpha(z^j)B^k \in N$ 
for some $j,k$ such that $0 \leq j \leq q^{2m} - 2$ and $1 \leq k \leq 2m-1$. 
Note that $k\ne 0$ since $|y|>2$ implies that $n$ does not centralise
$y$. Recall from Lemma~\ref{lem:gu_cent}(ii) that $Z^B = Z^{-q}$, so
$y^B = y^{-q}$. 

Suppose first that $k$ is even. Then 
\[
  y^{-1} = n^{-1}y n = y^{q^k}
\] 
and so $z^{iq^k} = z^{-i}$, which is equivalent to $z^{i(q^k + 1)} = 1$. 
This implies that $z^i \in \F_{q^{2k}} \cap \F_{q^{2m}} = 
\F_{q^{2(k, m)}}$ (identifying $z$ with an element of $\F_{q^{2m}}$). 
However, $z^i$ acts irreducibly on $\F_{q^2}^m$, 
so $z^i$ lies in no proper subfield of $\F_{q^{2m}}$ that contains
$\F_{q^2}$. 
Hence %$(k, m) = m$, that is, $m$ divides $k$ and therefore 
$k=m$, and 
%since $1\leq k < 2m$. 
in particular
$m$ is even (since $k$ is assumed to be even), and we are in Type~A
by Proposition~\ref{prop:cases}. Here $|y|=|z^i|$ divides $q^m+1$,
and hence $n = \alpha(z^j)B^m$ inverts $y$ for each $j \in \{0,
\ldots, q^{2m} - 2\}$.  
%Notice that
%since  $|y|$ divides $q^m+1$ with $m$ even
%and $q$ odd, $|y|_2 \leq 2$.

Now suppose  that $k$ is odd. Then 
\[
  y^{-1} = n^{-1}y n = y^{-q^k}
\] 
and so $z^{i(q^k - 1)} = 1$, whence 
$z^i \in \F_{q^{k}} \cap \F_{q^{2m}} = 
\F_{q^{(k, m)}} \subset \F_{q^{2(k, m)}}$ since $k$ is odd. 
%Again, since $z^i$ acts irreducibly 
%on $\F_{q^2}^m$,
As in the previous case,  $z^i$ lies in no proper 
subfield of $\F_{q^{2m}}$ containing $\F_{q^2}$.
Hence $(k, m) = m$, so again %$m$ divides $k$ and therefore 
$k=m$. 
%since $1\leq k < 2m$. 
Thus $m$ is odd (since $k$ is odd), 
and we are in Type~B 
by Proposition~\ref{prop:cases}. Here $|y|=|z^i|$ divides $q^m-1$, 
and hence $n = \alpha(z^j)B^m$ inverts $y$, for each $j \in \{0,
\ldots, q^{2m} - 2\}$.
% Notice also that
%since $m$ is odd,  $|y|_2$ is a divisor of $(q^m-1)_2 =
%(q-1)_2$. 
%

We now count the inverting involutions. Observe that, by Lemma~\ref{lem:gu_cent}(ii), $B^2=\alpha(z^\ell s)$ for some $\ell\in\{0,\dots,q^{2m}-2\}$,
where $s^{-1}zs=z^{q^2}$, with $z^{q^{2m}-1}=1$ and $s^m = 1$. Hence 
\begin{equation}\label{E:B2m}
  B^{2m} = \alpha(z^\ell s)^m = \alpha( (z^\ell s)^{m-1} s z^{\ell q^2}) =\dots = \alpha(z^{\ell (q^{2m}-1)/(q^2-1)}).
\end{equation}

\noindent
{\it Type~A:}\quad %Here %$m$ is even, $|y|$ divides $q^m+1$,
%and 
Each inverting element is of the form $n = \alpha(z^j)B^m$,
for some $j \in \{0, \ldots, q^{2m}-2\}$, and $m$ is even.
Such an element $n$ is an involution if and only if
\[
  1 = n^2 = \alpha(z^j)B^m \alpha(z^j)B^m = \alpha(z^j) \alpha(z^{jq^m})B^{2m} = \alpha(z^{j(q^m+1)})B^{2m}.
\] 
Hence, by \eqref{E:B2m},  $n^2=1$ if and only if 
$q^{2m}-1$ divides $j(q^m+1) + \ell (q^{2m}-1)/(q^2-1)$. 
Since $m$ is even, this holds
if and only if $q^{m}-1$ divides $j + \ell (q^{m}-1)/(q^2-1)$.  In particular
 $j=j' (q^m-1)/(q^2-1)$ (there are $(q^m+1)(q^2-1)$ 
integers $j$ with this property in $\{1, \ldots, q^{2m}-2\}$),
and in addition 
$q^2-1$ divides $j'+\ell $. Thus we have exactly $q^m+1$ possibilities for $j$,
and hence there are exactly $q^m+1$ involutions which invert $y$.

\medskip\noindent
{\it Type~B:}\quad 
Each inverting element is of the form $n = \alpha(z^j)B^m$,
for some $j \in \{0, \ldots, q^{2m}-2\}$, and $m$ is odd.
Such an element $n$ is an involution if and only if
\[
  1 = n^2 = \alpha(z^j)B^m \alpha(z^j)B^m = \alpha(z^j) \alpha(z^{-jq^m})B^{2m} = \alpha(z^{-j(q^m-1)})B^{2m}.
\] 
Thus, by \eqref{E:B2m},  $n^2=1$ if and only if 
$q^{2m}-1$ divides $-j(q^m-1) + \ell (q^{2m}-1)/(q^2-1)$. 
In particular,  $q^m-1$ must divide $\ell (q^{2m}-1)/(q^2-1)$,
or equivalently, $\ell (q^{m}+1)/(q^2-1)$ must be an integer. Given this condition,
$n^2=1$ if and only if $q^{m}+1$ divides $-j + \ell (q^{m}+1)/(q^2-1)$. As
$j$ runs through $\{0, \ldots, q^{2m}-2\}$, there are exactly $q^m-1$ values with this property.
Thus there are either $0$ or $q^m-1$ inverting involutions. 
To prove the latter holds, we construct an inverting involution. 

We have $m$ odd and $f(X)=f^\sigma(X)$, so all of the 
coefficients of $f(X)$ lie in $\F_{q}$,
and hence all of its roots lie in $\F_{q^m}$. This means that some conjugate of 
$z^i$ by an element of $\gl{m}{q^2}$ lies in 
$\gl{m}{q}$ (the subgroup of $\gl{m}{q^2}$ of matrices
with entries in $\F_q$). We may therefore conjugate $y=\alpha(z^i)$ 
by an element of $\alpha(\gl{m}{q^2}) \subseteq \gu{2m}{q}$ 
(see \eqref{defA}) and obtain an element in $\alpha(\gl{m}{q})$. Let us replace $y$
by  this element so that  $y=\alpha(z^i)$ with $z^i\in\gl{m}{q}$.
By Remark~\ref{rem:AB}, there is a symmetric matrix
$c\in\gl{m}{q}$ which conjugates $z^i$ to its transpose, that is, 
$c=c^T$ and $c^{-1}z^ic=(z^i)^T$. Note that $(z^i)^T=(z^i)^{\sigma T}$ 
and $c^{\sigma T}=c$, since
$z^i, c\in\gl{m}{q}$ and $c=c^T$. Therefore 
$c^{\sigma T} (z^{i})^{-\sigma T} c^{-\sigma T} = (c^{-1}z^{-i}c)^T = z^{-i}$, and it
follows from \eqref{defA} that $\alpha(c)^{-1} y \alpha(c)$ is the block diagonal matrix $\alpha(z^{i T})$ with
diagonal components $z^{iT}, z^{-i}$. Thus $C := \alpha(c)J$ conjugates
$y$ to $\alpha(z^{-i})=y^{-1}$.      
Moreover $J\alpha(c)J = \alpha(c^{-\sigma T}) = \alpha(c^{-1})$,
and hence   $C^2 = \alpha(c)J\alpha(c)J = \alpha(c) \alpha(c^{-1})=1$,
that is, $C$ is an involution inverting $y$.
\end{proof}

\subsection{Type~C: \texorpdfstring{$|\{f, f^\sigma, f^*, f^\sim\}| = 2$, $f = f^\sim$}{}}
We now consider regular semisimple $y\in\gu{2m}{q}$ with
$\UU*$-irreducible characteristic polynomial  
$c_y(X) = f(X)f^*(X)$,
where $f(X) = f^\sim(X)$ has degree $m$.  
So $y$ is in Type~C of Proposition~\ref{prop:cases}, and in particular $m$ is odd. 
The primary decomposition of $V$ as an $\F_{q^2}\langle y\rangle$-module is 
$V=U\oplus U^*$, where the restrictions $y_1 := y|_{U}$ 
and $y_2 := y|_{U^*}$ have characteristic
polynomials $f(X)$ and $f^*(X)$, respectively. Reasoning in exactly
the same way as in the proof of \cite[Lemma 2.4]{npp}, we see that 
$U^* \leq U^\perp$. Since $\dim U^\ast = m = \dim U$, we
deduce that $U^* = U^\perp$, and so both $U$ and $U^\ast$ are
nondegenerate.

 For the analysis in 
this subsection it is convenient to work with matrices with respect to an ordered 
basis $(v_1,\dots,v_{2m})$ where $U=\langle v_1,\dots,v_m\rangle$ and 
$U^*=\langle v_{m+1},\dots,v_{2m}\rangle$, and with Gram matrix
$J=I_{2m}$, where $I_{2m}$ denotes the identity matrix.
The stabiliser in $\gu{2m}{q}$ of the subspace $U$ (and hence also of 
$U^* = U^\perp$) is $H := \stab_{\gu{2m}{q}}(U) = \gu{}{U}\times \gu{}{U^*}
\cong \gu{m}{q}\times \gu{m}{q}$. It is convenient to write %vectors of $V$ as pairs 
%$(u,u')$ with $u,u'\in\F_{q^2}^m$, and to write 
elements of $H$
as pairs $(h,h')$ with $h,h'\in\gu{m}{q}$. The stabiliser in $\gu{2m}{q}$ of the 
decomposition $V=U\perp U^*$ is $\widehat{H} := H\cdot \langle \tau\rangle$, where
$\tau\colon (h,h') \mapsto (h',h)$ for $(h,h') \in H$.

By \cite[Satz II.7.3]{hupp},
we may replace $y$ by a conjugate in $H$ such that 
$y_1$ and $y_2$ are contained in the same Singer
subgroup $\langle z\rangle \cong C_{q^m + 1}$ of $\gu{m}{q}$, and moreover, such that 
$y_2$ is equal to $y_1^{-1}$.
%Note that $|z|=q^m+1$. %, with $c_{y_1}(X) = f(X)$ and $c_{y_1^{-1}}(X) = f^*(X)$.

\begin{lemma}\label{lem:TypeC}
  Let $y \in \gl{2m}{q^2}$ be regular semisimple, with $c_y(X)$ a
  $\UU*$-irreducible polynomial in Type C. Then up to conjugacy, $y
  \in \gu{2m}{q}$, $m$ is odd,  and with the notation from the previous
  two paragraphs
\begin{enumerate}[{\rm (i)}]

\item $N_{\gu{2m}{q}} (\langle z \rangle \times \langle z \rangle ) 
= \langle z, \phi \rangle \wr \langle \tau \rangle \cong \gamu{1}{q^m} \wr C_2$ 
where $\phi\colon z^i \mapsto z^{iq^2}$; 

\item $y \in C_{\gu{2m}{q}}(\langle y\rangle) \leq H$,  and $
  C_{\gu{2m}{q}}(\langle y\rangle)  = \langle z \rangle \times 
\langle z \rangle \cong C_{q^m + 1}^2$;

\item $y$ is inverted by precisely $q^m + 1$ involutions in
  $\gu{2m}{q}$.
%\item $|y|_2 \leq (q+1)_2$. 
\end{enumerate}
\end{lemma}

\begin{proof}
First note that by Corollary~\ref{cor:gu_cp}, up to conjugacy $y \in \
\gu{2m}{q}$, and then by our discussion before the lemma, we can
assume that $y = (y_1, y_1^{-1}) \in H$. 

\medskip\noindent 
(i)  Let $C:= \langle z \rangle \times \langle z \rangle$. Then 
the only proper non-trivial $\F_{q^2}C$-submodules are
$U$ and $U^*$, and so $C\leq H$ and $N_{\gu{2m}{q}} (C) \leq \widehat{H}$.
Now the normaliser of $\langle z \rangle$ in $\gu{m}{q}$ is $N_1:= \langle z, \phi
\rangle$, and so $N_{\gu{2m}{q}} (C) = N_1 \wr C_2$.

\medskip\noindent
(ii) As observed above, $U$ and $U^*$ are non-isomorphic irreducible 
$\F_{q^2}\langle y\rangle$-submodules, and we may assume that $y=(y_1,y_1^{-1})$.
Hence $C_{\gu{2m}{q}}(\langle y\rangle) \leq H$. Moreover,
since $y_1$ is irreducible on $U$, its centraliser in $\gu{m}{q}$
is $\langle z \rangle$, and hence  $C_{\gu{2m}{q}}(\langle y\rangle) = C$.

\medskip\noindent
(iii) Since $y=(y_1,y_1^{-1})$, the involutory map $\tau$ conjugates $y$ to $y^{-1}$.
It follows that the elements which conjugate $y$ to $y^{-1}$ 
are precisely the elements of the coset
$C\tau$. These elements are of the form $(z_1, z_2)\tau$, for some 
$z_1, z_2\in\langle z \rangle$, and
are involutions if and only if $z_2 = z_1^{-1}$. Thus
there are precisely $|z| = q^m + 1$ involutions which invert $y$.
%\noindent (iv)  Here $|y| = |y|_1$ is a divisor of $|z| = q^m+1$, so
%$|y|_2$ divides $(q^m+1)_2$. Since $m$ is odd, $(q^m+1)_2 = (q+1)_2$. 
\end{proof}

\subsection{Type~D: \texorpdfstring{$|\{f, f^\sigma, f^*, f^\sim\}| = 4$}{}}
We now consider regular semisimple $y\in\gu{4m}{q}$ with
$\UU*$-irreducible 
characteristic polynomial  
%Here we consider a regular semisimple element $y\in\gu{4m}{q}$ with characteristic polynomial  
\[
  c_y(X) = f(X)f^\sim(X) f^*(X)f^\sigma(X),
\]
where $f(X)$ has degree $m$, so that
$y$ is in Type~D of Proposition~\ref{prop:cases}. 
The primary decomposition of $V$ as an $\F_{q^2}\langle y\rangle$-module is 
$V=U\oplus U^\sim\oplus U^* \oplus U^\sigma$, where the restrictions 
$y|_{U}$, $y|_{U^\sim}$, $y|_{U^*}$, $y|_{U^\sigma}$ have characteristic
polynomials $f(X)$, $f^\sim(X)$, $f^*(X)$,  and $f^\sigma(X)$, respectively.
Hence these four $m$-dimensional subspaces are pairwise non-isomorphic 
$\F_{q^2}\langle y\rangle$-submodules,
and so the centraliser $C:= C_{\gu{4m}{q}}(y)$ lies in the stabiliser
$H$ in $\gu{4m}{q}$ of all four submodules $U$, $U^\sim$, $U^\ast$ and
$U^\sigma$.

Let $W := U \oplus U^\sim$ and $W^* := U^* \oplus U^\sigma$. Then the
characteristic polynomials of $y|_{W}$ and $y|_{W^*}$, namely $g(X):=f(X) f^\sim(X)$ and $g^*(X) = g^{\sigma}(X) =
f^*(X) f^\sigma(X)$,  are both $\sim$-invariant. Thus both $W$ and $W^*$ are non-degenerate,
and $V=W\perp W^*$. Moreover on considering $y|_{W}$, $y|_{W^*}$ 
as in \S\ref{sec:AB}, we see by Lemma~\ref{lem:gu_cent} that each of the four subspaces
$U, U^\sim, U^*, U^\sigma$ is totally isotropic. 

For the analysis in 
this subsection it is convenient to work with matrices with respect to an ordered 
basis $(v_1,\dots,v_{4m})$ of $V$, where $U=\langle v_1,\dots,v_m\rangle$,  
$U^\sim=\langle v_{m+1},\dots,v_{2m}\rangle$, $U^*=\langle v_{2m+1},\dots,v_{3m}\rangle$, and  
$U^\sigma=\langle v_{3m+1},\dots,v_{4m}\rangle$, and with the Gram matrix 
\[
  J = \left( \begin{array}{cccc}
  0 & I_m & 0 & 0\\
  I_m & 0 & 0 & 0 \\
  0 & 0 & 0 & I_m \\
  0 & 0 & I_m & 0 
  \end{array} \right),
\]
where $I_m$ denotes the identity matrix. 
Then, by Lemma~\ref{lem:gu_cent}, the subgroup of $\gu{}{J}$ leaving each of $U, U^\sim, U^*$ and $U^\sigma$ invariant is
\[
H= \left\{  \left( \begin{array}{cc}
\alpha(a) & 0 \\
0    & \alpha(b) \end{array} \right) \ \mid a, b\in \gl{m}{q^2} \right\}
\]
where the matrices $\alpha(a), \alpha(b)\in\gu{2m}{q}$ are as defined in \eqref{defA}.
We note that $\gu{}{J}$ contains
\[
\tau := \left( \begin{array}{cc}
0 & I_{2m} \\
I_{2m} & 0 \end{array} \right),
\]
which interchanges the subspaces $W$ and $W^*$ and normalises $H$.
We often write %vectors of $V$ as 4-tuples 
%$(u,u^\sim,u^*,u^\sigma)$ with $u,u^\sim,u^*,u^\sigma\in\F_{q^2}^m$, 
%and 
elements of $H$
as pairs $(\alpha(a), \alpha(b))$ with $a,b\in\gl{m}{q^2}$. Since $y_1 := y|_U\in\gl{m}{q^2}$ 
is irreducible, it is contained in a Singer subgroup $\langle z \rangle$ of 
$\gl{m}{q^2}$, and it follows from Lemma~\ref{lem:gu_cent} that $y|_W = \alpha(y_1)$ and
$C_{\gu{}{W}}(y|_W) = \langle Z\rangle$, where $Z:=\alpha(z)$. Also, since $y|_{U^*}$, 
$y|_{U^\sigma}$ have characteristic polynomials $f^*(X)$, $f^\sigma(X)=(f^{\sim}(X))^*$, 
we may replace $y$ by a conjugate in $H$ so that $y|_{W^*} = \alpha(y_1^{-1})$. Thus
we may assume that $y = (\alpha(y_1), \alpha(y_1^{-1}))$. 

\begin{lemma}\label{lem:TypeD} 
Let $y \in \gl{4m}{q^2}$ be regular semisimple, with $c_y(X)$ a
$\UU*$-irreducible polynomial in Type D. 
%, f, C, H, W, W^*, y_1, \tau  $ be as above, so
%that $V = W \perp W^*$ is non-degenerate. 
Then, up to conjugacy, and with the previous notation, the following hold:
\begin{enumerate}[{\rm (i)}]
\item $y = (\alpha(y_1), \alpha(y_1^{-1})) \in \gu{4m}{q}$, 
where $y_1\in\gl{m}{q^2}$, with
characteristic polynomial $f(X)$, $y_1$ is contained in a Singer
subgroup $\langle z\rangle$, and $C:= C_{\gu{4m}{q}}(y) 
= \langle Z \rangle  \times \langle Z \rangle \cong C_{q^{2m} - 1}^2$,
where $Z=\alpha(z)$; 

\item $N_{\gu{4m}{q}}(C) = N \wr C_2 = \langle Z, B \rangle
  \wr \langle \tau \rangle$, with $N, B$, as in
  Lemma~{\rm\ref{lem:gu_cent}};
%[Note that the normaliser of $\langle y \rangle$ is an index $|P|$ subgroup of $N_{\gu{4m}{q}}(C) $.]

\item $y$ is inverted by precisely $q^{2m} - 1$ involutions in
  $\gu{4m}{q}$
\item The integer $m$ can be even or odd. 
Let $m = 2^{b-1}r$ with $b \geq 1$ and $r$ odd. Then $|y|_2 \leq
  2^{b-1}(q^2-1)_2$, and equality can be attained in this bound. 
\end{enumerate}
\end{lemma}

\begin{proof}
\noindent (i)  The assertions about $y$ follow from
Corollary~\ref{cor:gu_cp}, 
and the discussion above.
The structure of $C$ follows from  Lemma~\ref{lem:gu_cent}
applied to $y|_W=\alpha(y_1)$ and $y|_{W^*} = \alpha(y_1^{-1})$.

\medskip
\noindent (ii)  The normaliser $N_{\gu{}{W}}(\langle y|_{W} \rangle)) 
= N = \langle Z, B\rangle$, as in
Lemma~\ref{lem:gu_cent}, and $N_{\gu{4m}{q}}(C)$ 
is therefore equal to $N \wr \langle \tau \rangle$. 

%$N_{\gu{4m}{q}}(\langle y\rangle)$ 
%is contained in the normaliser of $C_{\gu{4m}{q}}(\langle y
%\rangle) = C$

\medskip
\noindent (iii) The element $\tau$ is an involution which inverts $y=(y_1,y_1^{-1})$, 
and hence, if $x \in \gu{}{V}$ inverts $y$, then $x$ lies in the 
coset $C\tau$ of the centraliser $C$ of $y$, so $x  = (\alpha(z_1), \alpha(z_2))\tau$, 
for some $z_1, z_2\in \langle z\rangle$. The
condition $x^2 = 1$ is equivalent to $z_2 = z_1^{-1}$. Thus there
are precisely $q^{2m} - 1$ involutions inverting $y$.

\medskip

\noindent (iv) The order of $y$ is equal to $|y_1|$, which is a divisor of
$q^{2m}-1$. Now $(q^{2m} - 1)_2 = (q^{2^br} - 1)_2 =
2^{b-1}(q^2-1)_2$. To see that equality may be attained
in the bound, notice that we may set $y_1 = z$, so that $|y| =
q^{2m}-1$: in this case the roots of $f(X) = c_{y_1}(X)$ are of the form
$\{\zeta, \zeta^{q^2}, \ldots, \zeta^{q^{2m-2}}\}$, where $\zeta$ has
multiplicative order $q^{2m}-1$, so $f(X) \not\in \{f^*(X),
f^{\sigma}(X), f^{\sim}(X)\}$, as required.  This also
shows that $m$ may be even or odd. 
\end{proof}

\section{Involutions inverting regular semisimple elements}\label{sec:invol}

Having considered the regular semisimple elements whose 
characteristic polynomials are $\UU*$-irreducible, 
we now consider the general case. We remind the reader that we assume
throughout this paper that $q$ is an odd prime power. 

Let $y$ be a regular semisimple element of $G = \gu{n}{q}$,
 and suppose that $y^t = y^{-1}$ for some involution $t \in
G$. Let $g(X) := c_y(X)$. Then each of $X- 1$ and $X+1$ may divide $g(X)$ with 
multiplicity at most one, so $g(X)=g_0(X)(X-1)^{\delta_-} (X+1)^{\delta_+}$
where $\delta_-, \delta_+\in\{0,1\}$ and $g_0(X)$ is coprime to $X^2-1$
and multiplicity-free. 

\begin{defn}\label{def:a_g}
We define $\mathcal{A} \subset
\F_{q^2}[X]$ to contain one irreducible factor of each
$\UU*$-irreducible polynomial $g(X)$ in Type A.
 Similarly, we define $\mathcal{B}, \mathcal{C},
\mathcal{D} \subset \F_{q^2}[X]$ to contain one
irreducible factor of each $\UU*$-irreducible polynomial 
from Types
B, C and D, respectively. For a  $\UU*$-closed polynomial $g(X)$, we
shall write $\mathcal{A}_g$ to denote the set of irreducible factors of $g$
that lie in $\mathcal{A}$, and similarly for the other classes. 
\end{defn}

%For  each $X\in \{ A, B, C, D\}$, 
%let  $\mathcal{X}_g$ consist of one monic irreducible  factor $f(X)$ for each 
%$\langle *,\sigma\rangle$-orbit %set  $\{f(X), f^\sim(X), f^\sigma(X), f^*(X)\}$
%of 
%monic  $\UU*$-irreducible factor $g(X)$ of $g_0(X)$,  such that $f(X)$ lies in Type~X 
%of the case division in Proposition~\ref{prop:cases}. 
%In our notation, $\mathcal{X}_g \in \{\mathcal{A}_g, 
%\mathcal{B}_g, \mathcal{C}_g, \mathcal{D}_g\}$.  

Then $g_0(X)$ may be written as

\def\k{\kern-2pt}\def\kk{\kern-4pt}\def\kkk{\kern-6pt}
\begin{equation}\label{eqn:char_pol}
\left(\k\prod_{f \in \mathcal{A}_g}\kk f(X)f^\sigma(X) \kk\right) \kkk
\left(\k\prod_{f \in \mathcal{B}_g}\kk f(X)f^*(X) \kk\right) \kkk
\left(\k\prod_{f \in \mathcal{C}_g}\kk f(X)f^*(X) \kk\right) \kkk
\left(\k\prod_{f \in \mathcal{D}_g}\kk f(X)f^\sim(X)f^*(X)f^\sigma(X) \kk\right)\kk.
\end{equation}

We consider the primary decomposition of $V$ as an $\F_{q^2}\langle
y \rangle$-module, equipped with our unitary form, and combine the two
summands corresponding to $\{f(X), f^\sigma(X)\}$ in Type~A, 
the two summands corresponding to  $\{f(X), f^*(X)\}$ in Types B and C, and
the four summands corresponding to $\{f(X), f^\sim(X), f^*(X),
f^\sigma(X)\}$ in Type~D, to obtain the following uniquely determined
$y$-invariant direct sum decomposition of $V$:

\begin{equation}\label{eqn:v_decomp}
%V = \left( \perp_{f \in \mathcal{A}_g \cup \mathcal{B}_g \cup
%    \mathcal{C}_g \cup \mathcal{D}_g} V_f \right) \perp V_{\pm}
V = \bigoplus_{f\in\mathcal{A}_g} V_f \oplus \bigoplus_{f\in\mathcal{B}_g} V_f
\oplus \bigoplus_{f\in\mathcal{C}_g} V_f\oplus \bigoplus_{f\in\mathcal{D}_g} V_f\oplus V_{\pm}.
\end{equation}

such that
\begin{enumerate}
\item[(A)] for each $f \in \mathcal{A}_g$, the restriction $y_f = y
  |_{V_f} \in \gu{}{V_f}$ has characteristic polynomial
  $f(X)f^\sigma(X)$, with $f(X) = f^*(X)\ne f^\sigma(X)=f^\sim(X)$;
\item[(B)] for each $f \in \mathcal{B}_g$, the restriction $y_f = y |_{V_f} \in \gu{}{V_f}$ has characteristic polynomial
  $f(X)f^*(X)$, with $f(X) = f^\sigma(X)\ne f^*(X)=f^\sim(X)$;
\item[(C)] for each $f \in \mathcal{C}_g$, the restriction $y_f = y |_{V_f} \in \gu{}{V_f}$ has characteristic polynomial
  $f(X)f^*(X)$, with $f(X)=f^\sim(X) \neq f^*(X) = f^\sigma(X)$;
\item[(D)] for each $f \in \mathcal{D}_g$, the restriction $y_f = y |_{V_f} \in \gu{}{V_f}$ has characteristic polynomial
  $f(X)f^*(X)f^\sigma(X)f^\sim(X)$, with all four polynomials
  pairwise distinct;
\item[(E)] $\mathrm{dim} V_{\pm} \in \{0, 1,2\}$; if $\mathrm{dim}
  V_{\pm} = 1$ then $y |_{V_{\pm}}$ has characteristic polynomial
  $X-1$ or $X+1$; if $\mathrm{dim} V_{\pm} = 2$, then $V_{\pm} = V_+
  \oplus V_-$, and $y |_{V_+}$, $y |_{V_-}$, $y |_{V_\pm}$ has characteristic
  polynomial $X - 1$, $X+1$, $X^2-1$, respectively.
\end{enumerate}

\begin{lemma}\label{lem:final_centraliser}
Let $y \in \gu{n}{q} = \gu{}{V}$, where $y$ 
is regular semisimple
and conjugate in $\gu{n}{q}$ to $y^{-1}$,  with characteristic polynomial 
$g(X) = c_y(X) = g_0(X)(X-1)^{\delta_{+}} (X+1)^{\delta_{-}}$
with $\delta_{+}, \delta_{-}\in\{0,1\}$ and $g_0(X)$ as in 
\eqref{eqn:char_pol}. 
\begin{enumerate}[{\rm (i)}]
\item Each non-zero summand in \eqref{eqn:v_decomp} is a
non-degenerate unitary space, and distinct summands are pairwise orthogonal.
\item The centraliser $C_{\gu{}{V}}(y)$ has order
\[
  \left( \prod_{f \in \mathcal{A}_g \cup \mathcal{B}_g} (q^{2 \deg f} - 1) \right)
  \left( \prod_{f \in \mathcal{C}_g} (q^{\deg f} + 1)^2 \right)
  \left( \prod_{f \in \mathcal{D}_g} (q^{2 \deg f} - 1)^2 \right)
  (q+1)^{\delta_+ + \delta_-}.
\]

\item The number of involutions in $\mathcal{C}_\UU(V)$(see Definition~$\ref{def:more_invols}$) that invert $y$ is equal to
\[
  \left( \prod_{f \in \mathcal{A}_g} (q^{\deg f} + 1) \right)
  \left( \prod_{f \in \mathcal{B}_g} (q^{\deg f} - 1) \right)
  \left( \prod_{f \in \mathcal{C}_g} (q^{\deg f} + 1) \right)
  \left( \prod_{f \in \mathcal{D}_g} (q^{2 \deg f} - 1) \right)
  \varepsilon(y),
\]
where $\varepsilon(y) = 2$ if $X^2 - 1$ divides $g(X)$, 
and $\varepsilon(y) = 1$ otherwise.

\item If $n = 2m$ is even, and $g(X)$ is coprime to $X^2 - 1$, then 
the number of pairs $(t, y') \in \Delta_{\UU}(V)$ (as defined in \eqref{defRU}) such that $y'$ 
has characteristic polynomial $g(X)$ is
\[
  \frac{|\gu{2m}{q}|}{
  \left( \prod_{f \in \mathcal{A}_g} (q^{\deg f} - 1) \right)
  \left( \prod_{f \in \mathcal{B}_g} (q^{\deg f} + 1) \right)
  \left( \prod_{f \in \mathcal{C}_g} (q^{\deg f} + 1) \right)
  \left( \prod_{f \in \mathcal{D }_g} (q^{2 \deg f} - 1) \right).
  }
\]
\end{enumerate}

\end{lemma}

\begin{proof}
(i) Each space in the primary decomposition of $V$ as an
$\F_{q^2}\langle y \rangle$ module is either non-degenerate or totally
singular. It follows from \S\ref{sec:cent} that
the spaces corresponding to a $\UU*$-irreducible summand always
span a non-degenerate space. Let $U, W$ be distinct such summands, 
corresponding to $\UU*$-irreducible polynomials $h(X)$, $h'(X)$
respectively.  Let $h(X)=\sum_{i=0}^r a_iX^i \in \F_{q^2}[X]$, so that
$a_r=1$. Then $h, h'$ are coprime, 
and so $uh(y)= \sum_{i=0}^r a_i uy^i=0$, for each $u\in U$,
while $h(y)|_W$ is a bijection. Denote by
$(u,w)$ the value of the unitary form on $u\in U, w\in W$. Then
\[
0 = (u h(y),  w) = \sum_{i=0}^r a_i (u y^i, w) 
= \sum_{i=0}^r a_i(u,w y^{-i}) = (u, \sum_{i=0}^r a_i^q wy^{-i})
=(u,w h^\sim(y)y^{-r}).
\]
Since $h$ is $\UU*$-irreducible, $h^\sim(X) = h(X)$, and hence 
$w h^\sim(y)y^{-r}$ ranges over all of $W$ as $w$ does. It follows
that $W\subseteq U^\perp$.

Parts (ii) and (iii) follow from the remarks above on applying 
Lemmas~\ref{lem:gu_cent}, \ref{lem:TypeC}, \ref{lem:TypeD}  and  
Corollary~\ref{cor:ABinvols}.  
For (iv),  recall that the number of these pairs 
is equal to the number $|\gu{n}{q}|/|C_{\gu{n}{q}}(y)|$
of conjugates of $y$ times the number of $t\in\mathcal{C}_\UU(V)$ 
inverting $y$.
%If $t$ is an involution inverting $y$, then also $t':=ty$ 
%is an involution inverting $y$, and $y=tt'$. Moreover, 
By
%\cite[Lemma 3.1]{gl},
Lemma~\ref{L:ClassC_important}(i) if $g(X)=g_0(X)$ then all involutions
inverting $y$ lie in $\mathcal{C}(V)$, since $n$ is even.
%, 
%since in this case  also $-x, -x'\in\mathcal{C}_V$, 
%an
%Wall's Theorem~\ref{thm:gu_conj}, $x, x'\in \mathcal{C}_\UU(V)$. 
Thus all involutions in $\gu{n}{q}$ that invert $y$ 
lie in $\mathcal{C}_\UU(V)$, and hence Part (iv) follows from Parts (ii) and (iii).
\end{proof}

\section{Formulae for the number of \texorpdfstring{$\UU*$}{}-irreducible
  polynomials in each Type}\label{S:bds}

Recall the division of $\UU*$-irreducible polynomials into Types A
to E from Proposition~\ref{prop:cases}. We count the
number of polynomials with irreducible factors of degree $r$ in each
type.

First we present some standard counts of polynomials. 
Let $\Irr(r,\F_q)$ denote the set of monic irreducible polynomials in $\F_q[X]$
of degree~$r$. By Definition~\ref{def:nq_r},
$N(q, r)=|\Irr(r,\F_q)|$ if $r>1$ and $N(q, 1)=q-1$: we do
not count the polynomial $f(X)=X$ as the matrices we consider are
invertible. 
We define the following quantities as in~\cite[pp.\;23--26]{genfunc}:
\begin{align*}
  N^\sim(q, r)=&\textup{ number of $f\in\Irr(r,\F_{q^2})$ with $f^\sim=f$
   (over $\F_{q^2}$ not $\F_q$).}\\
  N^*(q, r)=&\textup{ number of $f\in\Irr(r,\F_q)$ with $f^*=f$.}\\
  M^*(q, r)=&\textup{ number of subsets $\{f,f^*\}$ with
    $f\in\Irr(r,\F_q)$ and $f^*\ne f$.}
\end{align*}

The M\"obius $\mu$ function is defined on $\mathbb{Z}_{>0}$, and takes values as follows:
\begin{align*}
\mu(n) & = \begin{cases}
(-1)^{k} & \mbox{if $n$ is a product of $k$ distinct primes}\\
0 & \mbox{if $n$ is not square-free.}
\end{cases}
\end{align*}

The following formulae can be found in \cite[
Lemmas 1.3.10(a), 12(a), 16(a) and (b)]{genfunc}.

\begin{thm}\label{T:Nformulae}
Let $r\geq1$ and let $q$ be an odd prime power. Then 
\begin{align*}
   N(q, r)&= \begin{cases} \frac1r\sum_{d\mid r}\mu(d)(q^{r/d}-1)\\
            \mathrlap{q-1}\phantom{\frac1r\sum_{d\mid
               r, d \, \textup{ odd}}\mu(d)(q^{r/(2d)}- 1)}&\textup{if $r=1$,}\\
            \frac1r\sum_{d\mid r}\mu(d)q^{r/d}&\textup{if $r>1$;}
             \end{cases} \\
  N^\sim(q,r)&=\begin{cases}
           \mathrlap{q+1}\phantom{\frac1r\sum_{d\mid r, \, d\textup{
                 odd}}\mu(d)(q^{r/(2d)} - 1)}&\textup{if $r=1$,}\\
           0&\textup{if $r$ is even,}\\
           N(q,r)&\textup{if $r>1$ is odd;}
           \end{cases} \\
  N^*(q,r)&=\begin{cases}
           2 &\textup{if $r=1$,}\\
           \frac1r\sum_{d\mid r, \, d\textup{ odd}}\mu(d)(q^{r/(2d)} - 1)
             &\textup{if $r$ is even,}\\
           0   &\textup{if $r>1$ is odd;}
           \end{cases}\\
  M^*(q,r)&=\begin{cases}
           \mathrlap{\frac12(q-3)}\phantom{\frac1r\sum_{d\mid
               r, d\textup{ odd}}\mu(d)(q^{r/(2d)} - 1)}&\textup{if $r=1$,}\\
           \frac12(N(q,r)-N^*(q,r))&\textup{if $r$ is even,}\\
           \frac12N(q,r)&\textup{if $r>1$ is odd.}
           \end{cases} 
\end{align*}
\end{thm}

We now prove some bounds on these quantities. 

\begin{lemma}\label{L:bds_N}
  Set $\xi= q/(q-1)$, and let $r \geq
  1$. 
  If $r\geq2$ then let $p_1<p_2<\cdots<p_t$ be the
  prime divisors of $r$. 
  \begin{enumerate}[{\rm (i)}]
\item $q^{r} - 2q^{r/2}  < q^r-\xi q^{r/p_1} <  rN(q, r) \leq q^r-1$, 
and $N(q, r) > 0.956(q^r-1)/r$ for $r\geq 5$.
\item $N(q, r+1) > N(q, r)$.
% \item $q^r-\xi q^{r/p_1}<rN(q,r)$. 
    %\item If $r = 2^b$ then
    % $rN^*(q,r)=q^{r/2}  -1$ if $q$ is odd, and $rN^*(q, r) =
    % q^{r/2}$ if $q$ is even. 
     \item Let $r$ be even. If $t = 2$ then $rN^*(q,
      r) = q^{r/2} - q^{r/(2p_2)}$, whilst if $t > 2$ then 
\[ q^{r/2} - q^{r/(2p_2)} - \xi q^{r/(2p_3)} < rN^*(q, r) <
q^{r/2} - q^{r/(2p_2)} - \frac{q-2}{q-1} q^{r/(2p_3)}.\]
  \end{enumerate}
\end{lemma}

\begin{proof}
(i) The upper bound,  and the claim for $r \geq 5$, are \cite[Lemma
2.9(ii)]{DPS}. If $r = 1$ then the result is trivial. If $r =
  p_1^a$
 is a prime power then 
$rN(q, r) = q^r - q^{r/p_1}$, and the result follows since $\xi >
1$. Hence assume that $t \geq 2$. Then 
 $
 rN(q,r) = q^r- q^{r/p_1} +
\delta$
where $\delta = \sum_{d \mid r, \, d > p_1} \mu(d) q^{r/d}$, so
\[|\delta| < \sum_{d \mid r, \, d > p_1}q^{r/d} < q^{r/p_1}\sum_{i =
  1}^\infty q^{-i} = q^{r/p_1} \left(\frac{1}{1 - \frac{1}{q}} - 1 \right) =
q^{r/p_1}(\xi - 1).\] %The result follows. 

 %\noindent  (ii)~%Suppose $r$ is even. Recall from Theorem~\ref{T:Nformulae} that
%\[rN^*(q, r) = \sum_{d\mid r, \, d\textup{ odd}}\mu(d)(q^{r/(2d)} +
%  1 - 2^{\qmodtwo})\]
%This follows immediately from Theorem~\ref{T:Nformulae}. 

\medskip

\noindent  (ii) This is \cite[Lemma 2.9(iii)]{DPS}.

\medskip

\noindent (iii) Let $r = 2^bk$ where $b \geq 1$ and $k > 1$ is odd. 
Then from Theorem~\ref{T:Nformulae} we get
\[
rN^*(q, r) %& = \sum_{d \mid r, \,
%d \textup{ odd}} \mu(d)  (q^{r/(2d)} + (-1)^q) \\
 = \sum_{d \mid k, \, d \textup{ odd}} \mu(d) q^{r/(2d)}
- \sum_{d \mid k, \, d \textup{ odd}} \mu(d) 
= \sum_{d \mid k} \mu(d)
q^{r/(2d)},\]
since $k > 1$ implies that $\sum_{d \mid k} \mu(d)
= 0$. If $t = 2$, then the result now follows, so 
 assume that $t \geq 3$. Then similarly to the proof of Part (i)
\[
q^{r/2} - q^{r/(2p_2)} - q^{r/(2p_3)}\left(1 + \sum_{i = 1}^\infty q^{-i} \right)
< rN^*(q, r) <
q^{r/2} - q^{r/(2p_2)} -  q^{r/(2p_3)} \left( 1 - \sum_{i = 1}^\infty q^{-i} \right).
\]
Thus writing $\xi=q/(q-1)$ gives
\begin{equation}\label{E:delta}
q^{r/2} - q^{r/(2p_2)} - \xi q^{r/(2p_3)}  <
rN^*(q, r) < q^{r/2} - q^{r/(2p_2)} -  (2- \xi) q^{r/(2p_3)}.\qedhere 
\end{equation}
%$$rN^*(q^2, r) = q^r - q^{r/(2p_2)}-  \beta_q q^{r/(2p_3)}$$
%where $2 - \xi <
%\beta_q < \xi$. 
% Suppose now that
%   this is not the case, i.e. $p_1=2$, and $p_2\geq3$.  Arguing as in part~(a)
%   shows $q^{r/2}-1\geq rN^*(q,r)> q^{r/2}-\xi q^{r/(2p_2)}$.
%   Suppose now that $q$ is odd. Counting the number of elements of
%   $\F_{q^{r/2}}^\times$ not in $K^\times$ where $K$ is a proper subfield
%   with $|\F_{q^{r/2}}:K|$ odd gives
%   $\frac1r\sum_{d\mid r, d\textup{ odd}}\mu(d)(q^{r/(2d)}-1)$ by the IEP.
%   This equals $rN^*(q,r)$ by~\eqref{E12}.  Since 
%   $q^{r/(2p_i)}-1\geq q(q^{r/(2p_{i+1})}-1)$ and $1<\xi\leq2$,
%   arguing as in part~(a) shows
%   \[
%     q^{r/2}-1\geq rN^*(q,r)\geq (q^{r/2}-1)-(q^{r/(2p_2)}-1)\sum_{i=0}^{m-1}q^{-i}
%     \geq q^{r/2}-\xi q^{r/(2p_2)}+1.
%   \]
%   This proves part~(b) as $\qmodtwo$ equals $1$ for $q$ odd, and 0 for
%   $q$ even.
\end{proof}

\begin{lemma}\label{lem:nqr} 
Let $r$ be a positive integer. Then
\[
\frac{1}{2}\left(N^*(q^2, 2r)+M^*(q^2, r) - N^\sim(q, r)\right)
= \begin{cases}
	N(q, 2r)- 3/2 &\mbox{if $r=1$},\\
	N(q, 2r) &\mbox{if $r>1$}.
	\end{cases}
	\]
\end{lemma}

\begin{proof}
Suppose first that $r=1$. By \cite[Corollary 1.3.16]{genfunc},
\[N^*(q^2,2)+M^*(q^2,1)= \frac{q^2-1}{2} + \frac{q^2-3}{2}=q^2-2.\]
Since  $N^\sim(q, 1)=q+1$ and $N(q,2)=(q^2-q)/2$, 
the $r=1$ case follows.  Now suppose that $r>1$. 
Then, by  \cite[Lemma 5.1]{gl}, 
$N^*(q^2, 2r) + M^*(q^2, r) 
= N(q^2, r),$
and it follows from \cite[Corollary 1.3.13]{genfunc} and its proof
that
$(N(q^2, r)- N^\sim(q, r))/2= N(q, 2r)$ 
(note this also holds for $r$ even since in that case $N^\sim(q,r)=0$).
\end{proof}

\begin{defn}\label{defn:poly_counts}
We define $A(q, r)$ to be the number of $\UU*$-irreducible
polynomials in Type A of degree $2r$ over $\F_{q^2}$ (so that the irreducible
factors have degree $r$). Similarly, we define $B(q, r)$ and $C(q, r)$
to be the number of $\UU*$-irreducible polynomials in Types B and C
of degree $2r$, and $D(q, r)$ to be the number of
$\UU*$-irreducible polynomials in Type D of degree $4r$. Recall Definition~\ref{def:D_minus}: it is immediate that $|\cD_{4r}|
= D(q, r)$.
\end{defn}

\begin{lemma}\label{lem:ABCD} 
Let $r\geq1$. Then 
\begin{align*}
  A(q,r)&=\begin{cases}
            \mathrlap{\frac{1}{2}N^*(q^2,1)-1 = 0}\hphantom{\frac{1}{2}(M^*(q^2,1)-N^\sim(q,1))+\frac{3}{2} = \frac{1}{4}(q-1)^2}  & \textup{if $r = 1$,}\\
            \frac{1}{2}N^*(q^2,r) &\textup{if $r > 1$;}
          \end{cases}\\
  B(q,r)&=\begin{cases}
        \mathrlap{\frac{1}{2}N^\sim(q,1)-2 = \frac{1}{2}(q-3)}\hphantom{\frac{1}{2}(M^*(q^2,1)-N^\sim(q,1))+\frac{3}{2} = \frac{1}{4}(q-1)^2}  &\textup{if $r = 1$,}\\
             \frac{1}{2}N^\sim(q,r)&\textup{if $r > 1$;}
          \end{cases}\\
  C(q,r)&=\begin{cases}
            %\mathrlap{\frac{1}{2}(q+1-2^\qmodtwo))}\hphantom{\frac{1}{4}((q-1)^2+2^\qmodtwo-2)}&\textup{if
             % $r=1$,}\\
            \mathrlap{\frac{1}{2}N^\sim(q,1)-1 = \frac{1}{2}(q-1)}\hphantom{\frac{1}{2}(M^*(q^2,1)-N^\sim(q,1))+\frac{3}{2} = \frac{1}{4}(q-1)^2} & \textup{if $r = 1$,}\\
            \frac{1}{2}N^\sim(q,r)&\textup{if $r > 1$;}
          \end{cases}\\
  D(q,r)&=\begin{cases}
            %\frac{1}{4}((q-1)^2+2^\qmodtwo-2)&\textup{if $r=1$,}\\
            \frac{1}{2}(M^*(q^2,1)-N^\sim(q,1))+\frac{3}{2} = \frac{1}{4}(q-1)^2 &
            \textup{if $r = 1$,}\\
            \frac{1}{2}(M^*(q^2,r)-N^\sim(q,r))&\textup{if $r > 1$.}
          \end{cases}
\end{align*}
\end{lemma}

\begin{proof}
In each type, the equivalence of the two statements for $r = 1$
follows immediately from Theorem~\ref{T:Nformulae}.
In each of the following types, we first count the number of
polynomials $f \in \Irr(r, \F_{q^2})$ satisfying the type conditions,
and then deduce the number of $\UU*$-irreducible polynomials of the
relevant degree. 

\medskip

\noindent {\sc Type~A.} By Proposition~\ref{prop:cases}, in this type
$f=f^*\ne f^\sigma$, and if $f$ exists then
 $r$ is even. Therefore 
%We remark that if $\zeta$ is a root of $f$, then $\zeta^{-1}$ is a different
%root of $f$. If not, then $\zeta=\zeta^{-1}$ implies $\zeta^2=1$.
%Thus $\zeta\in\F_q$, contrary to our assumption that $f\ne f^\sigma$.
%Hence the roots of $f$ come in pairs $\{\zeta,\zeta^{-1}\}$ and so~$r$
%must be even. 
$A(q, 1) = 0$, and if $r>1$ and $r$ is odd, then $A(q,r)= 
N^*(q^2, r)/2=0$. Suppose that $r$ is even. By Lemma~\ref{lem:irred},
$r$ even implies that $f\ne f^\sigma$, and hence in this case 
$A(q,r)$ is the number of pairs $\{f, f^\sigma\} \subseteq \Irr(r,\F_{q^2})$
satisfying $f=f^*$, namely $A(q,r)=N^*(q^2,r)/2$.

\medskip

\noindent {\sc Type~B.} By Proposition~\ref{prop:cases}, 
in this type
$f=f^\sigma\ne f^*$, and if $f$ exists then $r$ is odd.
Thus if $r$ is even then  $B(q, r) = 0 = N^\sim(q, r)/2$. 
Suppose that $r=1$ so $f(X)=X-\zeta$, for some $\zeta\in\F_{q^2}$. 
%Since the
%divisors of the characteristic polynomials we are considering are 
Since $f \neq f^*$, the polynomial $f$ is not 
$X\pm 1$ or $X$, and so the root $\zeta \not\in\{0, \pm 1\}$.
Moreover, since $f=f^\sigma$, we have
$\zeta=\zeta^\sigma$ so $\zeta\in\F_q\setminus\{0, \pm 1\}$. 
Note that, for each such $\zeta$, the polynomial $f\ne f^*$
since $\zeta\ne \zeta^{-1}$.
Thus there are $q-3$ possibilities for $\zeta$,  so the
number of pairs $\{f, f^*\}$ is $B(q,1)=(q-3)/2$.
Suppose now that $r>1$ and $r$ is odd.  
Then $f \ne f^*$ by Lemma~\ref{lem:irred}, and hence in this case
$B(q,r)$  is the number of pairs $\{f, f^*\} \subset\Irr(r,\F_{q^2})$
satisfying $f=f^\sigma$, namely $B(q,r)=N(q,r)/2 = N^{\sim}(q, r)/2$. 
%Finally, 
%we note that if $r > 1$ and $r$ is odd, then by
%\cite[Corollary 1.3.13]{genfunc}, $N^\sim(q, r) = N(q, r)$, as
%required. 

\medskip

\noindent {\sc Type~C.} By Proposition~\ref{prop:cases},  in this type
%the polynomials $f$ under consideration satisfy 
$f \ne f^* = f^\sigma$ (so $f = f^\sim$), 
%where $f\in\Irr(n,\F_{q^2})$, and 
and if $f$ exists then $r$ is odd.
Thus if $r$ is even then  $C(q, r) = 0 = N^\sim(q, r)/2$. 
Suppose that $r=1$ so $f(X)=X-\zeta$, for some $\zeta\in\F_{q^2}$.
Since $f^* = f^\sigma \ne f$, the root satisfies
$\zeta^{-1} = \zeta^q \ne \zeta$, so $\zeta^{q+1}=1$ and $\zeta^2\ne 1$,
whence $\zeta\ne\pm1$. 
Thus there are $q-1$ possibilities for $\zeta$, and the
number of pairs $\{f, f^*\}$ is $C(q,r) = (q-1)/2$. 
Suppose now that $r>1$ and $r$ is odd. 
%The formula for $D(q,r)$ follows from the equation
%$A(q,r)+B(q,r)+C(q,r)+D(q,r)=N(q^2,r)$, and the observation that
%$N^\sim(q,r)=N(q,r)-N^*(q^2,r)$ for $1<r$ odd.
%
Then $f \ne f^*$ by Lemma~\ref{lem:irred}, and hence in this case
$C(q,r)$  is the number of pairs $\{f, f^*\} \subset\Irr(r,\F_{q^2})$
satisfying $f=f^\sim$, namely $C(q,r)=N^\sim(q,r)/2$. 

\medskip

\noindent {\sc Type~D.} 
In this type the irreducible polynomials  %the polynomials $f\in\Irr(r,\F_{q^2})$ 
%under consideration are such that the four polynomials
$f, f^\sigma, f^*, f^\sim$ are pairwise distinct.
First suppose that $r=1$, so $f(X)=X-\zeta$, for some $\zeta\in\F_{q^2}$. 
The conditions $f\ne f^\sim$ and $f\ne f^*$ are equivalent to
$\zeta^{q+1}\ne1$ and $\zeta^{q-1}\ne 1$, respectively.
These two conditions together imply that 
$f, f^\sigma, f^*, f^\sim$ are pairwise distinct, and hence 
%the 
%number of possible $\zeta$ is 
\[
D(q, 1) = \frac{1}{4}((q^2-1) - (q+1) - (q-1) + 2) = \frac{1}{4}(q-1)^2.
\]
%and $D(q,1)$ is a quarter of this quantity. 
Suppose now that $r>1$. We will
prove that $B(q,r)+C(q,r) + 2D(q,r)= M^*(q^2,r)$. Solving for $D(q,r)$ then gives
the desired result. The number of pairs $\{f, f^*\}\subset \Irr(r,\F_{q^2})$ 
satisfying $f\ne f^*$ is, by definition, $M^*(q^2,r)$. We enumerate these pairs 
by a different argument. We showed under `Type~B' and `Type~C' above that the 
numbers of such pairs for which $f^\sigma =f$, or $f^\sigma = f^*$, is 
$B(q,r)$ or $C(q,r)$, respectively. For the remaining pairs the polynomials
$f, f^*, f^\sigma$ are pairwise distinct, giving a set   
$\{f, f^\sigma, f^*, f^\sim\}$ of size four: there are $D(q,r)$ such subsets
and each corresponds to two pairs, namely $\{f, f^*\}$ and
$\{f^\sigma, f^\sim\}$.  
\end{proof}

The following is an immediate corollary of Theorem~\ref{T:Nformulae}
and Lemma~\ref{lem:ABCD}. 

\begin{cor}\label{C:D_cases} The following identities hold. %
\begin{align*}
 D(q, r) &=\begin{cases}
            %\frac{1}{4}(q-1)^2&\textup{if $r=1$,}\\
            \frac{1}{4}(N(q^2, r) - N^*(q^2, r))&\textup{if $r$ is
              even,}\\
            \frac{1}{4}N(q^2, r) - \frac{1}{2}N(q, r) & \textup{if
              $r>1$ is odd.}
          \end{cases}
\end{align*}
\end{cor}

We now prove bounds on these polynomial counts that will
be useful later. 

\begin{lemma}\label{L:bds_D}
  \begin{enumerate}[{\rm (i)}]
\item    If $r = 2^b$ for $b \geq 0$ then  $D(q, r) =  (q^r-1)^2/(4r)$. 
  \item If $r = 3$, then $4rD(q, r) = q^6 - 2q^3 - q^2 + 2q < q^{2r}-q^r+\frac{q+1}{q} q^{r/3}$. 
For all other $r$, 
\[
q^{2r}- 2q^r-\frac{1}{q^2-1} q^{r}<
      4rD(q,r) <  q^{2r}-q^r+\frac{q+1}{q} q^{r/3} <  q^{2r} - 1.
\]
\item $4rD(q, r) = (q^{2r}-1) - \eta(q, r)(q^r-1)$, 
where $0 < 1 - 2q^{-2r/3} < \eta(q, r) < 2.2$. 
   \end{enumerate}
\end{lemma}

\begin{proof} If $r > 1$, then
  let the prime divisors of $r$ be $p_1 < p_2 < \ldots <
  p_t$. 

\noindent (i)
The result for $r = 1$ is immediate from
Lemma~\ref{lem:ABCD}. Otherwise, 
by Corollary~\ref{C:D_cases}, 
$D(q,r) = (N(q^2,r) - N^{*}(q^2,r))/4$. Then using
Theorem~\ref{T:Nformulae} and the fact that $r = 2^b > 1$, we deduce that
\[D(q, r) = \frac{1}{4}\left(\frac{q^{2r}-q^r}{r}-
    \frac{q^{r} - 1}{r} \right)
	=  \frac{(q^r-1)^2}{4r},\] as required. 

\medskip

\noindent (ii)~By Part (i) we may assume that $r$ is not a $2$-power, and in particular that $r\geq3$.
Before commencing the main part of the proof, notice first that
if $q^{2r} - q^r + \frac{q+1}{q}q^{r/3} \geq q^{2r} - 1$ then
$q^r-1 \leq  \frac{q+1}{q}q^{r/3} < 2 q^{r/3}$, so $(q^{r} - 1)^3
< 8q^r$, which is impossible, since $q \geq 3$ and $r \geq 3$. Thus the last
inequality holds.

Suppose first that $r$ is even, and hence is divisible by at least two primes.  
Then by Corollary~\ref{C:D_cases},
$D(q,r)=\left(N(q^2,r)-N^*(q^2,r)\right)/4.$ 
   We deduce from Lemma~\ref{L:bds_N}(i)(iii) that
   \begin{align*}
   4rD(q,r)  &>  (q^{2r}- q^2(q^2-1)^{-1} q^r) - (q^{r}
   - q^{r/p_2})\\
   &>  q^{2r} - (2q^2-1)(q^2-1)^{-1}q^r =  q^{2r} - 2q^r - \frac{1}{q^2-1}q^{r},
   \end{align*}
%   $$
%   4rD(q,r)  >  (q^{2r}- q^2(q^2-1)^{-1} q^r) - (q^{r}
%   - q^{r/p_2}) >  q^{2r} - (2q^2-1)(q^2-1)^{-1}q^r
%  =  q^{2r} - 2q^r - (q^2-1)^{-1}q^{r},
%   $$
     and (using the fact that $r/p_3\leq r/p_2-2$ if $p_3$ exists)
   $$
   4rD(q, r) < (q^{2r} - 1) - (q^{r} -
   (1 + q^{-1}))q^{r/p_2} <  q^{2r} - q^r + \frac{q+1}{q}q^{r/3}.
   $$
   
   Suppose now that $r > 1$ is odd, so that
   $D(q,r)=N(q^2,r)/4- N(q,r)/2$, by
   Corollary~\ref{C:D_cases}. 
 If $r$ is an odd prime then $rN(q^\varepsilon,r) = q^{\varepsilon r}-q^{\varepsilon}$ and so 
 $4rD(q,r) = (q^{2r}-q^2)-2(q^r-q) = q^{2r}-2q^r -q^2 +2q$. This is
 less than the required upper bound for all odd primes $r$, is greater
 than  $q^{2r}-2q^r -q^r/(q^2-1)$ for $r>3$, and is precisely the
 stated value when $r = 3$. If $r$ is composite (so $r\geq 9$), then 
 Lemma~\ref{L:bds_N}(i)  gives  
   \[
     4rD(q,r) >  (q^{2r}- \frac{q^2}{q^2-1} q^{2r/p_1})- 2(q^r-1)
       > q^{2r}-2q^{r}- \frac{q^2}{q^2-1} q^{2r/3} 
   \]
   and since $2r/3+2<r$ this is greater than $q^{2r}-2q^{r}- \frac{1}{q^2-1} q^{r}$.
   Also (setting $\xi'=q/(q+1)$)
    \[ 
    4rD(q, r) < (q^{2r} - 1) - 2(q^r -\xi' q^{r/p_1}) < q^{2r} - 2q^r
      + 2\xi' q^{r/3}  < q^{2r} - q^r + \frac{q+1}{q}q^{r/3}.
   \]

\medskip

\noindent (iii) Set $4rD(q, r) = (q^{2r}-1) - \eta(q,r) (q^r-1)$, and let $\eta =
\eta(q, r)$. When $r$ is a power of $2$, the result follows easily
from Part (i) (in fact here $\eta=2$), so assume that $r$ is not a power of $2$. The
upper bound in Part (ii) yields that, for all such $r$, 
$$
- 1 -\eta (q^r-1) \leq -q^r + \frac{q+1}{q} q^{r/3},\quad \mbox{or equivalently,}\quad \eta \geq 1 -  \frac{q+1}{q} \cdot \frac{ q^{r/3}}{q^r-1}.
$$
We must show that 
$$\frac{q+1}{q} \cdot \frac{q^{r/3}}{q^r-1} < \frac{2}{q^{2r/3}}.$$
For all $r \geq 3$ and $q \geq 3$, it is clear that $q^{r-1} <
q^r - 2$, and so $q^{r}  + q^{r-1} < 2q^r - 2$. Hence $(q + 1)q^r =
q^{r+1} + q^r = q(q^r + q^{r-1}) < q(2q^r - 2) = 2q(q^r-1)$. Thus
$((q+1)/q) \cdot q^r/(q^r - 1) < 2,$ from which the claimed
lower bound on $\eta$ follows.

For $r$ an odd prime 
\[4rD(q,r) =q^{2r}-2q^r -q^2 +2q =
q^{2r}-1 - (2q^r + q^2 -2q - 1)\]
which gives  $\eta(q,r) = 2 + (q-1)^2/(q^r-1)$ so $2<\eta(q,r)
\leq 2 + 2/13 < 2.2$.
Otherwise,  Part (ii) yields 
\[
\eta(q^r-1) = q^{2r}-1-4r D(q,r)\leq -1 + 2q^r + \frac{1}{q^2-1}q^r = 2(q^r-1) + \frac{q^r-1+q^2}{q^2-1}
\]
so using $r > 2$ and $q \geq 3$ gives
\begin{align*}
\eta & \leq 2 + \frac{1}{q^2-1} + \frac{q^2}{(q^r-1)(q^2-1)} =
2 + \frac{1}{q^2-1} + \frac{1}{q^r-1} + \frac{1}{(q^r-1)(q^2-1)} \\
& < 2.1683.\qedhere
%& \leq 2 + 1/8 + 1/26 + 1/(8\times 26) < 2.1683. 
\end{align*}
%
%
%
%The left hand inequality of \eqref{eq:eta} is equivalent to 
%\[
%\eta (q^r-1)\leq -1 + 2 q^r  +  \frac{q^2}{q^2-1} q^{r/3}
%\leq 2(q^r-1) + q^{r/3} + 1
%\]
%so 
%\[
%\eta \leq 2 +  \frac{q^{r/3} + 1}{q^r-1}
%\]%
%
%We claim that 
%$(q^{r/3} + 1)/(q^r-1) < 3/5$.
%from which the result follows. Assume otherwise, then $5q^{r/3} + 5 \geq
%3q^r-3$, and so $5q^{r/3} + 8 \geq 3q^r$. Since $q \geq 3$ and $r \geq
%3$, this is absurd, so the result
%follows. 
\end{proof}

\begin{lemma}\label{L:DquotBds}
Let $q\geq5$ and $r \geq 1$. Then
\[
{\rm (i)}
\displaystyle\frac{D(q,r)}{q^{2r}-1}\geq\frac{D(3,r)}{3^{2r}-1}\mbox{;}
\quad 
{\rm (ii)} \displaystyle\frac{N^*(q^2, r)}{q^r-1} \geq
  \frac{N^*(3^2, r)}{3^r-1} \mbox{ for $r > 1$;} \quad
{\rm (iii)} \displaystyle\frac{N^\sim(q, r)}{q^r + 1} \geq
  \frac{N^\sim(3, r)}{3^r+1}. 
\]
\end{lemma}

\begin{proof}
%(i)    For $r$ a power of $2$, it is immediate from
%Lemma~\ref{L:bds_D}(i) that
%    $\frac{D(q,r)}{q^{2r}-1}=\frac{q^r-1}{4r(q^r+1)}$, and so  is increasing with $q$.
  (i)    Suppose $r=2^b$ is a power of $2$. Then Lemma~\ref{L:bds_D}(i)
  implies that
  \[\frac{D(q,r)}{q^{2r}-1}=\frac{q^r-1}{4r(q^r+1)}.\]
  This is an increasing function of $q$. Hence Part~(i) holds for such $r$.
  A straightforward calculation
    shows the result when $r = 3$, so assume that
$r\geq 5$. Using Lemma~\ref{L:bds_D}(ii), 
$$
4r \frac{D(q, r)}{q^{2r}-1} \geq 
\frac{q^{2r} - 2q^r - \frac{1}{q^2-1}q^{r}}{q^{2r}-1} = 1 -
\frac{2q^r + \frac{1}{q^2-1}q^{r} - 1}{q^{2r}-1} \geq 1 -
\frac{\frac{49}{24}q^r}{q^{2r} - 1} \geq 1 - \frac{9q^r}{4(q^{2r}-1)}.
$$
Using Lemma~\ref{L:bds_D}(ii) again gives
\[
\displaystyle{4r \frac{D(3, r)}{3^{2r} - 1}}
\leq \displaystyle{\frac{3^{2r} - 3^{r} +
  \frac{4}{3}3^{r/3}}{3^{2r} - 1}} = 1 - \displaystyle{\frac{3^r -\frac{4}{3}3^{r/3}
  - 1}{3^{2r}-1}} 
% \leq 1 - \displaystyle{\frac{3^r - 4\cdot 3^{r/3 - 1} -
%  1}{3^{2r}}} 
\leq 1 -
\frac{1}{2 \cdot 3^r}. \]

Since $q \geq 5$ and $r \geq 5$, one may verify that $2q^r \geq 9
  \cdot 3^r + 2$. Hence $2q^{2r} \geq 9 \cdot 3^r q^r + 2$, and so
  $2(q^{2r}-1) \geq 9 \cdot 3^r q^r$. Hence $1/(2 \cdot 3^r)
  \geq 9 q^r/(4(q^{2r}-1))$, and so the result follows.

\medskip
\noindent (ii) 
The result is immediate  from 
Theorem~\ref{T:Nformulae} if $r$ is odd, or if $r$ is a
power of $2$, so let $r = 2^b\cdot k$, where $k$ is odd.
If $k = p^a$ is a prime power, then $rN^*(q^2, r) = q^r - q^{r/p}$,
and the result can be verified by direct calculation. 
So let $p_2 < p_3$ be the two smallest odd primes dividing  $r$, 
then Lemma~\ref{L:bds_N}(iii) states that
$rN^*(q^2, r) \geq q^r - q^{r/p_2} -
  \frac{5}{4}q^{r/p_3}$ and 
$rN^*(3^2, r) \leq 3^r - 3^{r/p_2} -
  \frac{1}{2}3^{r/p_3}.$
Assume, by way of contradiction, that 
$$(3^r-1)(q^r - q^{r/p_2} - \frac{5}{4}q^{r/p_3})< (q^r-1)(3^r -
  3^{r/p_2} - \frac{1}{2}3^{r/p_3}).$$
Then
$$q^r(3^{r/p_2} + \frac{1}{2} 3^{r/p_3} - 1) - 3^r(q^{r/p_2} +
\frac{5}{4}q^{r/p_3} - 1) + (q^{r/p_2} - 3^{r/p_2}) +
(\frac{5}{4}q^{r/p_3} - \frac{1}{2}3^{r/p_3} ) < 0$$
and so in particular
$q^r(3^{r/p_2} + \frac{1}{2} 3^{r/p_3} - 1) - 3^r(q^{r/p_2} +
\frac{5}{4}q^{r/p_3} - 1) < 0.$
Dividing by $(3q)^{r/p_2}$ %, this implies that
%$$q^{r - r/p_2} - 2 \cdot 3^{r - r/p_2} < 0,$$
yields a contradiction. Hence the result holds for all $r$ and $q$. 

\medskip
\noindent (iii) The arguments here are similar to the previous
two parts. By
Theorem~\ref{T:Nformulae}, the result is immediate if
$r$ is even or if $r = 1$. Assume that $r > 1$ is odd, so that
$N^\sim(q, r) = N(q, r)$.

Let $p$ be a prime divisor of $r$. We digress to prove
  \begin{equation}\label{E:red}
  \frac{q^r-q^{r/p}}{q^r+1}\geq\frac{3^r-3^{r/p}}{3^r+1}\  \textup{or equivalently }\ 
  \frac{(q^r+1)-(q^{r/p}+1)}{q^r+1}\geq\frac{(3^r+1)-(3^{r/p}+1)}{3^r+1}.
  \end{equation}
It suffices to prove $(q^{r/p} + 1)(3^r+1) \leq (3^{r/p}+1)(q^r+1)$.
This is true if $q^{r/p}3^r\leq q^r3^{r/p}$ and $q^{r/p}+3^r\leq q^r+3^{r/p}$.
The first inequality is true as $3^{r(1-1/p)}\leq q^{r(1-1/p)}$. The
second inequality is $3^r-3^{r/p}\leq q^r-q^{r/p}$ or $x_0^p-x_0\leq x^p-x$
where $x_0=3^{r/p}\leq q^{r/p}=x$. However, the function $x^p-x$ is increasing
for $x>1$, so the second inequality holds. This proves~\eqref{E:red}.

If $r=p^a$ is a prime power, then $rN(q, r) = q^r -q^{r/p}$. Thus Part~(iii)
is true by~\eqref{E:red}.
Suppose now that $r$ has distinct prime divisors $p_1<p_2$.
Thus $p_1\geq3$, $p_2\geq5$ and so $r\geq15$.
Then as in \eqref{E:delta} we see that
$rN(q, r) = q^r  - q^{r/p_1} - \delta q^{r/p_2}$, where
$2 - q/(q-1) < \delta < q/(q-1)$. Hence
\[
  \frac{rN(q, r)}{q^r+1} > \frac{q^r-q^{r/p_1}}{q^r+1} - \frac{\frac{3}{2}q^{r/p_2}}{q^r+1}\quad\textup{and}\quad
  \frac{3^r - 3^{r/p_1}}{3^r+1} - \frac{\frac{1}{2} 3^{r/p_2}}{3^r+1}>\frac{rN(3, r)}{3^r+1}.
\]
Using~\eqref{E:red}, it suffices to show that 
\[
\frac{3^{r/p_2}}{3^r+1} > \frac{3q^{r/p_2}}{q^r+1}\qquad\textup{or equivalently}\qquad
(q^r+1)3^{r/p_2}>3q^{r/p_2}(3^r+1).
\] 
As $q^r+1>q^r$ and $2\cdot 3^r>3^r+1$, 
it suffices to show that $3^{r/p_2}q^r > (6\cdot 3^r)q^{r/p_2}$.
This is true because $(q/3)^{r(1-1/p_2)}\geq (5/3)^{15(1-1/5)}=(5/3)^{12}>6$.
This concludes the proof.
\end{proof}

\section{The generating function \texorpdfstring{$R_{\UU}(q, u)$}{}}\label{sec:genfnrU}

%For a vector space $V=\F_{q^2}^n$, we will write $\mathbf{R}_\UU(2m,q)$ for the
%for the set $\mathbf{R}_\UU(V)$ defined in \eqref{defRU}. 

In this section, we define a key generating function,  analyse its convergence and
bound its coefficients. We continue to assume throughout that $q$ is an odd prime
power. 

\subsection{Introducing  \texorpdfstring{$R_{\UU}(q, u)$}{}}
Recall the definition of $\Delta_{\UU}(2n, q)$ from \eqref{defRU}.

\begin{defn}\label{def:ru}
We shall consider the `weighted proportions'
\[
r_\UU(2n,q):= \frac{\left\vert \Delta_\UU(2n, q) \right\vert}{\left\vert
\gu{2n}{q}\right\vert} \quad \mbox{for}\ n\geq1,\ \mbox{letting}\
r_\UU(0,q)=1, 
\]
and define the generating function
$R_\UU(q, u)=\sum_{n=0}^\infty \r_\UU(2n,q)u^{n}$. 
\end{defn}

Recall Types A to D from Proposition~\ref{prop:cases}, and
that we use these types to  describe $\UU*$-irreducible
polynomials.  Recall also Definition~\ref{def:a_g}.

Let $\mathcal{U}_n$ denote the set of all monic $\UU*$-closed
polynomials $g(X)$ of degree $2n$ such that $\mathrm{gcd}(g, X^2-1) =
1$. 
It follows from Lemma~\ref{lem:final_centraliser}(iv) that, 
for $n\geq1$, $\r_\UU(2n,q)$ is the sum over all $g(X)=c_y(X) \in \mathcal{U}_n$ 
%of degree $2n$ that coprime to $X^2-1$, 
%and having a factorisation 
%as in \eqref{eqn:char_pol}, 
of the expression
\[
\frac{1}{
\left( \prod_{f \in \mathcal{A}_g} (q^{\deg f} - 1) \right)
\left( \prod_{f \in \mathcal{B}_g} (q^{\deg f} + 1) \right)
\left( \prod_{f \in \mathcal{C}_g} (q^{\deg f} + 1) \right)
\left( \prod_{f \in \mathcal{D}_g} (q^{2 \deg f} - 1) \right)
}.
\]
Thus  the generating function $R_\UU(q, u)$ can be expressed as
\[ 
\sum_{n=0}^\infty\left(\sum_{
\mbox{\footnotesize $g \in \mathcal{U}_n$}}
%\begin{array}{c}f=f^*\\ \deg f=2n\end{array}}
\frac{u^{n}}{ 
\left( \prod_{f \in \mathcal{A}_g} (q^{\deg f} - 1) \right)
\left( \prod_{f \in \mathcal{B}_g \cup \mathcal{C}_g} (q^{\deg f} + 1) \right)
\left( \prod_{f \in \mathcal{D}_g} (q^{2 \deg f} - 1) \right)
%\left( \prod_{f \in \mathcal{D}} (q^{\deg f} + 1) \right)
} \right).
\]
%where in the inner summation we assume that $g(X)$ is coprime to $X^2-1$.

\begin{thm}\label{thm:r_prod}
$R_\UU(q, u)$ is equal as a complex function to $S_0(q, u) S(q, u)$, where
$S_0(q, u)$ equals $ \left(1+\frac{u}{q-1}\right)^{-1} \left(1+\frac{u}{q+1}\right)^{-3}$
and 
$S(q, u)$ is the infinite product 
\[
  (1+\frac{u^2}{q^2-1})^{\kern-1pt\frac{3}{2}} \kern-3pt \prod_{r \geq 1} \kern-3pt\left( 1\kern-1pt+\kern-1pt \frac{u^{r}}{q^{r} - 1} \right)^{\kern-3pt\frac{1}{2}N^*(q^2, r)}\kern-4pt
\prod_{r \geq 1} \kern-4pt\left( 1\kern-1pt+\kern-1pt \frac{u^{r}}{q^{r} + 1} \right)^{\kern-3ptN^\sim(q, r)} 
\kern-4pt\prod_{r\geq 1}\kern-3pt \left( 1\kern-1pt+\kern-1pt \frac{u^{2r}}{q^{2r} - 1}
  \right)^{\kern-3pt\frac{1}{2}M^*(q^2, r) - \frac{1}{2}N^\sim(q, r)}\kern-3pt. 
\] Furthermore, $R_\UU(q, u)$ is absolutely and uniformly
convergent on the open disc $|u| < 1$. 
\end{thm}

\begin{proof}
Let 
\[
  R'_{\UU}(q, u) = \prod_{f \in \mathcal{A}} \left( 1+ \frac{u^{\deg f}}{q^{\deg f} - 1} \right)
  \prod_{f \in \mathcal{B} \cup \mathcal{C}} \left( 1+ \frac{u^{\deg f}}{q^{\deg f} + 1} \right)
  \prod_{f \in \mathcal{D}} \left( 1+ \frac{u^{2\deg f}}{q^{2\deg f} - 1} \right).
\]
Then computing the coefficient of
$u^n$ for each $n$ shows that $R'_\UU(q, u)$ is equal to $R_\UU(q, u)$.

The contribution of each term of this infinite product 
depends only on the degree of the corresponding 
polynomial $f$, and so $R'_\UU(q, u)$ is equal, as a
complex function,  to
\begin{equation}\label{eq:prod2} 
R''_\UU(q, u)=\kern-4pt\prod_{r \mbox{\footnotesize{ even}}} \left( 1+ \frac{u^{r}}{q^{r} - 1} \right)^{A(q,r)}\kern-5pt
\prod_{r \mbox{\footnotesize{ odd}}} \left( 1+ \frac{u^{r}}{q^{r} + 1} \right)^{B(q,r) + C(q,r)}\kern-4pt
\prod_{\mbox{\footnotesize{all }} r} \left( 1+ \frac{u^{2r}}{q^{2r} - 1} \right)^{D(q,r)}\kern-12pt.
\end{equation}

Substituting the values from Lemma~\ref{lem:ABCD} into the above
expression for $R''_\UU(q, u)$, and noting from
Theorem~\ref{T:Nformulae} that $N^*(q, r) = 0$ for $r > 1$ odd,
whilst $N^\sim(q, r) = 0$ for $r$ even,  shows that $R''_\UU(q, u)
= S_0(q, u)S(q, u)$. 

%We showed in Lemma~\ref{lem:prod} and
% Theorem~\ref{thm:r_prod} that inside the open disc of radius
%$1 - 1/q$ the series $R_\UU(q, u)$ is absolutely and uniformly
%convergent, and is equal to an infinite product $S_0(q,
%u)S(q, u)$. 

We now consider convergence of $S(q, u)$. 
By \cite[Corollary 1.3.2]{genfunc},  each of
$$
\mbox{the product}\quad \prod_{r \geq 1}\left(1 +
\frac{u^r}{q^r-1}\right)^{\frac{1}{2}N^*(q^2, r)}\quad\mbox{and the sum}\quad 
\sum_{r \geq 1}\frac{1}{2}N^*(q^2, r)\frac{|u^r|}{q^r-1}
$$ 
is absolutely and uniformly convergent if and only if the other has these properties. 
Now, by \cite[Lemma 1.3.16(a)]{genfunc}, $N^*(q^2, r) = r^{-1}
q^r + O(q^{r/3})$ when $r$ is even, and is equal to $0$ when $r \geq
3$ is odd, so the displayed sum is absolutely and uniformly 
convergent for $|u| < 1$.
Similarly, for the product
$$
\prod_{r \geq 1} \left(1 + \frac{u^r}{q^r + 1}\right)^{N^\sim(q,  r)}\quad\mbox{we consider the sum}\quad 
\sum_{r \geq 1} N^\sim(q,r)\frac{|u^r|}{q^r+1}.
$$ 
By \cite[Lemma 1.3.12(a)]{genfunc}, $N^\sim(q, r) =
r^{-1}q^r - O(q^{r/3})$ when $r$ is odd, and is equal to $0$ when $r$
is even,  so as before this term is absolutely and uniformly
convergent for $|u| < 1$.
For $\prod_{r \geq 1}(1 + \frac{u^{2r}}{q^{2r} -
  1})^{\frac{1}{2}M^*(q^2, r) - \frac{1}{2}N^\sim(q, r)}$, we use the same
arguments: by Lemma~\ref{lem:ABCD} this exponent is equal to $D(q, r)$ for $r>
1$, and then Lemma~\ref{L:bds_D}(iii) 
gives bounds on $D(q, r)$ that
guarantee absolute and uniform convergence for $|u|<1$.
\end{proof}
%%%%%%%%%%%%%%%%

%\subsection{The limit of $r_{\UU}(2m, q)$}
%Next we prove that $\lim_{n \rightarrow \infty} r_\UU(2n, q)$ exists. 

\begin{thm}\label{thm:r_exp}
The limit $\lim_{n\rightarrow\infty}r_\UU(2n,q)$ exists and is equal
to 
\[\frac{1-\frac{1}{q}}{(1+\frac{1}{q+1})} \prod_{\mbox{\rm{\footnotesize{$r$ odd}}}}
  \left( 1- \frac{2}{q^r(q^{r} + 1)} \right)^{N(q, r)}.\]
\end{thm}

\begin{proof}
Consider the expression $R_\UU(q, u) = S_0(q, u)S(q, u)$ from
Theorem~\ref{thm:r_prod}. 
We use the fact that $N^*(q^2, r) = 0$ for $r > 1$ odd, by
Theorem~\ref{T:Nformulae}, to see that
%to replace $r$ by $2s$ in the first of the three products (see
%\cite[Lemmas 1.3.15, 1.3.16]{genfunc}). This first infinite product becomes
\[
 \prod_{r \geq 1}\left( 1 + \frac{u^r}{q^r-1}
 \right)^{\frac{1}{2}N^*(q, r)} =  \left(1+\frac{u}{q-1}\right)
  \prod_{s \geq 1} \left( 1+ \frac{u^{2s}}{q^{2s} - 1} \right)^{\frac{1}{2}N^*(q^2, 2s)}.
\]
Similarly, since $N^\sim(q,1)=N(q,1)+2$ and $N^\sim(q,r)=0$ for $r$
even, by Theorem~\ref{T:Nformulae}, 
\[
\prod_{r \geq 1}\left(1 + \frac{u^r}{q^r + 1} \right)^{N^\sim(q, r)} 
=   \left(1+\frac{u}{q+1}\right)^2\prod_{\mbox{\footnotesize{$r\geq 1$ odd}}}
   \left( 1+ \frac{u^{r}}{q^{r} + 1} \right)^{N(q, r)}.
\]
Since $R_{\UU}(q, u)$ is uniformly convergent, we can rearrange the infinite
product. Substituting the above displayed expression into
$R_\UU(q, u)$ gives
\[
  R_\UU(q, u)\kern-2pt=\kern-2pt\frac{(1+\frac{u^2}{q^2-1})^{3/2}}{1+\frac{u}{q+1}}
    \prod_{r \geq 1} \kern-2pt\left( 1+ \frac{u^{2r}}{q^{2r} - 1} \right)^{\kern-2pt\frac{1}{2}(N^*(q^2, 2r)+M^*(q^2, r) - N^\sim(q, r))}\kern-12pt
\prod_{r\geq1\ \mbox{\small odd}}\kern-6pt \left( 1+ \frac{u^{r}}{q^{r} + 1} \right)^{\kern-2pt N(q, r)}\kern-5pt.
\]

In the first infinite product, we rewrite the exponents using
Lemma~\ref{lem:nqr}, then replace $2r$ by $r$, and finally we combine the two infinite products to obtain
 %\begin{align*}
  %R_\UU(q, u) &= \left(1+\frac{u}{q+1}\right)^{-1}   \prod_{r \geq 1} \left( 1+ \frac{u^{2r}}{q^{2r} - 1} \right)^{N(q, 2r)}   \prod_{\mbox{\footnotesize{$r \geq 1$ odd}}} \left( 1+ \frac{u^{r}}{q^{r} + 1} \right)^{N(q, r)} \\
%	  &= \left(1+\frac{u}{q+1}\right)^{-1}   
 % \prod_{\mbox{\footnotesize{$r \geq 2$ even}}} \left( 1+ \frac{u^{r}}{q^{r} - 1} \right)^{N(q, r)}
 % \prod_{\mbox{\footnotesize{$r \geq 1$ odd}}} \left( 1+ \frac{u^{r}}{q^{r} + 1} \right)^{N(q, r)}
%\end{align*}
%so
\begin{equation}\label{R:compute}
 R_{\UU}(q, u) = \left(1+\frac{u}{q+1}\right)^{-1}   \prod_{r \geq1} \left( 1+ \frac{u^{r}}{q^{r} - (-1)^r} \right)^{N(q,
            r)}. 
\end{equation}
We have shown in Theorem~\ref{thm:r_prod} that this expression
converges for $|u|<1$. 
By \cite[Lemma 1.3.10(b)]{genfunc} with $u$ replaced with $u/q$, the
following equality holds for $|u|<1$,
\[
  \frac{1-u/q}{1-u} \prod_{r \geq 1} \left( 1- \frac{u^{r}}{q^{r}}\right)^{N(q, r)}=1.
\]
Multiplying this by our expression for $R_\UU(q, u)$ gives that for
$|u| < 1$ 
%\begin{equation}\label{eq:R-for-lim}
\begin{align}
R_\UU(q, u) &= \frac{1-\frac{u}{q}}{(1-u)(1+\frac{u}{q+1})}   
\prod_{r \geq1} \left(\left( 1+ \frac{u^{r}}{q^{r} - (-1)^r} \right)\left( 1- \frac{u^{r}}{q^{r}}\right)\right)^{N(q, r)} \label{eq:R-for-lim}\\ 
&=\frac{1-\frac{u}{q}}{(1-u)(1+\frac{u}{q+1})}   
\prod_{r\geq1} \left(1- \frac{u^{r}(u^r-(-1)^r)}{q^r(q^{r} - (-1)^r)}
  \right)^{N(q, r)}. \notag
\end{align}
%\begin{array}{rl}
%R_\UU(q, u) &= \frac{1-\frac{u}{q}}{(1-u)(1+\frac{u}{q+1})}   
%\prod_{r \geq1} \left(\left( 1+ \frac{u^{r}}{q^{r} - (-1)^r} \right)\left( 1- \frac%{u^{r}}{q^{r}}\right)\right)^{N(q, r)} \\
%&=\frac{1-\frac{u}{q}}{(1-u)(1+\frac{u}{q+1})}   
%\prod_{r\geq1} \left(1- \frac{u^{r}(u^r-(-1)^r)}{q^r(q^{r} - (-1)^r)}
%  \right)^{N(q, r)}. 
%\end{array}
%\end{equation}

Now consider the above expression for $R_{\UU}(q, u)$.
By \cite[Corollary 1.3.2 and Lemma 1.3.10(a)]{genfunc}, 
$R_\UU(q, u)$ has a simple pole 
at $u=1$ and is of the form $(1-u)^{-1}H(u)$ where 
\[H(u) = \frac{1-\frac{u}{q}}{(1+\frac{u}{q+1})}   
\prod_{r\geq1} \left(1- \frac{u^{r}(u^r-(-1)^r)}{q^r(q^{r} - (-1)^r)}
  \right)^{N(q, r)}.
\]
Using the bound $N(q, r) < q^{r}/r$ from
Lemma~\ref{L:bds_N}(i), and \cite[Corollary 1.3.2]{genfunc}, we see
that $H(u)$  is analytic in the disc $|u|< \sqrt{q}$.
Thus by \cite[Lemma 1.3.3]{genfunc}, 
$\lim_{n\rightarrow\infty}r_\UU(2n,q) = H(1)$, and the result
follows. 
%and  $|r_\UU(2m,q) - H(1)| = o(s^{-m})$ for any $s$ such that $1<s<\sqrt{q}$. 
%he value of $H(1)$ is the stated limit.
%\[
%H(1) =   \frac{1-\frac{1}{q}}{(1+\frac{1}{q+1})} \prod_{r \ \mbox{odd}}  \left( 1- \frac{2}{q^r(q^{r} + 1)} \right)^{N(q, r)}.
%\]
\end{proof}

\subsection{Upper and lower bounds on 
 $r_\UU(2n,q)$}
%We shall now prove various upper and lower bounds on the values of
%$r_{\UU}(2n, q)$.

\begin{notation}\label{notgen} 
% Let $\mathcal{P}$ be the set of all power series in $z$ with
%real coefficients. 
If $f(z):=\sum_{n\geq0}f_{n}z^{n}$ is a power series, we write $[z^{n}]f(z)$
to denote the coefficient $f_{n}$ of $z^{n}$, and we write $|f|(z)$
for the power series $\sum_{n \geq 0} |f_n|z^n$.  Let $g(z):=\sum
_{n\geq0}g_{n}z^{n}$. We write
$
f(z)\ll g(z)$  if $f_{n}\leq g_{n}$ for all $n$. 
\end{notation}

\begin{defn}\label{def:AB}
Recall \eqref{eq:R-for-lim}, and define
 $R_\UU(q, u) =  A_\UU(q, u) B_\UU(q, u)$, where 
\begin{align*}
  A_\UU(q, u) &= 
\sum_{n\geq 0} a_n u^n := 
\frac{1-\frac{u}{q}}{(1-u)(1+\frac{u}{q+1})}, \quad\textup{and}\\
  B_\UU(q, u) &= \sum_{n \geq 0}b_n u^n:= \prod_{r \geq 1} \left( 1- \frac{u^{r}(u^r - (-1)^r)}{q^{r}(q^r - (-1)^r)} \right)^{N(q, r)}.
\end{align*}
%Let  $ A_\UU(q, u)=
%\sum_{n\geq 0} a_n u^n$ and $ B_\UU(q, u) =  \sum_{n \geq 0}b_n u^n$.
We let $B_{\UU}(q, u) = \prod_{r \geq 1} B_{\UU}(r, q, u)$, where
\[B_{\UU}(r, q, u) = \left( 1- \frac{u^{r}(u^r -
        (-1)^r)}{q^{r}(q^r - (-1)^r)} \right)^{N(q, r)}.\] 
\end{defn}

We shall bound $r_\UU(2n, q)$ by first  bounding the $b_n$.

\begin{lemma}\label{lem:b}
For all $n$,  the absolute value $\left\vert b_{n}\right\vert \leq\beta q^{-n/2}$, where
$\beta:=q/(q-1)$. 
In particular, $B_{\UU}(q, u)$ is absolutely convergent for all $|u| <
q^{1/2}$. 
\end{lemma}

\begin{proof}
First we claim that for $r\geq1$ and $n\geq0$,
\begin{equation}\label{E:u_bound}
 |B_{\UU}|(r, q, u)  \ll
 \left( 1+ \frac{u^{r}(u^r +1)}{q^{r}(q^r - 1)} \right)^{N(q, r)}.
\end{equation}
To see this, let $N:= N(q,r)$. 
Then we calculate that when $r$ is even
\begin{align*}
B_{\UU}(r, q, u) = 
\left(  1-\frac{u^r(u^r+1)}{q^r(q^r+1)}\right)^{N} %&= \sum_{j=0}^N (-1)^j
                                                % \binom{N}{j}\frac{u^{rj}
                    % (u^r+1)^j}{(q^r(q^r+1))^j}
		%=  \sum_{j=0}^N \sum_{k=0}^j (-1)^{j} \binom{N}{j}\binom{j}{k}\frac{u^{r(j+k)}}{(q^r(q^r+1))^j}\\
		&=  \sum_{n=0}^{2N} \left( \sum_{n/2\leq j\leq n}
                  (-1)^{j}
                  \binom{N}{j}\binom{j}{n-j}\frac{1}{(q^r(q^r+1))^j}\right)
                  u^{rn}, 
\end{align*}
and hence when $r$ is even
\begin{align*}
|B_{\UU}|(r, q, u) 
%& \ll \sum_{n=0}^{2N} \left( \sum_{n/2\leq j\leq n}
 % \binom{N}{j}\binom{j}{n-j}\frac{1}{(q^r(q^r+1))^j}\right) u^{rn}\\	
& \ll  \left(  1+\frac{u^r(u^r+1)}{q^r(q^r+1)}\right)^{N} \ll \left(
  1 + \frac{u^r(u^r+1)}{q^r(q^r-1)} \right)^N.
\end{align*}
It is shown in \cite[p.\,433]{DPS} that when $r$ is odd
$$|B_{\UU}|(r, q, u) \ll
\left(  1+\frac{u^r(u^r+1)}{q^r(q^r-1)}\right)^{N}.$$
Hence \eqref{E:u_bound} holds for all values of $r$. 

Now, from  Definition~\ref{def:AB},  we deduce from \eqref{E:u_bound} that
$$|B_{\UU}|(q, u) \ll \prod_{r \geq 1} |B_{\UU}|(r, q, u) \ll 
\prod_{r \geq 1} \left( 1 + \frac{u^r(u^r+1)}{q^r(q^r-1)}
\right)^{N(q, r)}.$$

Comparing the expression for $B_{\UU}(q, u)$ with that for $B(q, u)$
in \cite[Equation (8)]{DPS}, we can reason just as in the proof of
Lemma 4.1 in \cite[p.\,434]{DPS} that the bound in 
\cite[Equation (10)]{DPS} is
valid for $|B_{\UU}|(q, u)$. This is precisely the bound in the
current lemma. 

The convergence claims are clear. 
\end{proof}

\begin{thm}\label{T:RuBu}
Let  %$B_\UU(q,u)$ be as in
%Definition~\ref{def:AB}, and let
% $\alpha\kern-1pt=\kern-1pt\frac{q^2-1}{q^2+2q}$.
 $\alpha=(q^2-1)/(q^2+2q)$.
  Then $B_\UU(q,1)=\sum_{n\geq0} b_n$ converges and
  $\lim_{n\to\infty} r_\UU(2n,q)$ equals $\alpha B_\UU(q,1)$. Furthermore
  \[
    \varepsilon_n:= \left\vert r_\UU(2n,q)-\alpha B_\UU(q,1)\right\vert
    <\frac{\alpha q^{1/2}+2}{q^{(n-1)/2} (q-1)(q^{1/2}-1)} =
    O(q^{-(n-1)/2}). 
    %\qquad\textup{and $\frac{8}{15}\le\alpha<1$.}
  \]
\end{thm}

\begin{proof}
  By Definition~\ref{def:AB},
  the coefficient
  $r_\UU(2n,q)=\sum_{k=0}^na_{n-k}b_k$. Notice that 
\[
A_\UU(q,u) = \frac{1-1/q^2}{1+2/q} \cdot \frac{1}{1-u} +
\frac{2q+1}{q(q+2)}\cdot \frac{1}{1+u/(q+1)}. 
\] 
Hence $a_n= \frac{q^2-1}{q^2 + 2q} + c_n$, where $c_n := 
(-1)^n \frac{2q+1}{q(q+2)(q+1)^n}.$  By Lemma~\ref{lem:b}, 
  $B_\UU(q,1)$ converges. 
Therefore 
  \begin{equation}\label{E:rUlim}
    r_\UU(2n,q)-\alpha B_\UU(q,1)
    =\sum_{k=0}^n(\alpha+c_{n-k})b_k-\alpha\sum_{k\geq0}b_k
    =\sum_{k=0}^nc_{n-k}b_k-\alpha\sum_{k>n}b_k.
  \end{equation}
  We bound the terms on the right side of~\eqref{E:rUlim} as follows.
  Using $(2q+1)/(q(q+2))<2/(q+1)$ gives
  \[|c_n| = (2q+1)(q(q+2)(q+1)^n)^{-1} < 2(q+1)^{-n-1} <
    2 q^{-n-1}.\] 
Hence, by Lemma~\ref{lem:b}, 
  \[
    \left\vert\sum_{k=0}^n c_kb_{n-k}\right\vert
    \kern-1pt<\kern-1pt\sum_{k=0}^n \frac{2}{q^{k+1}} \frac{\beta}{q^{(n-k)/2}}\kern-1pt=\kern-1pt
    \frac{2\beta}{q^{n/2 + 1}}\sum_{k=0}^n q^{-k/2}
    \kern-1pt<\kern-1pt\frac{2\beta }{q^{n/2+1}(1-q^{-1/2})} \kern-1pt=\kern-1pt
    \frac{2\beta}{q^{(n+1)/2}(q^{1/2} - 1)}.
    \]
Similarly, from Lemma~\ref{lem:b}, we deduce that 
  \[
    \left\vert\alpha\sum_{k>n} b_k\right\vert<\alpha\sum_{k>n} \beta q^{-k/2}
    =\frac{\alpha\beta}{q^{(n+1)/2}(1-q^{-1/2})} = \frac{\alpha\beta
      q^{1/2}}{q^{(n+1)/2}(q^{1/2} + 1)}.
  \]
 Substituting the previous two displayed equations into
 \eqref{E:rUlim}, and setting $\beta =q/(q-1)$, gives
  \[   \left\vert r_\UU(2n,q)-\alpha B_\UU(q,1)\right\vert
    <\frac{(\alpha q^{1/2}+2)\beta }{q^{(n+1)/2} (q^{1/2}-1)} =
    \frac{\alpha q^{1/2} + 2}{q^{(n-1)/2} (q - 1)(q^{1/2} - 1)}.
  \]
 Therefore $r_\UU(2n,q)\to\alpha B_\UU(q,1)$ as $n\to\infty$ as claimed.
\end{proof}

%\subsection{Proof of Theorem~\ref{thm:ru_bound}}

We next record a technical lemma.

\begin{lemma}\label{L:power_bound}
Let $a, b \in \mathbb{R}_{>0}$ such that $b>1$ and $ab < 1$. Then
$(1-a)^b \geq  1 -
ab$.
\end{lemma}

\begin{proof}
It suffices to show that $b \log (1-a) \geq
\log(1-ab)$. We use the  expansion $\log (1-x) =
- \sum_{n = 1}^\infty x^n/n$, valid for $0 < x < 1$. 
Notice that 
\[
b \log(1-a) =  %& = - b \sum_{n \geq 1}\frac{a^i}{i} = 
-ab
- a^2b\left( \frac{1}{2} + \frac{a}{3} + \cdots\right), \quad 
\log(1-ab) %&= - \sum_{n \geq 1} \frac{a^i b^i}{i} 
= - ab - a^2b\left( \frac{b}{2} + \frac{ab^2}{3} + \cdots\right).
\]
Since $b > 1$,  it follows that $a^ib^{i+2}/(i+2) >
a^i/(i+2)$ for all $i$, from which the result follows.
\end{proof}

\begin{thm}\label{T:R_bounds}
Let 
$$\mu=\frac{q^2-1}{q^2+2q}\left(1-\frac{2}{q(q+1)}\right)^{q-1},$$ 
let $\delta=1-3/(4q^3)$, and let $\varepsilon_n$ be as in
Theorem~$\ref{T:RuBu}$. Then 
$  \mu \delta- \varepsilon_n <r_\UU(2n,q)<\mu + \varepsilon_n$.
In particular,  $r_{\UU}(2, 3) = 0.25$, and
$0.3433  < r_{\UU}(2n, 3) < 0.3795$ for $n\geq 2$. 
\end{thm}

\begin{proof}
We first bound $B_{\UU}(q, 1)$. 
It follows from  Definition~\ref{def:AB} that
\[
B_\UU(q, 1) = \prod_{\mbox{\footnotesize{ $r \geq 1$ odd}}} \left( 1-
  \frac{2}{q^{r}(q^r + 1)} \right)^{N(q, r)} \leq \left( 1 -
  \frac{2}{q(q+1)}\right)^{N(q, 1)}.
\]
By Theorem~\ref{T:Nformulae}, the upper bound above is
$\gamma:= (1 - 2/(q(q+1)))^{q-1}$. For a lower bound,
note that $1-2/(q^r(q^r+1))>1-2/q^{2r}$,  and $N(q,
r)\leq  q^r/r $ by Lemma~\ref{L:bds_N}(i). Hence
\begin{align*}
  B_\UU(q, 1) & 
  %= \prod_{r\ \mbox{\footnotesize{odd}}}\left( 1 - \frac{2}{q^r(q^r+1)}\right)^{N(q,r)}
  = \gamma \prod_{r\geq3\ \mbox{\footnotesize{odd}}}\left(1-\frac{2}{q^r(q^r+1)}\right)^{N(q, r)} 
  \geq \gamma\prod_{r\geq3\ \mbox{\footnotesize{odd}}}\left(
    1-\frac{2}{q^{2r}}\right)^{q^r/r}.
\end{align*}
 Lemma~\ref{L:power_bound} with $a = 2q^{-2r}$
  and $b = q^{r}/r$ gives $(1-2q^{-2r})^{q^r/r} \geq 
 1 - 2/(rq^r)$, and by induction 
$$B_{\UU}(q, 1) %\geq \gamma\prod_{r\geq3\ \mbox{odd}}\left(
                %1-\frac{2}{q^{2r}}\right)^{q^r/r}
 \geq \gamma\prod_{r\geq3\ \mbox{\footnotesize{odd}}}
\left(1-\frac{2}{rq^{r}}\right)
\geq \gamma\left(1-\sum_{r\geq3\ \mbox{\footnotesize{odd}}} \frac{2}{rq^r}\right).$$
However, 
\[
  \sum_{r\geq3\ \mbox{\footnotesize{odd}}} \frac{2}{rq^r}<\sum_{r\geq3\ \mbox{\footnotesize{odd}}}
  \frac{2}{3q^r}
 = \frac{2}{3q^3} \sum_{r \geq 0 \mbox{\footnotesize{ even}}}\frac{1}{q^r}  =
 \frac{2}{3q^3} \sum_{s \geq 0}\frac{1}{q^{2s}} = \frac{2}{3q^3}
 \frac{1}{1 - q^{-2}}  \leq \frac{3}{4q^3}\]
and so $B_\UU(q, 1) >\gamma\left(1-3/(4q^3)\right)$. 
Setting $\delta = 1 - 3/(4q^3)$  and 
$$
\mu = \alpha\gamma = \frac{q^2-1}{q(q+1)} \left(1 - \frac{2}{q(q+1)}\right)^{q-1}
$$ 
gives $\mu\delta <  \alpha B_\UU(q, 1)
  <\mu$. The main claim follows  from Theorem~\ref{T:RuBu}. 
%\[
%  \sum_{r\ \mbox{odd}} \frac{2}{rq^r}=\log\left(\frac{1+q^{-1}}{1-q^{-1}}\right)
%  \quad\textup{implies that}\quad  \gamma>B_\UU(q, 1)
%  >\gamma\left(1+\frac{2}{q}-\log\left(\frac{1+q^{-1}}{1-q^{-1}}\right)\right).
%\]
When $q=3$, this becomes 
%have $\delta =\frac{35}{36}$, 
%$\alpha=\frac8{15}$ and $\gamma=\frac{25}{36}$, so
%that $\mu = \alpha\gamma = 10/27$. 
%Putting these bounds together yields
$  0.3601 < \alpha B_\UU(3, 1) < 0.3704$. 

Finally, we estimate $r_{\UU}(2n,3)$ for $n\geq3$.
We compute the values of $r_{\UU}(2n, q)$ directly for $n \leq
20$, using the expression for $R_{\UU}(q, u)$ given
in \eqref{R:compute}, and we find that $r_{\UU}(0, q)$ and
$r_{\UU}(2, q)$ are as given, and that for $n \geq 2$ we can bound
$0.3433 < r_{\UU}(2n, q) < 0.3795$.
Assume therefore that $n \geq 21$.  For $q = 3$, 
Theorem~\ref{T:RuBu} simplifies to 
%\[ 
%  |r_{\UU}(2n, q) - \alpha B_{\UU}(3, 1)| \leq \eps_n=\frac{\alpha q^{1/2}
%    +2}{q^{(n-1)/2}(q-1)(q^{1/2} - 1)}.
%\]
%Substituting $q = 3$ and simplifying we get
\[ 
 |r_{\UU}(2n, 3) - \alpha B_{\UU}(3, 1)| \leq \eps_n =  \frac{27 + 19
    \sqrt{3}}{30}3^{-(n-1)/2}.
\]
However, $(27 + 19\sqrt{3})/30 < 2$ and $3^{-(n-1)/2}\leq 3^{-10}$, and
so 
\[ 
  \alpha B_\UU(3, 1) - 2/3^{10} \leq r_{\UU}(2n, 3) < 
\alpha B_{\UU}(3, 1) + 2/3^{10}.
\]
The bounds for $n\geq3$ now follow from $  0.3601 < \alpha B_\UU(3, 1) < 0.3704$. 
\end{proof}

\section{Controlling the eigenspaces of \texorpdfstring{$\inv(y)$}{}}\label{sec:r_u_as_prod}

We wish to estimate the proportion of pairs $(t,y) \in \Delta_\UU(2n,q)$ for which $\inv(y)$ 
induces a strong involution on one of the $t$-eigenspaces. 
A central issue underpinning this is the link between the eigenspaces of 
$\inv(y)$ and the characteristic polynomial of $y$ (acting on some
$U \leq V$). Suppose that there is a $y$-invariant decomposition $U = U^+ \oplus U^-$ such that

\begin{enumerate}  
\item[(a)]	the restriction $y^- := y|_{U^-}$ has a certain $2$-part order, say $2^B$, and

\item[(b)]	the restriction $y^+ := y|_{U^+}$ is guaranteed to have $2$-part order strictly less than $2^B$.
\end{enumerate}
%
%Then $U^\pm$ is the $(\pm 1)$-eigenspace of $\inv(y)|_U$. 
The $\eps$-eigenspace of $\inv(y)|_U$ is $U^\eps$, and it
is possible to detect whether conditions (a,b) hold from the
characteristic polynomial of $y|_U$. 
In Subsections~\ref{sec:related} and \ref{sub:trunc} we introduce
functions $G_{\UU,b}(q,u)$,  $R_{\UU,b}(q,u)$ and $G_{\UU,b}^-(q,u)$,
each 
related to $R_{\UU}(q,u)$,  for certain non-negative integers $b$. These three
functions will help detect these properties.  We will see that 
$G_{\UU,b}^-(q,u)$ counts pairs $(t^-, y^-)$ for which the 
$2$-part order of $y^-$ equals $2^{b-1}(q^2-1)_2$,  while the
pairs  $(t^+, y^+)$  counted by $R_{\UU,b}(q,u)$ are such that the
$2$-part order of $y^+$ is less than $2^{b-1}(q^2-1)_2$.
Thus properties (a) and (b)  are determined by the
characteristic polynomials of $y^\pm$. 
  
The functions $G_{\UU,b}^-(q,u)$ and $R_{\UU,b}(q,u)$ are therefore crucial. 
In particular we will need lower bounds on the sizes of the
coefficients of their power series. In Subsection~\ref{sec:related} we 
define functions $T_{\UU,b}(q,u)$, for positive integers
$b$,  and prove that $R_{\UU,b}(q,u) = R_{\UU}(q,u)
T_{\UU,b}(q,u)^{-1}$  (Theorem~\ref{thm:RbGb}).   The 2-part orders of
the roots will play a critical role. 
In Subsection~\ref{sub:trunc}, we introduce a truncated version  
$F_{\UU,b}(q,u)$ of $G_{\UU,b}^-(q,u)$ from which it is easier to
deduce lower bounds for the coefficients of $G_{\UU,b}^-(q,u)$.
We also prove the fundamental
Lemma~\ref{lem:akl} that links the number of pairs $(t,y)$ in
$\Delta_\UU(2n,q)$, where $y$ has a particular type of characteristic
polynomial, with products of certain coefficients of  
$R_{\UU,b}(q,u)$ and $G_{\UU,b}^-(q,u)$.  In the remaining two
technical subsections (\ref{sub:T} and \ref{sub:RUb}) we obtain
the required lower bounds: for
the coefficients of $T_{\UU,b}(q,u)^{-1}$ (in \ref{sub:T}),
then for $R_{\UU,b}(q,u)$ and 
$F_{\UU,b}(q,u)$  (in~\ref{sub:RUb}).  These bounds are used
in \S\ref{sec:thm_main} to prove Theorem~\ref{main}.   

Our methods in this section are guided by the work of Dixon, Praeger
and Seress in \cite{DPS}, and we have used
similar notation to facilitate comparisons between the two
analyses. However, the results of \cite{DPS} unfortunately do not
carry over without careful re-analysis. 

We shall continue to assume that $q$ is an odd prime power. 

\subsection{Related functions $G_{\UU,b}(q,u)$, $R_{\UU,b}(q, u)$ and $T_{\UU,b}(q,u)$}\label{sec:related}

Recall from 
Theorem~\ref{thm:r_prod} that $R_\UU(q, u) = S_0(q, u)S(q, u)$. Recall
also the relationships between the 
power series given in Lemma~\ref{lem:ABCD}. 
For each $b \geq 0$, we now define an infinite series $G_{\UU, b}(q, u) = \sum_{n
  \geq 0}g_b(2n, q) u^n$ as follows.
First define 
\begin{equation}\label{defG0} 
\begin{array}{rl}
G_{\UU, 0} (q,u)& :=  \prod_{r \geq 1}  \left( 1 + \frac{u^{r}}{q^{r} - 1} \right)^{A(q, r)}
\prod_{r \geq 1} \left( 1+ \frac{u^{r}}{q^{r} + 1} \right)^{B(q, r) +
                  C(q, r)}\\
& = S_0(q, u)   \prod_{r \geq 1} \left( 1+ \frac{u^{r}}{q^{r} - 1} \right)^{\frac{1}{2}N^*(q^2, r)}
\prod_{r \geq 1} \left( 1+ \frac{u^{r}}{q^{r} + 1} \right)^{N^\sim(q,
  r)}.\\
\end{array}
\end{equation}

It follows from \S\ref{sec:genfnrU} that 
$g_0(2n, q) |\gu{2n}{q}|$ is equal to the number of
pairs $(t, y) \in \Delta_\UU(2n, q)$ for which each factor in the
$\UU*$-factorisation of $c_y(X)$ is
of type A, B or C.

For the infinite product $S(q, u)=R_\UU(q,u)/S_0(q,u)$, the terms are labelled by integers $r$ such that $r=2^{b-1}m$ for some positive integers $b, m$ with $m$ odd. %
We henceforth abbreviate ``all odd integers $m\geq1$" simply as ``$m$ odd".
For each $b \geq 1$, define
\begin{equation}\label{defGb} 
G_{\UU, b}(q,u)  := \prod_{m~{\rm odd}}\left(
                  1+\frac{u^{2^{b}m}}{q^{2^{b}m}-1}\right)^{D(q,
                  2^{b-1}m)}, 
\end{equation}
and so by Lemma~\ref{lem:ABCD}
\[
G_{\UU, b}(q, u)=\begin{cases} \displaystyle\prod_{m~{\rm odd}}\left(  1+\frac{u^{2^{b}m}}{q^{2^{b}m}-1}\right)
^{\frac{1}{2}M^*(q^2,
2^{b-1}m) - \frac{1}{2}N^\sim(q,2^{b-1}m)} & \textup{for  $b > 1$,}\\
 \displaystyle\left(1 + \frac{u^2}{q^2-1}\right)^{3/2}
\prod_{m~{\rm odd}}\left(  1+\frac{u^{2 m}}{q^{2 m}-1}\right)
^{\frac{1}{2}M^*(q^2,
m) - \frac{1}{2}N^\sim(q,m)} &\textup{for  $b=1$.}\\
\end{cases}
\]

%
%\[
%\begin{array}{rll}
%G_{\UU, b}(q, u)& = \prod_{m~{\rm odd}}\left(  1+\frac{u^{2^{b}m}}{q^{2^{b}m}-1}\right)
%^{\frac{1}{2}M^*(q^2, 2^{b-1}m) - \frac{1}{2}N^\sim(q,
%                  2^{b-1}m)} & \text{for  $b > 1$}\\
%G_{\UU, 1}(q, u) & = (1 + \frac{u^2}{q^2-1})^{3/2}
%\prod_{m~{\rm odd}}\left(  1+\frac{u^{2m}}{q^{2m}-1}\right)
%^{\frac{1}{2}M^*(q^2, m) - \frac{1}{2}N^\sim(q, m)} &\\
%\end{array}
%\]

%Notice that since for $b \geq 1$ the series
%$G_{\UU, b}(q, u)$ is the terms from Type D, the
%term $u^{2^bm}$ corresponds to irreducible polynomials of degree
% $r = 2^{b-1}m$, and $\UU*$-irreducible polynomials of degree
% $2^{b+1}m$. Furthermore, since the Type is D, the 2-part of the order
% of the root of each irreducible polynomial is at most $(q^{2^bm}
% -1)_2 = 2^{b-1}(q^2-1)_2$. 

It follows from \S\ref{sec:genfnrU} that for $b \geq 1$ the quantity 
$[u^{n}]G_{\UU, b}(q,u)\left\vert \gu{2n}{q}\right\vert$, that is to
say, 
$g_b(2n,q)\left\vert \gu{2n}{q}\right\vert$, is 
equal to the number of pairs $(t,y)\in\Delta_\UU(2n,q)$ for which 
each factor in the $\UU*$-factorization of 
$c_y(X)$ is of type D,  and each $\UU*$-irreducible has four irreducible factors over
$\F_{q^2}$ each of degree $r = 2^{b-1}m$ for some
odd $m$. In particular, 
the $\UU*$-irreducible polynomial $g(X)$ has degree $4r = 2^{b+1}m$ with $m$
odd, and $\omega_2(g) \leq (q^{2r}-1)_2 = 2^{b-1}(q^2-1)_2$; moreover a large fraction of these polynomials $g(X)$ have $\omega_2(g) = (q^{2r}-1)_2 = 2^{b-1}(q^2-1)_2$  (see Definition~\ref{def:D_minus} and Lemma~\ref{lem:gamma_12}). 

For $b \geq 1$, we now define an ascending chain of subsets
$\Delta_{\UU, b}(2n,q)$ of $\Delta_{\UU}(2n, q)$. 
Let
$\Delta_{\UU, b}(2n, q)$ consist of those $(t, y) \in
  \Delta_{\UU}(2n, q)$ such that each $\UU*$-irreducible factor
  $g(X)$ of
  $c_y(X)$
 is either of type A, B,  or 
C,  or is of type D and has the 2-part of its degree dividing
$2^b$. Thus in particular $\Delta_{\UU, 1}(2n, q)$ contains only those
$(t, y)$ where each $\UU*$-irreducible factor of $c_y(X)$ is of
type A, B, or C; whilst $\Delta_{\UU, 2}(2n, q)$ also allows
factors of type D, provided that their
degree is $4m$ for some odd $m$.

\begin{defn}\label{def:Rb}
For $b \geq 1$, let
$r_{\UU, b}(2n,q):=\left\vert \Delta_{\UU,b}(2n,q)\right\vert/|\gu{2n}{q}|$ for $n > 0$, and let
$r_{\UU, b}(0, q):= 1$.  We define $R_{\UU, b}(q, u):= \sum_{n=0}^{\infty}%
r_{\UU, b}(2n,q)u^{n}$, and 
for $b \geq 1$, set
$T_{\UU, b}(q,u)
:=\prod_{k\geq b} G_{\UU, k}(q,u)$. 
\end{defn}

%The functions $G_{\UU,b}(q,u)$,  $R_{\UU,b}(q,u)$ and $R_{\UU}(q,u)$ are related as follows.

\begin{thm}\label{thm:RbGb}
The power series $R_{\UU}(q, u)$, and $G_{\UU, b}(q, u)$ (for
$b\geq0$), and $R_{\UU, b}(q, u)$  and $T_{\UU, b}(q, u)^{-1}$
(for $b \geq 1$) all converge absolutely and uniformly in the open disc  $|u|
< 1$. In this disc,
\[
R_\UU(q,u)=\prod_{b=0}^{\infty}G_{\UU, b}(q,u),  \ \ 
R_{\UU, b}(q,u) =\prod_{k=0}^{b-1}G_{\UU, k}(q,u) \ \mbox{and} \ R_{\UU, b}(q,u) = R_{\UU}(q,u) T_{\UU, b}(q,u)^{-1}\text{.}
\]
\end{thm}

\begin{proof}
By Theorem~\ref{thm:r_prod},  $R_{\UU}(q, u)$ converges absolutely and uniformly in the disc  $|u|< 1$. 
A similar argument shows that the  $G_{\UU, b}(q, u)$
converge absolutely and uniformly for $|u|<1$. Hence 
$T_{\UU, b}(q, u)$ is also absolutely convergent for each $b$. Since convergent
products converge to nonzero limits, it follows that $T_{\UU, b}(q,
u)^{-1}$ is also absolutely convergent.

%Next, looking at $G_{\UU, b}(q, u)$, we see that $G_{\UU, 0}(q, u)$
%converges absolutely and uniformly just as in the second and third paragraph of the proof of 
%Theorem~\ref{thm:r_prod}. For $G_{\UU, b}(q, u)$ with $b > 0$, we again use
%Lemma~\ref{L:bds_D}(ii)  to see (absolute and uniform) convergence. 
%
From \eqref{eq:prod2}, we see that the terms of $R_{\UU}(q, u)$ are 
a permutation of the terms of  $\prod_{b = 0}^\infty G_{\UU, b}(q, u)$.
The absolute convergence  for $|u| < 1$ of
each infinite expression in the first displayed  equality in the statement implies that this equality holds.

Next, let $b\geq1$. Since $0 < r_{\UU, b}(2n, q) < r_{\UU}(2n, q)$ for
all $n$,  $R_{\UU,b}(q, u)$ converges absolutely and uniformly for
$|u|< 1$. 
It follows from the discussion after
\eqref{defGb}, and Definition~\ref{def:Rb}, that  $R_{\UU, b}(q, u)$ is a product
of a permutation of the terms of $\prod_{k = 0}^{b-1} G_{\UU, k}(q,
u)$. 
The absolute convergence for $|u| < 1$ of
each infinite expression in the second displayed equality in the
statement   implies the equality of these functions. The final
equality is now immediate. 
\end{proof}

\subsection{Truncations of the power series  $G_{\UU, b}(q, u)$}\label{sub:trunc}

For the definitions of the subset $\cD^-_{4r}$ of $\cD_{4r}$ and the
quantity  $N^-_{\UU}(q, 4r)$, see Definition~\ref{def:D_minus}. 

\begin{defn}\label{def:Gubminus}
For $b\geq1$, we `truncate' the infinite product defined by 
\eqref{defGb} by reducing the exponent of each term, and hence 
removing some of the factors.
We set%
\begin{equation}
G_{\UU, b}^-(q,u)  = \sum_{n\geq0}
g_{\UU, b}^-(2n,q)u^n :=\prod_{m~{\rm odd}}\left(  1+\frac{u^{2^{b}m}}{q^{2^{b}m}%
-1}\right)  ^{N^-_{\UU}(q,2^{b+1}m)}. \label{E: Gb0}%
\end{equation}
\end{defn}

\begin{remark}\label{rem:Gb0}
For $b > 1$ the product expression for $G_{\UU, b}^-(q, u)$ is 
 a truncation of the one for $G_{\UU, b}$,  because $\cD^-_{4r} \subset
\cD_{4r}$. 
For $b = 1$ notice that replacing the exponent $D(q, m)$ in
\eqref{defGb}  by the exponent $N_{\UU}^-(q, 2m)$
 either preserves or decreases exponents, even for the term $m = 1$, 
as the exponent of $(1
+ \frac{u^2}{q^2 - 1})$ in $G_{\UU, 1}(q, u)$ is 
$$
(3/2) + M^*(q^2, 1)/2 - N^\sim(q, 1)/2 =
D(q, 1) = (q-1)^2/4 \geq (q^2-1)/8,
$$ 
since $q \geq 3$. 
%
%Notice that %since $G_{\UU, b}^-(q, u)$ is a truncation of $G_{\UU,
  %b}(q, u)$, 
Theorem~\ref{thm:RbGb} shows that each $G_{\UU,
  b}^-(q, u)$ is absolutely convergent for $|u|<1$.
\end{remark}

We do not know the precise value of $N_{\UU}^{-}(q,2^{b+1}m)$,  but we
found a lower bound for it in Lemma~\ref{lem:gamma_12}(iii). Hence, 
 rather than calculate 
$G_{\UU, b}^-(q,u)$ it is simpler to compute
\begin{equation}\label{E: FF} 
F_{\UU, b}(q,u) = \sum_{n = 0}^\infty f_{\UU, b}(2n, q)u^n:
= \prod_{m~{\rm odd}}\left(  1+\frac{u^{2^{b}m}}{q^{2^{b}m}%
-1}\right)  ^{\left\lceil \frac{1}{8}N(q^2,2^{b-1}m)\right\rceil }.%
\end{equation}

Our next result shows the important role that the coefficients of  $R_{\UU, b}(q, u)$ and  $G_{\UU, b}^-(q, u)$ (and hence also of $F_{\UU, b}(q,u)$) play in estimating the proportion of pairs 
$(t,y)\in\Delta_{\UU}(2k,q)$ with the properties (a) and (b) discussed at the beginning of this section.

\begin{lemma}\label{lem:akl}
Fix $b > 1$, let $k\geq \ell\geq 0$ with $k>0$, and let $a_{k\ell}:= r_{\UU, b}(2k-2\ell, q)g_{\UU, b}^-(2\ell, q)$.
Then $a_{k \ell} |\gu{2k}{q}|$  is equal to the number of pairs 
$(t,y)\in\Delta_{\UU}(2k,q)$ such that the characteristic polynomial $c_y(X)$ for $y$
has the form $c_y(X)=c^-_y(X)c^+_{y}(X)$, where:
\begin{enumerate}[{\rm (i)}]
 \item $c^-_{y}(X)$ is the product of the $\UU*$-irreducible
   factors $g(X)$
of $c_y(X)$ which lie in the set
$\bigcup_{m~{\rm odd}}\cD_{2^{b+1}m}^{-}$; so in particular each has degree with
$2$-part $2^{b+1}$ and satisfies $\omega_2(g) = 
2^{b-1}(q^2-1)_2$. Furthermore, $\deg c^-_{y}(X)=2\ell$.

\item $c^+_{y}(X)$ is a product of $\UU*$-irreducible polynomials $g(X)$ which are either
 not of type D  or have degree with $2$-part dividing
$2^{b}$, and satisfy $\omega_2(g) \leq 2^{b-2}(q^2-1)_2$.
 \item If $\ell>0$ then \textrm{inv}$(y)$ is of type
   $(2k-2\ell,2\ell)$. 
\item $0\ll
F_{\UU,b}(q,u)\ll G_{\UU,b}^-(q,u)$, and if $f_{\UU, b}(2n, q) \neq 0$ then $2^b$ divides $n$.
\item $[u^n] F_{\UU,b}(q,u)\,|\gu{2n}{q}|$ is at 
most the number of pairs $(t,y)$ in $\Delta_{\UU}(2n,q)$ such that each 
$\UU*$-irreducible factor $g(X)$ of $c_y(X)$ 
satisfies $\omega_2(g) = 2^{b-1}(q^2-1)_2$.
\end{enumerate}
\end{lemma}

\begin{proof}
  Recall  that $r_{\UU, b}(2k-2\ell, q)|\gu{2k-2\ell}{q}|$  counts certain pairs
  $(t, y) \in \Delta_\UU(2k-2\ell,q)$ (Definition~\ref{def:Rb}). 
% such that each $\UU^*$-irreducible is either of type $A, B$ or
%C, or of type D with the $2$-part of its degree dividing $2^b$. 
By
Lemma~\ref{lem:gamma_12}(i)(ii), it follows from
 $b \geq 2$ that
$|y|_2 \leq 2^{b-2}(q^2-1)_2$. 

By construction, 
 $g_{\UU, b}^-(2\ell, q) \left\vert \gu{2\ell}{q}\right\vert $ is
the number of pairs $(t,y)\in\Delta_{\UU}(2\ell,q)$ such that each $\UU*$-irreducible
factor of $c_y(X)$ lies in $\bigcup
_{m~odd}\cD^-_{2^{b+1}m}$.
% from the definition of $\cD^-_{2^{b+1}m}$
%it follows that 
Such a $y$ satisfies $|y|_2 = 2^{b-1}(q^2-1)_2$, and the
$2$-part of the degree of each $\UU*$-irreducible factor is $2^{b+1}$. 
Notice that $$a_{k\ell}|\gu{2k}{q}| =
r_{\UU, b}(2k-2\ell, q)|\gu{2k-2\ell}{q}| \,\cdot \,
g_{\UU, b}^-(2\ell, q)
|\gu{2\ell}{q}| \,
\cdot \frac{|\gu{2k}{q}|}{|\gu{2k-2\ell}{q} \times \gu{2\ell}{q}|}.$$

By Lemma~\ref{lem:final_centraliser}, to count the 
pairs $(t, y) \in \Delta_{\UU}(2k, q)$ with decomposition $c_y(X) =
c_y^-(X)c_y^+(X)$ satisfying (i) and (ii), we can first count the number of
decompositions of $V$ as $U \perp W$ with $U$ and $W$ non-degenerate, and
$\dim(U) = 2k-2\ell$: this is \[\frac{|\gu{2k}{q}|}{|\gu{2k-2\ell}{q} \times
  \gu{2\ell}{q}|}.\] We then multiply by 
the number of possible actions of $t |_U$ and $y|_U$ such that
 all $\UU*$-irreducible factors of $y$ are of
type A, B or C, or of type D with $2$-part of the degree at most
$2^b$: this is exactly $r_{\UU, b}(2k-2\ell)|\gu{2k-2\ell}{q}|$.
 Finally we multiply by $g^-_{\UU, b}(2\ell, q)|\gu{2\ell}{q}|$ for
the number of choices of $t|_W$ and $y|_W$ that ensure that each
irreducible factor $g(X)$ of $c_{y|_W}(X)$ lies in
$\mathcal{D}^-_{2^{b+1}m}$ for some odd $m$. 
Parts (i) and (ii) now follow immediately.

For Part (iii), notice that by Part (i), $\omega_2(c_y^-(X)) =
2^{b-1}(q^2-1)_2$, whilst  by Part (ii), $\omega_2(c_y^+(X)) \leq
2^{b-2}(q^2-1)_2$. Hence if $\ell > 0$ then
$\inv(y)$ has $(-1)$-eigenspace of 
dimension $\deg(c_y^-)=2\ell$, and $\inv(y)$ has
type $(2k-2\ell,2\ell)$. 

For Part (iv) it is immediate from \eqref{E: FF} that each coefficient
$f_{\UU, b}(2n, q)$ 
of $F_{\UU, b}(q, u)$ is non-negative, and from 
Lemma \ref{lem:gamma_12}(iii) that $g^-_{\UU, b}(2n, q) \geq f_{\UU, b}(2n,
q)$ for all $n$. For the final claim, notice that if $f_{\UU, b}(2n,
q) > 0$ then $g^-_{\UU, b}(2n, q) > 0$, and so, as argued for Part (i)
above, there exists 
$(t, y) \in \Delta_{\UU}(2n, q)$ such that each $\UU*$-irreducible
factor of $c_y(X)$ has degree divisible by $2^{b+1}$. Hence in
particular $2^{b+1}$ divides $2n$, and the result follows. 

Part (v) now follows from Part (iv) and the proof of Part (i). 
\end{proof}

\subsection{Bounding the coefficients of \texorpdfstring{$T_{\UU,b}(q,u)^{-1}$}{}}\label{sub:T}

Recall the power series $T_{\UU, b}(q,u)$, see Definition~\ref{def:Rb}.
We will use the bounds derived for $r_{\UU}(2n,q)$ in
Theorem~\ref{T:R_bounds}, together with bounds we shall derive in
this subsection for the  coefficients of 
$T_{\UU, b}(q,u)^{-1}$,  to obtain bounds for the coefficients of $R_{\UU, b}(q,u)$.
It will suffice  to consider
only  $b\geq3$.
Now
\begin{equation}\label{defTb2}
T_{\UU, b}(q,u)^{-1}
= \prod_{k= b}^{\infty}  \prod_{m\,{\rm odd}}\left(  1+\frac{u^{2^{k}m}}{q^{2^{k}m}-1}\right)
^{-D(q, 2^{k-1}m)}
= \prod_{m= 1}^\infty \left(  1+\frac{u^{2^{b}m}}{q^{2^{b}m}-1}\right)
^{-D(q, 2^{b-1}m)}
\end{equation}
where the second rearrangement is permissible in
the disc $|u| < 1$ due to Theorem~\ref{thm:RbGb}.

Fix a value of $b \geq 3$ and define $d:=2^{b}$, $U:=u^{2^b}$ and
$Q:=q^{2^b}$. We will now bound the coefficients $t_{n}:=[U^{n}]T_{\UU}(U)$ of the
power series $T_{\UU}(U)$,  where
\begin{equation}\label{defTU} 
1-T_{\UU}(U):= \prod_{m=1}^{\infty}\left(  1+\frac{U^{m}}{Q^{m}-1}\right)
^{-D(q, dm/2)}.%
\end{equation}
However, we will need to take a somewhat indirect route to do so. 
\begin{lemma}\label{L:Tval}
Assume that $b\geq3$, that is, $d\geq8$, and define 
\[
W_{\UU}(U)  =\sum_{n\geq0} w_{n}U^{n} :=-\log\left(
1-T_{\UU}(U)\right)  +\frac{1}{2d}\log(1-U).
\] 
\begin{enumerate}[{\rm (i)}]
\item $T_{\UU}(U)$ and $W_{\UU}(U)$ are absolutely and uniformly
  convergent in the open disc
$|U|<1$.
\item In this disc, $1 - T_{\UU}(U) = T_{\UU,
  b}(q, u)^{-1}$. 
\item 
$w_{0}=0$, and  $\left\vert w_{n}\right\vert  < 2 d^{-1}n^{-1}(Q-1)^{-n/2}$
for all $n\geq1$.
\end{enumerate}
\end{lemma}

\begin{proof}
(i) Notice that the product in \eqref{defTU} converges
absolutely and uniformly
 if and only if the product $\prod_{m=1}^{\infty}\left(  1+U^{m}/(Q^{m}-1))\right)
^{D(q, dm/2)}$ does so too. By \cite[Lemma 1.3.1]{genfunc},
this happens if and only if $\sum_{m = 1}^\infty D(q,
dm/2)|U^m|/(Q^m-1)$ converges absolutely and uniformly. 
By Lemma~\ref{L:bds_D}(ii), 
 $D(q,dm/2)  < (q^{dm}-1)/2dm <  Q^m/2dm$, so the result follows.

\medskip

\noindent (ii) This is now immediate from \eqref{defTb2} and \eqref{defTU}. 

\medskip

\noindent (iii) We follow the same strategy (but with $W_{\UU}(U)$ in
place of $W(U)$) as in the proof of \cite[Lemma 4.2]{DPS}, to write
$W_{\UU}(U)  =W_{\UU, 1}(U)+W_{\UU, 2}(U)$
where%
\begin{align*}
W_{\UU, 1}(U)& :=\sum_{m=1}^{\infty}\left\{  D(q,dm/2)\frac{U^{m}}{Q^{m}-1}-\frac
{U^{m}}{2dm}\right\}  =\sum_{n=1}^{\infty}w_{1,n}U^{n}\text{, say} \\
W_{\UU, 2}(U) & :=\sum_{m=1}^{\infty}\sum_{k=2}^{\infty}(-1)^{k+1} D( q, dm/2)\frac{U^{mk}%
}{k(Q^{m}-1)^{k}}=\sum_{n=2}^{\infty}w_{2,n}U^{n},\text{ say.}%
\end{align*}
Notice that $w_0 = 0$. 
In order to treat the $w_{1,n}$, we use Lemma~\ref{L:bds_D}(iii) to 
get 
\[
D(q,dm/2)\frac{U^{m}}{Q^{m}-1}
= \frac{U^{m}}{2dm} \left(1 - \frac{\eta(q, dm/2)}{Q^{m/2}+1}\right)
\]
where $1 - 2 Q^{-m/3}\leq\eta(q,\frac{md}{2})<2.2$. Thus for all $n\geq1$,
\[
\left\vert w_{1,n} \right\vert  = \left\vert \frac{D(q, dn/2)}{Q^n-1} -
\frac{1}{2dn} \right\vert 
=   \left\vert \frac{- \eta(q, dn/2)}{2dn(Q^{n/2}+1)} \right\vert 
 \leq\frac{1}{2dn}\cdot \frac{2.2}{Q^{n/2}+1} \leq \frac{1.1 (Q-1)^{-n/2}}{dn}.
\]

Since
 $D(q,dm/2)\leq(Q^{m}-1)/2dm$, we can mimic the proof of \cite[Lemma
   4.2]{DPS} to deduce that
%\[
%\left\vert w_{2,n}\right\vert \leq %\sum_{m\mid n,~m<n}\frac{1}{2nd(Q^{m}%
%%-1)^{n/m\,-1}}\leq
%\sum_{1\leq m\leq n/2}\frac{1}{2dn(Q^{m}-1)^{n/m\,-1}}
%\]
%and hence to see that 
\[
\left\vert w_{2,n}\right\vert <(2dn)^{-1}(Q-1)^{-n/2}\left(  1-(Q-1)^{-1}%
\right)  ^{-1}\leq (2dn)^{-1}1.0002(Q-1)^{-n/2}.
\]

Hence
$\left\vert w_{n}\right\vert \leq\left\vert w_{1n}\right\vert +\left\vert
w_{2n}\right\vert <2d^{-1}n^{-1}(Q-1)^{-n/2}$
for all $n\geq1$ as required.
\end{proof}

Let $W_{\UU}(U)$ be as in Lemma~\textup{\ref{L:Tval}}, and let
$E(U):=\exp(-W_{\UU}(U))-1=\sum_{n=1}^{\infty}e_{n}U^{n}$. 
Let $h(U)=\sum_{k=1}^{\infty}h_{k}U^{k}$, say, be
the series for $1-\left(  1-U\right)  ^{1/2d}$. Then 
\begin{equation}\label{E:T(U)}
1-T_{\UU}(U)=(1-U)^{1/2d}(1+E(U))=\left(  1-h(U)\right)  (1+E(U)). %
\end{equation}
Comparing \eqref{E:T(U)} with \cite[Equation (13)]{DPS}, we see that replacing
$d$ by $2d$ in the discussion in \cite{DPS}, we may deduce from
\cite[Equation (14)]{DPS} that  for $k \geq 2$ 
%Now for $k\geq1$ we have
%\[
%h_{k}:=-\binom{1/2d}{k}(-1)^{k}=\frac{1}{2dk}\prod_{i=1}^{k-1}\left(  1-\frac
%{1}{2di}\right)  .
%\]
%In particular $2dh_1=1$. For $k\geq2$, since $1-\xi>\exp\left(  -\frac{\xi}{1-\xi}\right)  $
%for $0\leq\xi<1$, applying this to each term of the expression for $h_k$ gives
\begin{equation}\label{dhk} 
1 >  2dkh_{k}
%>\exp\left(  -\sum_{i=1}^{k-1}\frac{1}{2di-1}\right) \\
%  & \geq \exp\left( \frac{-1}{2d-1}  \sum_{i=1}^{k-1}\frac{1}{i}\right)
   >\exp\left( \frac{-(1 + \log k)}{2d-1}\right)\text{.}  %
\end{equation}
%Notice in particular that $1 > h_k > 0$. 
We use this to estimate the values of the coefficients $e_n$ and $t_{k}$.

\begin{lemma}\label{L: T(U) coeff}
Let $d = 2^b\geq8$, and  let $\gamma =(1+ d^{-1})(Q - 1)^{-1/2}$. 
\begin{enumerate}[{\rm (i)}]
\item $\left\vert
e_{n}\right\vert \leq\frac{2}{1+d}\gamma^{n}$ for all $n\geq1$. In particular, 
$\gamma\leq0.014$ and $d\,|e_1| < 0.025$. 
\item $dkt_{k}<0.5065$ for
$k\geq1$, whenever $dk\leq e^{d/2}$.
\end{enumerate}
\end{lemma}

\begin{proof}
(i) Lemma~\ref{L:Tval} shows that $\left\vert w_{n}\right\vert \leq
2 d^{-1}(Q-1)^{-n/2}$ for all $n\geq1$. Let $\beta=(Q-1)^{-1/2}$ and $\alpha=2 d^{-1}\beta$, so that
$|w_n|\leq \alpha\beta^{n-1}$ for all $n\geq1$, and 
$\gamma: =\alpha/2 + \beta = (1+ d^{-1})(Q-1)^{-1/2} \leq 1$. 
Thus \cite[Lemma 3.4]{DPS} applies to $-W_{\UU}(U)$ with this $\alpha$ and
$\beta$, 
 and yields $|e_n|\leq \alpha \gamma^{n-1} = \frac{2}{1 +d}\gamma^{n}$ for all $n\geq1$. 
From $q^d \geq 3^8$ we see that $\gamma\leq 0.014$, and 
that $d\,|e_1| \leq \frac{2d}{1+d}\gamma  
= 2 (Q-1)^{-1/2} < 0.025$.

\medskip

\noindent (ii) The proof is similar to that of the upper bound in \cite[Lemma
4.4]{DPS}, and we only give the necessary details. 
First let $k =1$. 
From %the product formula for $1 - T_{\UU}(U)$ in 
\eqref{defTU} we see that
$t_{1}= D(q,d/2)(q^{d}-1)^{-1} =
0.5d^{-1}(q^{d/2}-1)/(q^{d/2}+1)$ by Lemma~\ref{L:bds_D}(i), and so 
$t_{1} \leq  0.5d^{-1}$. 
Suppose therefore that $k \geq 2$.  

%Then by \eqref{dhk}, we have $2dkh_k < 1$.
%$d/2\geq\log dk\geq\log
%k+\log8>\log k+2$, so by \eqref{dhk},
%$\log(2dkh_{k})> -(1 + \log k) /(2d-1) > (1-d/2)/(2d-1) > -0.25$. This 
%and \eqref{dhk} implies that $1 > 2dkh_{k}>0.7788$.

Equations \eqref{E:T(U)} and \eqref{dhk} show that 
$t_{k}=h_{k}-e_{k}+\sum_{i=1}^{k-1}e_{k-i}h_{i}$ and $0 < h_k < 1$.  Thus
Part (i) 
gives
\[
 t_{k}- h_{k}  \leq - e_{k} +\sum
_{i=1}^{k-1}\frac{e_{k-i} }{2di}\leq\frac{2}%
{1+d}\left\{  \gamma^{k}+\frac{1}{2d}\sum_{i=1}^{k-1}\frac{1}{i}\gamma
^{k-i}\right\}  \text{ for }k\geq2
\]
where 
$\gamma
 \leq 1.125(q^d-1)^{-1/2}.$ 
Hence 
$dk\gamma\leq 1.125 e^4 (3^8-1)^{-1/2} < 0.759$.

Part (i) yields 
$(1-\gamma)^{-2} < 1.0295$.
Hence,  as in the proof of \cite[Lemma 4.4]{DPS}, 
\[
t_{k}-h_{k}   \leq\frac{2}{1+d}\left(
\gamma^{k}+\frac{\gamma}{2d(k-1)(1-\gamma)^{2}}\right)\\
%&=\frac{2.1 \gamma}{(1.05+d)dk}\left\{
%\gamma^{k-1}dk+\frac{k}{2(k-1)(1-\gamma)^{2}}\right\} \\
%&\leq\frac{2.1 \gamma}{d^2k}\left\{
%\gamma dk+\frac{k}{2(k-1)(1-\gamma)^{2}}\right\} \\
%  <\frac{2.1\gamma}{d^{2}k}\left( 0.763+1.0286%
%\right)  
< 3.577\gamma d^{-2}k^{-1}.
\]
Since $3.577\gamma d^{-1} %\leq 3.577\times 0.0144 \times0.125
<0.0065$ 
we have $t_{k}-h_{k} <0.0065d^{-1}k^{-1}$ for $2 \leq k \leq e^{d/2}d^{-1}$. It is
immediate from \eqref{dhk} that $dkh_k < 0.5$, 
 so we conclude that
$dkt_{k} =  dk(t_k - h_k) + dkh_k < 0.5065$ for $2 \leq k \leq
e^{d/2}d^{-1}$,  as required.
\end{proof}

\subsection{Bounding the coefficients of \texorpdfstring{$R_{\UU,
      b}(q, u)$}{}    and \texorpdfstring{$F_{\UU, b}(q,u)$}{}}
\label{sub:RUb}

Recall from Definitions~\ref{def:ru} and \ref{def:Rb} that $R_\UU(q,u)=\sum_{n=0}^{\infty
}r_{\UU}(2n,q)u^{n}$ and $R_{\UU, b}%
(q,u)=\sum_{n=0}^{\infty}r_{\UU, b}(2n,q)u^{n}$. We now prove
a lower bound on the coefficients $r_{\UU, b}(2n, q)$, provided that
$n$ is not too large.

\begin{lemma}\label{L:three_bound}
For all $b \geq 1$, 
  $0 \ll R_{\UU}(3, u) \ll R_{\UU}(q, u)$ and  
$0 \ll R_{\UU, b}(3, u) \ll R_{\UU, b}(q, u)$. 
Furthermore, for $b \geq 2$, 
$R_{\UU, b-1}(3, u) \ll R_{\UU, b}(3, u)$.
\end{lemma}

\begin{proof}
First we claim that $0 \ll G_{\UU, 0}(3, u) \ll G_{\UU,
  0}(q, u)$. 
Recall (\ref{defG0}).
%$$
%G_{\UU, 0} (q,u):=S_0(u)   \prod_{r \geq 1} \left( 1+ \frac{u^{r}}{q^{r} - 1} \right)^{\frac{1}{2}N^*(q^2, r)}
%\prod_{r \geq 1} \left( 1+ \frac{u^{r}}{q^{r} + 1} \right)^{N^\sim(q,
%  r)},
%$$
%where $S_0(u) = (1 + \frac{u}{q-1})^{-1}(q + \frac{u}{q+1})^{-1}$.
From Theorem~\ref{T:Nformulae} we find that
$N^*(q^2, 1) = 2$ and $N^\sim(q, 1) = q+1$, and so
\[
 G_{\UU, 0}(q, u) =  
\left(1+\frac{u}{q+1}\right)^{q-2}
\prod_{r \geq 2} \left( 1+ \frac{u^{r}}{q^{r} - 1} \right)^{\frac{1}{2}N^*(q^2, r)}
\prod_{r \geq 2} \left( 1+ \frac{u^{r}}{q^{r} + 1} \right)^{N^\sim(q,
  r)}.
\]
It is clear that $0\ll (1+\frac{u}{3+1})^{3-2} \ll (1 + \frac{u}{q
  +1})^{q-2}$, so consider next the $r$th term of the first
infinite product. %By Lemma~\ref{L:DquotBds}(ii)
%$$\frac{N^*(q^2, r)}{q^r-1} \geq \frac{N^*(3^2, r)}{3^r-1}.$$
Since $\frac{1}{2}N^*(q^2, r) = A(q, r)$ counts certain polynomials
over $\F_{q^2}$, it is an integer, and so 
it follows from Lemma~\ref{L:DquotBds}(ii) and \cite[Lemma 3.1]{DPS}
(with $N = \frac{1}{2}N^*(3^2, r)$, $M =
\frac{1}{2}N^*(q^2, r)$, $a = (q^r-1)^{-1}$ and $b = (3^r - 1)^{-1}$) that 
$$0 \ll \left( 1+ \frac{u^{r}}{3^{r} - 1} \right)^{\frac{1}{2}N^*(3^2, r)}
\ll \left( 1+ \frac{u^{r}}{q^{r} - 1} \right)^{\frac{1}{2}N^*(q^2,
  r)}$$
for all $r  >1$.
Next consider the $r$th term of the second infinite
product. 
%By Lemma~\ref{L:DquotBds}(iii), 
%$$\frac{N^\sim(q, r)}{q^r+1} \geq \frac{N^\sim(3, r)}{q^r+1}$$ for all
%$r \geq 1$, so 
As in the previous paragraph,  from Lemma~\ref{L:DquotBds}(iii) and
\cite[Lemma 3.1]{DPS}
we deduce that
$$0 \ll \left( 1 + \frac{u^{r}}{3^{r} + 1} \right)^{N^\sim(3,
  r)} \ll \left( 1+ \frac{u^{r}}{q^{r} + 1} \right)^{N^\sim(q,
  r)}.$$
The claim now follows by multiplying all of these terms together.

Now we claim that $0 \ll G_{\UU, b}(3, u) \ll G_{\UU,
  b}(q, u)$ for $b \geq 1$. This follows for all $m\geq 1$ from \eqref{defGb},
Lemma~\ref{L:DquotBds}(i) and \cite[Lemma 3.1]{DPS}:
\[\left(  1+\frac{u^{2m}}{3^{2m}-1}\right)
^{D(3, m)} \ll \left(  1+\frac{u^{2m}}{q^{2m}-1}\right)
^{D(q, m)}.
\]

Now we prove the lemma. 
For the first two bounds on $R_\UU(3, u)$ and $R_{\UU, b}(3, u)$, recall that
 $R_{\UU}(q, u) = \prod_{b = 0}^\infty G_{\UU, b}(q, u)$, and 
$R_{\UU, b}(q, u) = \prod_{k = 0}^{b-1} G_{\UU, k}(q, u)$, so the results
follow immediately from the bounds on $G_{\UU, b}(3, u)$. 

The final bound follows from noting that $R_{\UU, b}(3, u) = R_{\UU,
  b-1}(3, u) G_{\UU, b-1}(3, u)$, and that $G_{\UU, b-1}(3, u)$ is a
power series with non-negative coefficients and constant term $1$. 
\end{proof}

\begin{lemma}\label{L:coeffsR_b(q u)}
Let $b\geq3$ and $d=2^{b}$. Then $r_{\UU, b}%
(2n,q)> 0.247$ for all $n\leq e^{d/2}$.
\end{lemma}

\begin{proof}
The proof of this lemma is similar to that of \cite[Lemma 4.5]{DPS},
so we indicate only the relevant earlier results. 
By  Lemma~\ref{L:three_bound}, we may assume that $q=3$. 
%From Definition~\ref{def:Rb} we see that  $r_{\UU, b}(0, 3) = 1$.
The values of $r_{\UU, 3}(2n,3)$ for $1 \leq  n <  24$ may be 
computed: they
are all at least $0.25$. 
Lemma~\ref{L:three_bound} then shows that if $b\geq3$ then
$r_{\UU, b}(2n,3)\geq r_{\UU, 3}(2n,3) \geq 0.25 > 0.247$ for all  $n<24$.

Hence, using Lemma~\ref{L:three_bound},  it suffices to show that
  $r_{\UU, b}(2n,3)>0.247$ for each consecutive $b$, and $n$ in the range
$3\cdot 2^b \leq n\leq e^{2^{b-1}}$. 
Using \eqref{defTb2} and \eqref{defTU}, and setting  $k_0 =
\left\lfloor n/d\right\rfloor \geq 3$,
%\[
%R_{\UU, b}(q,u)=R_{\UU}(q,u)T_{\UU, b}(q,u)^{-1} = R_{\UU}(q,
%u)(1-T_{\UU}(u^d))=\left (\sum_{n = 0}^\infty r_{\UU}(2n, q)u^n \right) \left(1-\sum_{k=1}^{\infty}t_{k}u^{dk}\right).
%\] 
%Hence, setting $k_0 =  \left\lfloor n/d\right\rfloor$,
 we get
$r_{\UU, b}(2n,3) %&= r_{\UU}(2n, 3)(1 - \sum_{k = 1}^\infty t_k u^{kd}) \\
=r_{\UU}(2n,3)-\sum_{1\leq k\leq k_{0}}r_\UU(2(n - k d),3)t_{k}$. 

%Our assumption that $n \geq 3d$ implies that
% $k_{0}\geq3$. 
%If $k<k_{0}$ then $n-kd\geq d \geq8$, and so, by
Using Theorem~\ref{T:R_bounds} in place of \cite[Lemma 4.1]{DPS}, we
deduce that 
$r_{\UU}(2n,3)> 0.3433$ for $n\geq d$; 
$r_\UU(2(n-k_{0}d),3)\leq1$; and $r_\UU(2(n-kd),3)\leq 0.3795$ for $1
\leq k \leq k_{0}-1$.
Since $kdt_{k}\leq 0.5065$ for all $k$ such that $dk \leq e^{d/2}$ by 
Lemma~\ref{L: T(U) coeff}, 
 we proceed as in the proof of \cite[Lemma 4.5]{DPS},
but with $(0.3433, 0.5065, 0.3795)$ in place of $(0.4346, 1.02,
0.4543)$, to deduce that 
%\begin{align*}
%r_{\UU, b}(2n,3) &  \geq 0.3433 -\frac{0.5065}{d}\left\{  \frac{1}{k_{0}}%
%+ 0.3795 \sum_{1\leq k\leq k_{0}-1}\frac{1}{k}\right\}  \\
%&  \geq 0.3433 -\frac{0.5065}{d}\left\{  \frac{1}{k_{0}}+ 0.3795\left(  \log
%k_{0}+1\right)  \right\}  \text{.}%
%\end{align*}
%However $\log k_{0}+1\leq\log(n/d)+1\leq\log n-\log d+1\leq\log n-1\leq d/2-1$,
%by the hypothesis on $n$, and $1/k_{0}\leq1/3 < 0.3795$. Therefore 
$r_{\UU, b}(2n,3)\geq 0.3433 - 0.5065 \cdot 0.3795/2 > 0.247$. 
\end{proof}

Finally, we find a lower  bound for certain coefficients of $F_{\UU, b}(q,u)$. 
Setting $b\geq3$, $d:=2^{b}\geq8$, $U=u^{d}$ and
$Q=q^{d}$, 
\eqref{E: FF} becomes

\[
F_{\UU, b}(q,u) = F_b(U) := 
\prod_{m~\text{odd}}\left(  1+\frac{U^{m}}{Q^{m}-1}\right)  ^{\left\lceil
\frac{1}{8}N(q^2,md/2)\right\rceil }\text{.}%
\]
%We shall write $F_{\UU, b}(q,u)=\sum_{n\geq0}f_{\UU, b}(2n,q)u^n$.

Recall from Lemma~\ref{lem:akl}(v) that 
$[u^{n}]F_{\UU, b}(q,u)\left\vert \gu{2n}{q}\right\vert $ is a lower
bound on the number of pairs $(t,y) \in \Delta_{\UU}(2n,q)$ such 
that the $2$-part of the
order of each eigenvalue of $y$ is $2^{b-1}(q^2-1)_2$.

\begin{lemma}\label{L: F(U) coeffs}
Assume $b\geq3$, so $d=2^{b}\geq8$. Then for all $k$ such that $kd\leq e^{d/2}$
\[
f_{\UU, b}(2dk,q) = [U^{k}]F_{b}(U)\geq0.2117d^{-1}k^{-1}\text{.}
\] 
%for all $k$ such that $kd\leq e^{d/2}$.
%Furthermore, if $f_{\UU,
%  b}(2n,q)\ne 0$ then
%$n=2^bk$ for some integer $k$ and
%$f_{\UU, b}(2n,q)= [U^k] F_{b}(U)$.
\end{lemma}

The proof of this  lemma is almost identical to
that of \cite[Lemma 4.6]{DPS}, and so is omitted. To see why these
proofs are equivalent, notice that in \cite{DPS} the exponent of the
$m$th term in the infinite product for $F_{b}(q, u)$ is 
$\lceil N(q, md)/4 \rceil$, and
the bounds  \[\frac{N(q, md)}{4}\geq \frac{(q^{md} - 2q^{md})}{4md} \quad
  \mbox{and} \quad \frac{N(q,
md)}{4} \geq \frac{0.956 (q^{md} - 1)}{4md}\] are used. Our exponent is
$N(q^2, md/2)/8$, and Lemma~\ref{L:bds_N}(i) gives
\[\frac{N(q^2, md/2)}{8}\geq \frac{2 (q^{md} - 2q^{md})}{8md} \quad
  \mbox{and} \quad
\frac{N(q^2, md/2)}{8} > \frac{2 \times 0.956 (q^{md} - 1)}{8md}\] for $md/2 \geq
5$. It is not hard to find an equivalent bound when $md/2 = 4$. 
Notice also that the assumption that $k$ is odd in \cite[Lemma
4.6]{DPS} is unnecessary.

\section{Proof of Theorem~\ref{main}}\label{sec:thm_main}

\begin{defn}
Suppose that $0\leq\alpha
<\beta\leq1$, and let $J_\UU(2m,q;\alpha,\beta)$ be the set of all $(t,y)\in
\Delta_\UU(2m,q)$ for which $\mathrm{inv}(y)$ is $(\alpha,\beta)$-balanced. Set
\begin{equation}\label{defj} 
j_\UU(2m,q;\alpha,\beta):=\left\vert J_\UU(2m,q;\alpha,\beta)\right\vert /\left\vert
\gu{2m}{q}\right\vert .
\end{equation}
\end{defn}

\begin{defn}\label{def:f_trunc}
For $0 \leq \alpha < \beta \leq 1$ and $b\geq 1$, 
let  $F_{\UU,b}(q,u;m(1-\beta),m(1-\alpha))$ be the truncated power series
obtained from $F_{\UU,b}(q,u)=\sum_{k=0}^{\infty}f_{\UU,b}(2^{b+1}k,q)u^{2^bk}$ by keeping only the
terms $f_{\UU,b}(2^{b+1}k,q)u^{2^bk}$ for which $m(1-\beta)\leq 2^bk\leq m(1-\alpha)$.
\end{defn}

\begin{lemma}\label{L: j bound}
Fix $m > 0$, and let $0 \leq \alpha < \beta \leq 1$. Then  
\[
j_\UU(2m,q;\alpha,\beta)\geq\lbrack u^{m}]\sum_{b=2}^{\infty}R_{\UU,b}(q,u)F_{\UU,b}%
(q,u;m(1-\beta),m(1-\alpha)).
\]
\end{lemma}

\begin{proof}
If $c_y(X)\in \Pi_\UU(2m,q)$ is the characteristic polynomial for $y$ 
and $c^-_y(X)$ is the polynomial as in Lemma~\ref{lem:akl}, then 
by Lemma~\ref{lem:akl}(iii),  $\inv(y)$ has $(-1)$-eigenspace of
dimension $\deg c_y^-(X)$, and so $\inv(y)$ is 
$(\alpha,\beta)$-balanced if and only if $2m(1-\beta)\leq\deg c_y^-(X)\leq
2m(1-\alpha)$. 
It follows from Lemma~\ref{lem:akl}
that we may bound  $j_\UU(2m,q;\alpha,\beta)$
by summing the coefficients of $u^mz^\ell$ in the power series
$R_{\UU,b}( q, u) F_{\UU,b}(q, uz)$, over $b > 1$, and over $\ell$ in the range 
$[m (1 - \beta), m (1 - \alpha)]$ (and we recall that non-zero summands
in $F_{\UU,b}(q,uz)=\sum_{\ell\geq0} f_{\UU,b}(2\ell,q)(uz)^\ell$ 
occur only if $2^b$ divides $\ell$). 
\end{proof}

\begin{lemma}
\label{L:harmonic} Let $a, c $ be real with $0<a<c$. Then%
\[
\sum_{\substack{a\leq k\leq c, \, k\in \Z}}\frac{1}{k}\geq
\log\left(\frac{c}{a}\right)-\frac{1}{a}.
\]
\end{lemma}

\begin{proof}
The proof is similar to that of \cite[Lemma
3.8]{DPS}, so we only sketch the details. 
Each non-negative integer $\ell$ satisfies
$1/\ell\geq \int_{\ell}^{\ell+ 1}x^{-1}dx$. 
Set $\ell_{0}:=\left\lceil a \right\rceil $ and $\ell_{1}:=\left\lfloor
c \right\rfloor $. The sum in question equals%
\begin{align*}
\sum_{\ell=\ell_{0}}^{\ell_{1}}\frac{1}{\ell}  &  \geq%
\int_{\ell_{0}}^{\ell_{1} + 1}x^{-1}dx\geq \int_{\ell_{0}}%
^{c}x^{-1}dx = \log\left(\frac{c}{a}\right)-\log\left(  \frac{\ell_{0}}%
{a}\right)  \text{.}%
\end{align*}
From $\ell_{o}/a <1+1/a$ we deduce that
$\log\left( \ell_{0}/a \right)<1/a$, 
so the required inequality follows.
\end{proof}

\begin{lemma}\label{L: j unequal}
Let $b\geq3$ and $d:=2^{b}$. Then for all $\alpha$ and
$\beta$ with $0\leq\alpha<\beta<1$ and positive integers $m\leq e^{d/2}$, 
\[
j_\UU(2m,q;\alpha,\beta)\geq\frac{0.05228}{d}\left(  \log\left(  \frac{1-\alpha
}{1-\beta}\right)  -\frac{d}{m(1-\beta)}\right).
\]
\end{lemma}

\begin{proof}
We mimic the proof of \cite[Lemma 5.1]{DPS}. First we use Lemma \ref{L:
  F(U) coeffs} bounding $f_{\UU, b}(2dk, q) \geq 0.2117/dk$ in place of their Lemma 4.6.
%The coefficient $f_{\UU,b}(2dk,q)$
%of $F_{\UU,b}(q,u)$ is $[U^k]F_{\UU,b}(U)$, and  by Lemma \ref{L: F(U) coeffs}
%for all $k$ such that $dk\leq e^{d/2}$, 
%we have $f_{\UU,b}(2dk,q)\geq 0.2117/dk$. 
%Set $a:=m(1-\alpha)$ and $c:=m(1-\beta)$, and
%consider the sum $s$ of the coefficients of $F_{\UU,b}(q,u)$ for terms with degrees between $c$ and
%$a$. 
Next, we use Lemma \ref{lem:akl}(iv) to see that  if
 $f_{\UU,b}(2n,q)\ne 0$ then $n=dk$ for some $k$.
%By Lemma \ref{lem:akl}(iv), if
% $f_{\UU,b}(2n,q)\ne 0$ then $n=dk$ for some $k$. Thus,  
We let $s$ be the sum of the coefficients of $F_{\UU,b}(q,u)$ for
terms with degrees between 
$m(1-\beta)$ and
$m(1-\alpha)$. 
We then use Lemma~\ref{L:harmonic}, Lemma \ref{L:coeffsR_b(q u)} and
Lemma~\ref{L: j bound} to see that
$j_\UU     (2m,q;\alpha,\beta)\geq0.247\,s$, so the result follows.
\end{proof}

\begin{defn}\label{def:l}
Let 
\[
\ell_\UU(n,s,q;\alpha,\beta):=\left\vert L_\UU(n,s,q;\alpha,\beta)\right\vert
/\left\vert K_{\UU,s}\right\vert ^{2}
\]
be the proportion of pairs in $K_{\UU,s}\times
K_{\UU, s}$ which lie in $L_\UU(n,s,q;\alpha,\beta)$ (as in Definition~\ref{def:L}). 
\end{defn}

We define
\begin{equation}\label{eqn:phidef}
\varphi_{\UU}(m, z) = \prod_{i = 1}^m (1 - (-1)^i z^{-i}), \mbox{ for } z>1 \mbox{ and }
m \in \mathbb{N}
\end{equation}
and note that $|\gu{m}{q}|=q^{m^2}\varphi_\UU(m,q)$.

\begin{lemma}\label{L: L-val}
Let $2n/3 \geq s\geq n/2\geq 2$, let $j_\UU(2n-2s,q;\alpha,\beta)$ be
as in \eqref{defj}, and let 
\[
\theta(n,s,q)=\frac{\varphi_{\UU}(n-s,q)^{2}\varphi_{\UU}(s,q)^{2}}{\varphi_{\UU}(n,q)\varphi_{\UU}
(2s-n,q)}.
\]
Then
$\ell_\UU(n,s,q;\alpha,\beta)=\theta(n,s,q)j_\UU(2n-2s,q;\alpha,\beta)$, 
and $\theta(n,s,q) > \frac{81}{98}$.
\end{lemma}

\begin{proof}
For the main claim, we mimic the proof of \cite[Lemma 5.3]{DPS}, but
modify
to count
decompositions into 
non-degenerate unitary subspaces. 
Let $h:= 2s-n$, and let $\Omega_\UU$ be the set of pairs $(V_{1},V_{2})$ of non-degenerate subspaces of  $V$, of dimensions $h, n-h$ respectively, such that $V_1^\perp=V_2$. Then for each $(V_{1},V_{2})\in\Omega_\UU$, and each 
 $(t_2,y_2)\in \Delta_\UU(n-h,q)$ acting on $V_{2}$ such that $\mathrm{inv}(y_{2})$ 
is $(\alpha,\beta)$-balanced, there is (see Lemma~\ref{lem:delta_right}) a unique pair $(t,t^\prime)$ of involutions
in $\gu{n}{q}$ 
such that $t_{|V_{2}}=t_{2}$, $t_{|V_{2}}^{\prime}=t_{2}y_2$, and $t_{|V_{1}}=
t_{|V_{1}}^{\prime}=I$. %Moreover, since each of $t_{2}$ and
%$t_{2}y_2$ conjugates $y_2$ to its inverse, 
%Each of $t_{2}$ and
%$t_{2}y_2$  of
%type $(\frac{1}{2}(n-h),\frac{1}{2}(n-h))$ 
It follows from Definition~\ref{def:L} and \cite[Lemma~2.2]{DPS} that 
$(t,t^\prime)\in  L_\UU(n,s,q;\alpha,\beta)$.   

Conversely, for each $(t,t^\prime)\in  L_\UU(n,s,q;\alpha,\beta)$, relative to $V_1, V_2$
as in Definition~\ref{def:L}, the pair $(t_{|V_{2}},tt_{|V_{2}}^{\prime})\in\Delta_\UU(n-h,q)$.
%Hence $\left\vert L_\UU(n,s,q;\alpha
%,\beta)\right\vert =\left\vert \Omega_\UU\right\vert \times$ $\left\vert J_\UU(n-h,q;\alpha
%,\beta)\right\vert $. 
Thus by \eqref{defj} and Definition~\ref{def:l}, we conclude that 
\[\ell_\UU(n,s,q;\alpha,\beta)= \left\vert \Omega_\UU\right\vert \left\vert
\gu{n-h}{q}\right\vert  j_\UU(n-h,q;\alpha,\beta)/\left\vert K_{\UU,s}\right\vert ^{2}.\] 
Using the obvious expressions for $|\Omega_{\UU}|$ and $\left\vert
  K_{\UU,s}\right\vert$, and recalling from \eqref{eqn:phidef} that
$|\gu{n}{q}|=q^{n^2}\varphi_\UU(n,q)$, 
%Then since
%\[
%\left\vert
%\Omega_\UU\right\vert = \frac{\left\vert \gu{n}{q}\right\vert }{ \left\vert
%\gu{h}{q}\right\vert \left\vert \gu{n-h}{q}\right\vert } \ \mbox{and}
%\left\vert K_{\UU,s}\right\vert %
%= \frac{\left\vert \gu{n}{q}\right\vert }{ \left\vert
%\gu{s}{q}\right\vert \left\vert \gu{n-s}{q}\right\vert }
%\] 
we obtain   the
expression for $\theta(n,s,q)$ given in the statement.

To prove that $\theta(n,s,q) > 81/98$, we use \cite{PS11}. 
For $1\leq k\leq n, 1\leq r<n$, define
\[
\Omega(k,n;-q) :=\prod_{i=k}^n(1 - (-q)^{-i})\quad \mbox{and}\quad 
\Delta(r,n;-q) :=\frac{\Omega(1,n;-q)}{\Omega(1,r;-q) \Omega(1,n-r;-q)}
\]
so that (noting $2s-n\leq n/3$ by assumption),
$\theta(n,s,q) = \Omega(2s-n+1,n;-q)/\Delta(s,n;-q)^2$. 
%where  if $s=n$ we interpret $\Omega(2s-n+1,n;-q)=1$.
By \cite[Lemma 3.2(b)]{PS11}, $\Delta(s,n;-q)$ is less than 1 if $s$ is odd and less than $28/27$ if $s$ is even. 
Since $s < n$, it follows from \cite[Lemma 3.2(a)]{PS11} that 
$\Omega(2s-n+1,n;-q)$ is greater than $1$ if $n$ is even and greater than $1-q^{-2s+n-1}\geq 8/9$ 
if $n$ is odd. Hence $\theta(n, s, q) > \left( \frac{27}{28} \right)^2
\left( \frac{8}{9} \right)$. 
%
%These bounds yield immediately that $\theta(n,s,q) > 2/3$ unless $n$ is odd and $s$ is even.
%In this remaining case, $s\geq (n+1)/2$ so $\Omega(2s-n+1,n;-q)\geq 1-q^{-2s+n-1}\geq 8/9$,
%and hence $\theta(n,s,q) > (27/28)^2(8/9) > 2/3$.
\end{proof}

\begin{proof}
[Proof of Theorem~\ref{main}] 
We shall prove this theorem  first for uniformly distributed random
elements  of
$\gu{n}{q}$, then generalise
to nearly uniformly distributed elements.

Let $t$ be a strong involution in $\gu{n}{q}.$ Then $t$ is
$(1/3,2/3)$-balanced and so is of type $(s,n-s)$ with $n/3\leq s\leq2n/3$. We
claim that it is enough to consider the case where $s\geq n/2$. Indeed, if
$s<n/2$, then $-t$ is a strong involution in $\gu{n}{q}$ of type $(n-s,s)$ with $n-s>n/2$, and
since $(-t)(-t^{g})=tt^{g}$ for all $g \in \gu{n}{q}$ the value of $z(g):=\mathrm{inv}(tt^{g})$ is
unchanged. Thus by replacing $t$ by $-t$ where necessary, we assume for the
rest of this proof that $n/2 \leq s\leq 2n/3$. 
%Set $h:=2s-n$ and note that $0\leq h\leq n/3$.

\medskip

\noindent (i) Let $K_{\UU,s}$ be the conjugacy class of $t$ in
$\gu{n}{q}$, that is, the set of involutions of type $(s,n-s)$, and let $\pi_{+}$  
be the probability that, for a random $g\in\gu{n}{q}$, the restriction of 
$\inv(tt^g)$  to $E_+(t)$  is $(1/3,2/3)$-balanced. Now $\pi_+$ is independent 
of the choice of $t$ in $K_{\UU,s}$. A straightforward counting argument shows that 
$\pi_{+}$ is equal to the proportion of $(t,t^{\prime})\in
K_{\UU,s}\times K_{\UU,s}$ such that the restriction of
$\inv(tt^{\prime})$ to $E_{+}(t)$ is $(1/3,2/3)$-balanced. 

If we choose $\alpha$ and $\beta$ as in Lemma~\ref{L:parta},
we see from Lemma~\ref{L:parta}(ii)  that 
the restriction of
$\inv(tt^{\prime})$ to $E_{+}(t)$ is $(1/3,2/3)$-balanced
whenever $(t,t^{\prime})\in L_\UU(n,s,q;\alpha
,\beta)$. It is immediate from Definition~\ref{def:l} that
$\pi_{+}\geq \ell_\UU(n,s,q;\alpha,\beta)$.

We now find $\kappa$ and $n_{0}$ such that
$\ell_\UU(n,s,q;\alpha,\beta)\geq\kappa/\log n$ for
all $n\geq n_{0}$. Suppose that $n>e^{4}$,
or equivalently, that $n>54$. 
Then there exists a unique $b\geq4$ such that $2^{b-2}< \log n \leq 2^{b-1}$,
or equivalently, setting $d:=2^b$, $n$ satisfies $e^{d/4}< n \leq e^{d/2}$. 
Thus the conditions of Lemma~\ref{L: j unequal} hold with $m=n-s< e^{d/2}$, 
and hence
%We start by proving Theorem~\ref{main}~(i). 

\begin{equation}\label{E:j near final}
j_\UU(2(n-s),q;\alpha,\beta)\geq\frac{0.05228}{d}\left(  \log\left(  \frac{1-\alpha
}{1-\beta}\right)  -\frac{d}{(n-s)(1-\beta)}\right)  \text{.}%
\end{equation}
Using the definitions of $\alpha$ and $\beta$  from Lemma~\ref{L:parta}, 
we first deduce that 
\[(n-s)(1-\beta)=(n-s)\frac{s}{3(n-s)}= \frac{s}{3}\geq\frac{n}{6}\]
so that $1/((n-s)(1-\beta)) \leq 6/n$. We also see that 
\[
\frac{1-\alpha}{1-\beta}=\left\{
\begin{array}
[c]{cl}%
2			&\text{ if }n/2\leq s\leq 3n/5\\
3(n-s)/s	&\text{ if } 3n/5 \leq s \leq 2n/3,
\end{array}
\right.  
\]
which implies that $\log\left(  (1-\alpha)/(1-\beta)\right)
\geq\log3/2>0.4054$. Substituting into (\ref{E:j near final}) we get 
\[
j_\UU(2(n-s),q;\alpha,\beta)\geq 0.05228 d^{-1}\left(  0.4054- 6d/n\right)
=0.05228\left(0.4054/d -6/n\right),
\]
so let $\zeta_{1}(n,d) =0.05228\left(0.4054/d -6/n\right)$. 
Elementary calculus shows that $\zeta_{1}(n,d)\log n$ increases with $n$, for fixed
$d>0$ and $n\geq3$. First consider  $b=4$ (so $54=\left\lfloor e^{4}\right\rfloor <n\leq2980=\left\lfloor
e^{8}\right\rfloor $). Since $\zeta_{1}(250,16)\log250>0.0003$, we have
$\zeta_{1}(n,16)\log n>0.0003$ for all $n\geq 250$. 
Conversely, when
$b\geq5$, $d:=2^{b}$ and $e^{d/4}<n\leq e^{d/2}$, 
\[d/n<e^{-d/4}%
d\leq32e^{-8}<0.01074 \, \textup{ and } \, \log n>d/4.\] 
Thus 
\[\zeta_{1}(n,d)\log n>0.05228d^{-1}\cdot (0.4054
-6 \cdot 0.01074) \cdot d/4
> 0.0044\] in this
case. Hence $j_\UU(2(n-s),q;\alpha,\beta)>0.0003/(\log n)$ holds for all
$n\geq 250$. Finally, applying Lemma~\ref{L: L-val},  for all $n\geq150$
\begin{align*}
\pi_{+}
&
  \geq\ell_\UU(n,s,q;\alpha,\beta)=\theta(n,s,q)j_\UU(n-h,q,\alpha,\beta) \\
& >(81/98) \times 0.0003/\log n>0.0002/\log n\text{.}%
\end{align*}
This proves (i) with $n_0=150$ and $\kappa =
0.0002$, for uniformly distributed random elements.

The proof of Part (ii) is similar. We take $\alpha=1/3$ and
$\beta=2/3$ since Lemma~\ref{L:parta}~(ii) shows that $z_{|V_{2}}$ is of type
$(2k_+, 2k_-)$ and $z_{|E_{-}(t)}$ is of type $(k_+, k_-)$ so the former is
strong exactly when the latter is. We make a similar
estimate using  Lemma \ref{L: j
  unequal} for
$j_\UU(2(n-s),q;1/3,2/3)$ (noting that now $1/((n-s)(1-\beta))=3/(n-s)\leq
9/n$). This
  shows that, for all $n\geq 250$, and $d$ chosen as the
power of $2$ such that $e^{d/4}<n\leq e^{d/2}$
\[
j_\UU(2(n-s),q;1/3, 2/3)\geq 0.05228 d^{-1}(\log2- 9d/n)=\zeta
_{2}(n,d)\text{, say.}%
\]

We calculate that $\zeta_{2}(250,16)\log250>0.00211$ and hence
$\zeta_{2}(n,16)\log n>0.00211$ for $n\geq250.$ Furthermore an argument
similar to the one above shows that if $b\geq5$ and $e^{d/4}\leq n\leq
e^{d/2}$ then 
\[\zeta_{2}(n,d)\log n>\frac{0.05228}{d}(\log 2 - 9\times 
0.01074)\frac{d}{4}> 0.0077.\] Hence for all $n\geq250$, we get
$
j_\UU(2(n-s),q; 1/3, 2/3)> 0.00211/\log n%
$
and so
\[
\pi_{-}\geq\ell_\UU(n,s,q; 1/3, 2/3)> (81/98) 0.00211/\log n>
0.00174/\log n.
\]
This proves (ii) with $n_0=250$ and $\kappa=0.0017$, and thus completes the proof of Theorem~\ref{main}  with
$n_{0}=250$ and $\kappa =0.0002$, for uniformly distributed random
elements.

For nearly uniformly distributed random elements $g$ of  $G$, 
let $\pi_+$ be the probability that 
the restriction of $\inv(tt^g)$ to $E_+(t)$ is $(1/3,
2/3)$-balanced. It follows from Definition~\ref{def:nearunif} that $\pi_+ >
\ell_{\UU}(n, s, q; \alpha, \beta)/2$. The rest of the proof
follows as before, but the final value of $\kappa$ is halved. The
argument for $\pi_-$ is similar. 
\end{proof}

\section{Proofs of remaining main theorems}\label{sec:final_proofs}

\subsection{Proof of Theorems~\ref{thm:gen_cent} and \ref{thm:gen_cent2}}

Before proving Theorem~\ref{thm:gen_cent}, we give a lemma which
reduces the problem to proving the result for uniform distributions.

\begin{lemma}\label{lem:nearly_unif}
Let $X$ be a nearly uniform random variable on a finite group $G$. Let
$Y$ be the results of three independent trials. Then
$\mathbb{P}( g \in Y) \geq 1/|G|$ for all $g \in G$. 
\end{lemma}

\begin{proof}
  The result is trivially true if $|G|=1$. Suppose now that $|G|\geq2$.
  By definition of nearly uniform, 
$\mathbb{P}(X = g) > \rho$ where $\rho= 1/(2|G|)$. Then
\[
\mathbb{P}(g \not\in Y)\kern-1pt  =\kern-1pt  (1 - \mathbb{P}(X = g))^3\kern-1pt < \kern-1pt (1-\rho)^3\textup{ so }
\mathbb{P}(g \in Y)\kern-1pt > \kern-1pt 3\rho-3\rho^2+\rho^3\kern-1pt =\kern-1pt 2\rho+\rho\left(1-3\rho+\rho^2\right).
\]
However, $0<\rho<(3-\sqrt{5})/2$ as $|G|\geq2$, so
$1-3\rho+\rho^2>0$ and  $\mathbb{P}(g \in Y)>2\rho=1/|G|$.
\end{proof}

%The following definition is from \cite{PS11}. 

\begin{defn} (See \cite{PS11}). 
Let $H$ be a group, and let $\mathcal{H} = (\mathcal{C}_1,\ldots, \mathcal{C}_c)$ be a
sequence of conjugacy classes of $H$. A $c$-tuple $(h_1, \ldots,
h_c)$ is a \emph{class-random sequence from $\mathcal{H}$} if $h_i$ is
  a uniformly distributed random element of $\mathcal{C}_i$ for all $i$, and the
  $h_i$ are independent.
\end{defn}

%\begin{remark}\label{rem:class_random}
%As noted in \cite{Bray00}, if $g$ is a uniformly distributed random
%element of $G$, and $t \in G$ is an involution, then the element
%$\inv(tt^g)$ is 
%$C_G(t)$-class-random. %For let $h \in C_G(t)$, and let
%$\inv(tt^g) = (tt^g)^m$. Then $h^{-1}(tt^g)^mh = (t^h t^{gh})^m = (t
%t^{gh})^m = (t t^{g^h})^m$.
%\end{remark}

\begin{proof}[Proof of Theorem~\ref{thm:gen_cent}]
By Lemma~\ref{lem:nearly_unif}, it suffices to prove the theorem for
uniform random elements. We first consider $G =
\gu{n}{q}$, and address $\gl{n}{q}$ at the end.

Let $V_\eps=E_\eps(t)$ for $\eps \in \{+,-\}$. 
We construct involutions $\inv(tt^g)$ which have determinant $1$ and hence lie in
 $\su{n}{q}$. However their restrictions $\inv(tt^g)|_{V_\eps}$ are guaranteed to
 lie only in the subgroup  $\su{}{V_\eps}.2$ of $\gu{}{V_\eps}$ consisting of
elements with determinant $\pm 1$. 

We shall now choose an  $n_1$, as in the statement of the theorem,  and then
show that  the
result holds for all $n \geq n_1$.  Let $\kappa$ and $n_0$ be as in Theorem~\ref{main}. 
 In \cite[Theorem 1.1]{PS11} it is shown that there exist constants $c$ and $n_2$ such that
for $\ell \geq n_2$ and for every sequence $\mathcal{H}$ of $c$ conjugacy classes
of strong involutions of $\su{\ell}{q}.2$,  a class-random sequence from 
$\mathcal{H}$ generates a group containing $\su{\ell}{q}$ with probability at least $1-
q^{-\ell}$.
We let $n_1 = \mathrm{max}\{3n_2, n_0\}$.

By Theorem~\ref{main}, since $n \geq n_1 \geq n_0$, the
probability that a sequence of $N=\lceil\kappa^{-1} \log n \rceil$ random elements $g$ do not produce at least one 
strong involution $\inv(tt^g)|_{V_+}$ and at least one strong involution  $\inv(tt^g)|_{V_-}$ is at most
\[
\left(1 - \frac{\kappa}{\log n}\right)^N
\leq\left(1 - \frac{1}{N}\right)^N<e^{-1}.
\]
%By Theorem~\ref{main}, the
%probability that a sequence of $\kappa^{-1} \log n $ random elements $g$
%produce no strong
%involutions $\inv(tt^g)|_{V_\eps}$ is at most $(1 - \kappa/\log n)^{\kappa^{-1}\log n}
%< 1 - 1 + 1/2 - 1/6 + 1/24 = 3/8$.
Let $m \in \mathbb{Z}$. Then the
probability that $m N$ random elements $g$ do not produce  
 at least one strong involution $\inv(tt^g)|_{V_+}$ and at least one strong involution  
 $\inv(tt^g)|_{V_-}$ is at most $e^{-m}\leq(3/8)^m$. This can be
made as small as required, by choosing $m$ sufficiently
large. Furthermore, each such strong involution
$\inv(tt^g)|_{V_{\eps}}$ is class-random in $\su{}{V_\eps}.2$.

We now define the sequence $A$ and constant $\lambda$ from the
statement of the theorem. The sequence $A$ will be thought of as the concatenation of three disjoint
subsequences, $A_+$, $A_-$ and $B$, and will have total length $\lceil \lambda \log n
\rceil$. 
The constant $\lambda 
> 0$ is chosen such that, with (combined) probability at least $0.9$, 
all of the following three independent events occur: (i) the subsequence $A_+$
contains at least $c$ elements
$g$ such that $\inv(tt^g)|_{V_+}$ is a
strong involution; (ii) the subsequence $A_-$ contains  at
 least $c$ elements $g$
such that $\inv(tt^g)|_{V_-}$ is a strong involution; (iii) the 
subsequence $B$ contains at least one
$g \in G$ such that $z = \inv(tt^g)$ is an additional  strong involution on
$V_+$. Assume now that all three of these events occur.

Let $s = \dim(V_+)$, so that $\dim(V_-)= n-s$. 
 Let $K_1 = \langle \inv(tt^g) \mid g \in A_+ \rangle$, and $K_2 =
\langle \inv(tt^g) \mid g \in A_- \rangle$. Set $K = \langle K_1, K_2
\rangle$ and $H =\langle \inv(tt^g) \mid g \in A \rangle$, so that 
$$
K \leq H \leq  \left(\su{s}{q} \times \su{n-s}{q}\right).2 \leq C_{\gu{n}{q}}(t)
$$ 
where $(\su{s}{q} \times \su{n-s}{q}).2=\su{n}{q}\cap
\left(\su{s}{q}.2 \times \su{n-s}{q}.2\right)$.

Since $t$ is a strong involution in $\gu{n}{q}$, and $n \geq n_1 \geq
3n_2$, 
 both $\dim(V_+)=s\geq n/3\geq n_2$ 
and $\dim(V_-)=n-s \geq n/3\geq n_2$.
It therefore
follows from \cite[Theorem 1.1]{PS11} 
that $\mathbb{P}(K_1 |_{V_+} \mbox{ contains } \su{}{V_+}) \geq
1-q^{-s}$, 
and independently
$\mathbb{P}(K_2 |_{V_-} \mbox{ contains } \su{}{V_{-}}) \geq 1-q^{-(n-s)}$. 
Hence the probability that both $K|_{V_+}\geq \su{}{V_+}$ and $K|_{V_-}\geq \su{}{V_-}$ is at least 
$(1-q^{-s})(1 - q^{-(n-s)})$, and since  this expression is
increasing as $s$ goes from $n/3$ to $n/2$, this probability is at least 
$1 - q^{-n/3} - q^{-2n/3} + q^{-n}$. Suppose then that both $K|_{V_+}\geq \su{}{V_+}$ and $K|_{V_-}\geq \su{}{V_-}$.

If $s \neq n/2$ then every subdirect subgroup of $\su{s}{q} \times
\su{n-s}{q}$ is the full direct product, and therefore $K$, and hence also $H$, contains $\su{s}{q} \times
\su{n-s}{q}$. Suppose now that $s = n/2$. We show that, with high probability, in this case also 
$H$  contains
$\su{n/2}{q} \times\su{n/2}{q}$. Suppose that $K$ does not contain
$\su{n/2}{q} \times \su{n/2}{q}$. Then $K \cong K|_{V_\eps}\cong 
\su{n/2}{q}$ or $\su{n/2}{q}.2$, and $K$ is  a diagonal subgroup of $K|_{V_+}\times K|_{V_-}$, with 
isomorphism $\phi:K|_{V_+}\rightarrow K|_{V_-}$.

Recall the element $z$ defined by the final subsequence $B$ of $A$. Let  $z_+ = z |_{V_+}$ and 
$z_- = z |_{V_-}$, so that $z_+$ is a strong involution on $V_+$, 
with $1$-eigenspace of dimension $a$ and $(-1)$-eigenspace of dimension $b$, where
$(1/3)(n/2)\leq a\leq(2/3)(n/2)$ and $a+b=n/2$.   If $H$ is also
 a diagonal subgroup of
$H|_{V_+}\times H|_{V_-}$ then $\phi$ naturally extends to $H|_{V_+}$,
and 
 $\phi(z_+)=z_-$. Hence $z_-$ acting on $V_-$ also has
$(\pm 1)$-eigenspaces of dimensions $a, b$, respectively. 
As we noted in the Introduction, $z$ is a uniformly distributed random element of its conjugacy class $\mathcal{C}$  in $C_G(t)=C_{\gu{n}{q}}(t)$,
and the members of $\mathcal{C}$ are elements $z'$ such that $z' |_{V_+}$,  
$z' |_{V_-}$ are $\gu{n/2}{q}$-conjugate to $z_+, z_-$ respectively. In particular, for a given $z_+=z|_{V_+}$, each element of the $\gu{n/2}{q}$-conjugacy class of 
$z_-$ would occur as $z|_{V_-}$  with equal probability (which we show is very small).  
Using \cite[Table 4]{PS11},  the $\gu{n/2}{q}$-conjugacy class of  $z_-$ has~size
$$
\frac{|\gu{n/2}{q}|}{|\gu{a}{q} \times \gu{b}{q}|} \geq
\frac{9}{16}q^{(n/2)^2 - a^2 - b^2} = \frac{9}{16}q^{2ab}  
\geq \frac{9}{16}q^{n^2/9}. 
$$
%Each conjugacy class in $\gu{n/2}{q}$ splits into at most $q+1 < q^2$ conjugacy classes in $H|_{V_-}$, so each $H|_{V_-}$-conjugacy class of strong involutions has size greater than $\frac{9}{16}q^{n^2/9 - 2}$. Since $z_+$ is a strong involution in $H|_{V_+}$, its image $\phi(z_+)$ is a strong involution in $H|_{V_-}$, and  hence
Hence $\mathbb{P}(z|_{V_-}=\phi(z_+)) < (16/9)q^{-n^2/9}$. 
Drawing everything together, $\mathbb{P}(H$ contains $\su{n/2}{q} \times \su{n/2}{q})$ is greater than
\[
0.9(1 - q^{-n/3} - q^{-2n/3} + q^{-n})(1 - \frac{16}{9}q^{-n^2/9}) > 0.9(1
- q^{-n/3} - q^{-2n/3})
\]
where we increase $n_1$ if necessary so that the final inequality holds.

For $\gl{n}{q}$, \cite[Theorem 1.1]{DPS} states a similar result to
our Theorem~\ref{main} for uniformly distributed random
elements. An argument identical to the final paragraph of our proof of
Theorem~\ref{main} upgrades \cite[Theorem 1.1]{DPS} to nearly
uniformly random elements, and then the remainder of the proof is
identical, but with linear groups in place of unitary groups. 
\end{proof}

\noindent {\sc Proof of Theorem~\ref{thm:gen_cent2}.}
By \cite[Theorem 1.1]{LNP}, there exists an absolute constant $c$ such that
$\mathbb{P}(g$ powers to a strong involution$) \geq c/\log n$, for $g$
a uniformly distributed random element of $G$. So there exists a
constant $\delta_1$ such that $\delta_1 \log n$ independent uniform
random elements suffice to produce such a strong involution 
$t$ with probability at least
$0.89/0.9$.  We first run this random process to produce such a $t$. 

For $n \geq n_1$, we now apply Theorem~\ref{thm:gen_cent},  with this
known strong involution $t$,  to see that a further $\lambda \log n$ random
elements of $G$ will suffice to produce generators for a subgroup of
$C_G(t)$ that contains the last term, $C_G(t)^\infty$, in the derived
series of $C_G(t)$. This
step succeeds with probability at least $0.9(1 - q^{-n/3} -
q^{-2n/3})$. Thus we set $\mu_1 = \delta_1 + \lambda$, to get 
that the  overall probability of success
in this case is at least
$(0.89/0.9) \cdot 0.9(1 - q^{-n/3} -
q^{-2n/3})$. 

Assume instead therefore that $n < n_1$. By \cite[Theorem 2]{ParkerWilson}, 
there is a positive constant $a$ such that if $t_1$ is any
involution in $G$, then the proportion of ordered pairs $(t_1,
t_1^g)$ such that $t_1t_1^g$  has odd order is bounded below by
$an^{-1}$. We set $t_1$ to be our known strong involution $t$.  If
$tt^g$ has odd order $2k+1$, then  $g[t, g]^k$ is a uniformly
distributed random element of $C_G(t)$ (see \cite[Theorem 3.1]{Bray00}). Since the probability that two random
elements of  $S = \mathrm{SL}_m(q)$ or $\su{m}{q}$ generate $S$ is greater than $1/2$
(see \cite[Theorem 1.1]{MQRD}), reasoning as in the case $n \geq n_1$, 
there is a constant $\delta_2$ such that if $A$ is a sequence of
$\delta_2 \log n$
uniformly distributed random
elements $g \in G$ then the probability that $\langle R(g, t) \mid g \in A \rangle$ contains
$C_G(t)^\infty$ is 
greater than $0.9(1 - q^{-n/3} - q^{2n/3})$. We therefore may set $\mu_2
= \delta_1 + \delta_2$ to get that the overall probability of success
in this case is at least
$0.89(1 - q^{-n/3} - q^{2n/3})$.

Finally, we set $\mu = \mathrm{max}\{\mu_1, \mu_2\}$, and the result
follows.  \hfill $\Box$

\subsection{Proof of Theorem~\ref{thm:iota_bound}}

%We now prove the assertions in Theorem~\ref{thm:iota_bound} concerning
%$\iota_\UU(n, q)$. 

The proof of the following lemma is similar to that of
\cite[Lemma 4.1(b)]{gl}. 
We assume that the Gram matrix of the unitary form is the
identity matrix. 
Recall \eqref{eqn:phidef}. We let 
$\varphi_{\UU}(z)=\lim_{m\rightarrow\infty}\varphi_{\UU}(m,z)$, and
define
 $\varphi_{\UU}(0,z)=1$.

\begin{lemma}\label{lem:iota_even_odd}
Let $\Phi_\UU(m,q) = \varphi_{\UU}(m, q)^4/\varphi_{\UU}(2m, q)$. Then
for $m \geq 1$, $\iota_\UU(2m, q)$ equals $\r_\UU(2m, q) \Phi_\UU(m,q)$, and 
\[
  \iota_\UU(2m+1, q) = \iota_{\UU}(2m, q) \frac{(1 - (-1)^{m+1}q^{-m - 1})^2}{(1 +
  q^{-2m -1})(1 + q^{-1})} = \r_{\UU}(2m, q) \frac{\varphi_{\UU}(m, q)^2 \varphi_{\UU}(m+1, q)^2}{(1 +
  1/q)\varphi_{\UU}(2m+1, q)}.
\]
\end{lemma}

\begin{proof}
First let $n = 2m$ be even, and let $x \in
\mathcal{C}_\UU(V)$ (see Definition~\ref{def:more_invols}). Then 
%$x$
%acts as the identity $I_m$ on a non-degenerate $m$-space, and as
%$-I_m$ on its orthogonal complement. Hence 
%\[
%  C_{\gu{2m}{q}}(x) \cong \gu{m}{q} \times \gu{m}{q}
%\]
%and so 
$|\mathcal{C}_{\UU}(V)| = |\gu{2m}{q}| \cdot |\gu{m}{q}|^{-2}$.
Hence, by \eqref{defIU}, 
%\[
%  \begin{array}{rll}
\begin{align*}
  \iota_\UU(2m, q) & =
                   \frac{|\mathbf{I}_\UU(V)|}{|\mathcal{C}_\UU(V)|^2}
    = \frac{|\Delta_\UU(V)|}{|\mathcal{C}_\UU(V)|^2}  \quad \quad
    \mbox{(by  Lemma~\ref{lem:delta_right})}
    \\
  & = |\Delta_\UU(V)| \frac{|\gu{m}{q}|^4}{|\gu{2m}{q}|^2} 
= \r_\UU(2m, q) \frac{|\gu{m}{q}|^4}{|\gu{2m}{q}|}  
 \quad \quad \mbox{(by  Definition~\ref{def:ru})}
\\
%&= \r_\UU(2m, q) \frac{\left( q^{m(m-1)/2}\prod_{i = 1}^{m}(q^i -
%    (-1)^i) \right)^4}{q^{m(2m-1)} \prod_{i = 1}^{2m}(q^i - (-1)^i)} & \\
  & = \r_\UU(2m, q) \frac{ \left(q^{m^2} \varphi_{\UU}(m, q)\right)^4}{
    q^{4m^2} \varphi_{\UU}(2m, q)}
 = \r_\UU(2m, q) \frac{\varphi_{\UU}(m, q)^4}{\varphi_{\UU}(2m, q)}\\
&=
                                                                     \r_\UU(2m,
                                                                     q)
                                                                     \Phi_\UU(m,q).
\end{align*}
%  \end{array}
%\]

%
Now let $n = 2m+1$ be odd.
Then
$
  |\mathcal{C}_\UU(V)| = |\gu{2m+1}{q}|/(|\gu{m}{q}|\cdot
  |\gu{m+1}{q}|)$. 
Let $(x, x') \in \mathbf{I}_\UU(V)$, as in \eqref{defIU}, and $y =
xx'$. 
By \cite[Lemma 3.1(b)]{gl}, $\mathrm{gcd}(c_y(X), X^2 - 1) =
X-1$, and $x$ and $x'$ both negate the $1$-dimensional fixed point
space $V_+$ of $y $.
The element $y$, and hence also the pair $(x, x')$, determine a
decomposition of $V$ as in \eqref{eqn:v_decomp}, which we can write as
$ V = V_0 \perp V_{\pm}$, 
where $V_0$ is the sum of the $V_f$ for $f$ of Type~A, B, C and D. We
define $x_0 := x |_{V_0}$ and $y_0 = y |_{V_0}$, and note that
since $V_0$ is non-degenerate, $x_0, y_0 \in \gu{}{V_0}$ and so in
particular $(x_0, y_0) \in \Delta_\UU(V_0)$. 

Conversely, the decomposition $V = V_0 \perp V_{\pm}$, together with
the pair $(x_0, y_0)$, uniquely determines $(x, y)$ and hence also
$(x, x')$, because (i) $x$ negates $V_{\pm}$, so $x = -I_{V_\pm}
\oplus x_0$; and (ii) $y$ fixes $V_{\pm}$, so $y = I_{V_\pm} \oplus
y_0$.

Thus $|\mathbf{I}_{\UU}(V)|$ is equal to $|\Delta_\UU(V_0)|
= |\Delta_{\UU}(2m, q)|$ times the number of
decompositions of $V$ as an orthogonal direct sum of a non-degenerate $2m$-space
and a non-degenerate $1$-space. The orbit-stabiliser theorem then yields
\begin{align*}
\iota_\UU(2m+1, q) & =
                     \frac{|\mathbf{I}_{\UU}(V)|}{|\mathcal{C}(V)|^2}
 = \frac{|\gu{2m+1}{q}|}{|\gu{1}{q}| \cdot
                        |\gu{2m}{q}|} |\Delta_{\UU}(2m, q)| \cdot
    \frac{|\gu{m}{q}|^2\cdot |\gu{m+1}{q}|^2}{|\gu{2m+1}{q}|^2}\\
& = \r_\UU(2m, q) \cdot \frac{|\gu{m}{q}|^2 \cdot
  |\gu{m+1}{q}|^2}{(q+1)|\gu{2m+1}{q}|}\\
&= \r_\UU(2m, q) \cdot \frac{\left(q^{m^2} \varphi_{\UU}(m, q) \right)^2 \cdot
  \left(q^{(m+1)^2} \varphi_{\UU}(m+1, q) \right)^2}{(q+1)q^{(2m+1)^2}
\varphi_{\UU}(2m+1, q)} \\
& = \r_{\UU}(2m, q) \cdot \frac{\varphi_{\UU}(m, q)^2 \varphi_{\UU}(m+1, q)^2}{(1 +
  1/q)\varphi_{\UU}(2m+1, q)}\\
&= \iota_{\UU}(2m, q) \cdot \frac{\varphi_{\UU}(2m, q)}{\varphi_{\UU}(m, q)^4} \cdot
  \frac{\varphi_{\UU}(m, q)^2 \varphi_{\UU}(m+1, q)^2}{(1 + 1/q)\varphi_{\UU}(2m + 1, q)}\\
& = \iota_{\UU}(2m, q) \cdot \frac{(1 - (-1)^{m+1}q^{-m - 1})^2}{(1 +
  q^{-2m -1})(1 + q^{-1})} \qquad\qquad\qquad\qquad\textup{as required.}\qedhere
\end{align*}
%\[ \begin{array}{rl}
%\iota_\UU(2m+1, q) & =
%                     \frac{|\mathbf{I}_{\UU}(V)|}{|\mathcal{C}(V)|^2}
% = \frac{|\gu{2m+1}{q}|}{|\gu{1}{q}| \cdot
%                        |\gu{2m}{q}|} |\Delta_{\UU}(2m, q)| \cdot
%    \frac{|\gu{m}{q}|^2\cdot |\gu{m+1}{q}|^2}{|\gu{2m+1}{q}|^2}\\
%& = \r_\UU(2m, q) \cdot \frac{|\gu{m}{q}|^2 \cdot
%  |\gu{m+1}{q}|^2}{(q+1)|\gu{2m+1}{q}|}
%= \r_\UU(2m, q) \cdot \frac{\left(q^{m^2} \varphi_{\UU}(m, q) \right)^2 \cdot
%  \left(q^{(m+1)^2} \varphi_{\UU}(m+1, q) \right)^2}{(q+1)q^{(2m+1)^2}
%\varphi_{\UU}(2m+1, q)} \\
%& = \r_{\UU}(2m, q) \cdot \frac{\varphi_{\UU}(m, q)^2 \varphi_{\UU}(m+1, q)^2}{(1 +
%  1/q)\varphi_{\UU}(2m+1, q)}
%= \iota_{\UU}(2m, q) \cdot \frac{\varphi_{\UU}(2m, q)}{\varphi_{\UU}(m, q)^4} \cdot
%  \frac{\varphi_{\UU}(m, q)^2 \varphi_{\UU}(m+1, q)^2}{(1 + 1/q)\varphi_{\UU}(2m + 1, q)}\\
%& = \iota_{\UU}(2m, q) \cdot \frac{(1 - (-1)^{m+1}q^{-m - 1})^2}{(1 +
%  q^{-2m -1})(1 + q^{-1})} 
%\end{array}\]
%as required. 
\end{proof}

Theorem~\ref{thm:r_exp} showed that $\r_{\UU}(\infty, q):= \lim_{m \rightarrow \infty}
\r_{\UU}(2m, q)$ exists.
%The following is immediate. 

\begin{cor}\label{cor:iota_lim}
The limits as $m \rightarrow \infty$ of
$\iota_{\UU}(2m, q)$ and $\iota_{\UU}(2m+1, q)$ satisfy
\[
\lim_{m \rightarrow \infty} \iota_{\UU}(2m, q)  = \r_{\UU}(\infty, q)
                                                 \cdot
                                                 \varphi_{\UU}(q)^3,
                                                 \quad 
\lim_{m \rightarrow \infty} \iota_{\UU}(2m + 1, q)  =
                                     \r_{\UU}(\infty,
                                                     q) \cdot
                                     \varphi_{\UU}(q)^3/(1 + q^{-1}). 
\]
\end{cor}

%$(y',x)\mapsto (x,xy')$ defines a bijection between
%the set of such pairs $(y',x)$ and the set of pairs $(x,x')\in \mathbf{I}_\UU(V)$ (defined in
%\eqref{defIU}) such that $xx'$ has  

\medskip

\noindent {\sc Proof of Theorem~\ref{thm:iota_bound}.}
By Lemma~\ref{L:three_bound}, $r_\UU(2m, q) \geq r_\UU(2m, 3)$.
 Hence by Theorem~\ref{T:R_bounds}, $r_\UU(2m, q)
\geq 0.3433$ for $m \geq 2$, whilst $r_\UU(2, q) \geq 0.25$.

We first claim that for $m\geq1$, 
\[ \Phi_\UU(m,q) =  \frac{\varphi(m,
  q)^4}{\varphi(2m, q)} > 1\] 
To see this, first note that \[\Phi_{\UU}(1,q) =  \frac{(1+q^{-1})^4}{(1+q^{-1})(1-q^{-2})} = 
 \frac{(1+q^{-1})^2}{(1-q^{-1})}>1.\] Similarly, for $m=2$, 
 \[
 \Phi_\UU(2,q) =  \frac{(1+q^{-1})^4(1-q^{-2})^4}{(1+q^{-1})(1-q^{-2})(1+q^{-3})(1-q^{-4})}
 = \frac{(1+q^{-1})^2(1-q^{-2})^2}{(1-q^{-1}+q^{-2})(1+q^{-2})}>1 
 \]  
 since $(1+q^{-1})^2>1+q^{-2}$ and $(1-q^{-2})^2> 1-q^{-1}+q^{-2}$ for $q\geq3$.
 So assume that $m\geq3$. If $m$ is odd then
 an easy calculation shows that
 \[
\Phi_\UU(m,q) = \Phi_\UU(m-1,q)\cdot  \frac{(1+q^{-m})^4}{(1+q^{-2m+1})(1-q^{-2m})}>\Phi_\UU(m-1,q).
 \]  
Thus it is sufficient to prove that $\Phi_\UU(m,q)\geq
\Phi_\UU(m-2,q)$ for even $m\geq 4$:  from this we shall conclude that,
for $m\geq2$, $\Phi_\UU(m,q)\geq \Phi_\UU(2,q)>1$.
For $m\geq 4$ even,
\[
\Phi_\UU(m,q) = \Phi_\UU(m-2,q)  
 \cdot \frac{(1 + q^{-m+1})^4  (1 - q^{-m})^4}{(1 + q^{-2m + 3})(1 - q^{-2m+2})(1 + q^{-2m+1 })(1 - q^{-2m})}.
\]
%\begin{equation*}
% \begin{array}{rl}
%\Phi_\UU(m,q) &= \Phi_\UU(m-2,q)  
% \cdot \frac{(1 + q^{-m+1})^4  (1 - q^{-m})^4}{(1 + q^{-2m + 3})(1 - q^{-2m+2})(1 + q^{-2m+1 })(1 - q^{-2m})}.\\
% \end{array}
% \end{equation*}
 The largest of the four terms in the denominator is $1 + q^{-2m + 3}
 < (1 + q^{-m+1}) (1 - q^{-m})$. Thus $\Phi_\UU(m,q) >\Phi_\UU(m-2,q)$, and the claim is proved.
 
Hence by Lemma~\ref{lem:iota_even_odd}
we find that $\iota_{\UU}(2m, q) > r_{\UU}(2m,
q)\geq 0.3433$ for $m \geq 2$, whilst $\iota(2, q) > r_{\UU}(2, q) \geq
0.25$. This concludes the arguments for even dimension.

%Let $\delta(m, q) = \frac{(1 - (-q)^{-m-1})^2}{(1+q^{-2m-1})(1 +  q^{-1})}$.
Let 
\[\delta(m, q) = \frac{(1 - (-q)^{-m-1})^2}{(1+q^{-2m-1})(1 +  q^{-1})}.\]
Then $\delta(1, q) > 1 - q^{-1} - q^{-2}$, and 
$\delta(1, 3) = 4/7$. For $q \geq 5$ our bound shows
that $\delta(1, q) > 19/25 > 4/7$, so $\delta(1, q) >
4/7$ for all $q$. If $m \geq 2$ then performing the division
shows that $\delta(m, q) > 1 - q^{-1} + q^{-2} - q^{-3} \geq 20/27$. 
The result for $\iota(2m+1, q)$ follows from
Lemma~\ref{lem:iota_even_odd} and our bounds for $\iota(2m, q)$. \hfill $\Box$

\end{document}